\documentclass{article}

\usepackage{microtype}
\usepackage{longtable}
\usepackage{graphicx}
\usepackage{subcaption}
\usepackage{booktabs}
\usepackage{array}  
\usepackage{hyperref}

\newcommand{\bm}[1]{\boldsymbol{#1}}

\newcommand{\cX}{{\mathcal{X}}}
\newcommand{\cY}{{\mathcal{Y}}}
\newcommand{\cN}{{\mathcal{N}}}
\newcommand{\cO}{{\mathcal{O}}}
\newcommand{\EE}{{\mathbb{E}}}

\usepackage[preprint]{icml2026}

\usepackage{amsmath}
\usepackage{amssymb}
\usepackage{mathtools}
\usepackage{amsthm}
\usepackage{multirow}

\usepackage[capitalize,noabbrev]{cleveref}

\theoremstyle{plain}
\newtheorem{theorem}{Theorem}[section]
\newtheorem{proposition}[theorem]{Proposition}
\newtheorem{lemma}[theorem]{Lemma}

\theoremstyle{definition}
\newtheorem{definition}[theorem]{Definition}
\newtheorem{assumption}[theorem]{Assumption}
\theoremstyle{remark}
\newtheorem{remark}[theorem]{Remark}

\usepackage[textsize=tiny]{todonotes}

\icmltitlerunning{VR framework for (non)smooth nonconvex--nonconcave stochastic minimax problems with K\L{}}

\begin{document}

\twocolumn[
  \icmltitle{A variance reduced framework for (non)smooth nonconvex–nonconcave stochastic minimax problems with extended Kurdyka–Łojasiewicz property}



  \icmlsetsymbol{equal}{*}

  \begin{icmlauthorlist}
    \icmlauthor{Muhammad Khan}{yyy}
    \icmlauthor{Yangyang Xu}{yyy}
  \end{icmlauthorlist}

  \icmlaffiliation{yyy}{Department of Mathematical Sciences, Rensselaer Polytechnic Institute, Troy, NY, USA}

  \icmlcorrespondingauthor{Muhammad Khan}{khanm7@rpi.edu}
  \icmlcorrespondingauthor{Yangyang Xu}{xuy21@rpi.edu}

  \icmlkeywords{Machine Learning, ICML}

  \vskip 0.3in
]



\printAffiliationsAndNotice{}  

\begin{abstract}
 In this paper, we study stochastic constrained minimax optimization problems with nonconvex--nonconcave structure, a central problem in modern machine learning, for which reliable and efficient algorithms 
 remain largely unexplored due to its inherent challenges.  
 Prior approaches for nonconvex minimax optimization often require (strong) concavity on the maximization part, or certain restrictive geometric assumptions on the joint objective to have guaranteed convergence. In contrast, our method only assumes weak convexity in the primal variable and the extended Kurdyka–Łojasiewicz (K\L{}) property, with exponent $\theta \in [0,1]$, in the dual variable, significantly broadening the class of tractable problems. To this end, we propose a variance reduced algorithm that provably handles this general setting and achieves an $\varepsilon$-stationary solution with state-of-the-art sample complexity: in the smooth finite-sum setting, the sample complexity is 
$\mathcal{O}\left(\sqrt{N}\,\varepsilon^{-\max\{4\theta,2\}}\right)$, where $N$ is the number of total samples,
and in the online smooth setting, it is 
$\mathcal{O}\Big(\varepsilon^{-\max\{6\theta,3\}}\Big)$. 
For the structured nonsmooth problem, the sample complexity is 
$\mathcal{O}\left(\sqrt{N}\,\max\Big\{\varepsilon^{-3}, \varepsilon^{-5\theta}, \varepsilon^{-\frac{11\theta-3}{2\theta}}\Big\}\right)$
and $\mathcal{O}\left(\max\left\{\varepsilon^{-4}, \varepsilon^{-\frac{15\theta-1}{2}}, \varepsilon^{-\frac{31\theta-9}{4\theta}}\right\}\right)$ 
respectively for the two 
settings. To the best of our knowledge, this is the first unified framework that jointly accommodates weak convexity, the extended K\L{} property, and variance-reduced stochastic updates, making it highly suitable for large-scale applications.
\end{abstract}
\section{Introduction}
In recent years, minimax optimization problems have emerged as a central framework in modern Machine Learning, encompassing a wide range of applications such as Generative Adversarial Networks (GANs)~\cite{goodfellow2014gan}, (RL)~\cite{pinto2017robust}, Adversarial Training~\cite{madry2018towards}, Distributionally Robust Optimization (DRO)~\cite{rahimian2019distributionally}, and Optimal Transport (OT)~\cite{kim2025optimal}. Despite their broad applicability, solving minimax problems remains challenging in general, especially when the objective is nonconvex–nonconcave and possibly nonsmooth. In this work, we study a class of constrained minimax optimization problems of the form
\begin{align}
    \min_{x \in \mathcal{X}} \max_{y \in \mathcal{Y}} \;
    F(x,y)
    := {\mathbb{E}_{\bm{\xi} \sim \mathbb{P}}\left[f(x,y;\bm{\xi})\right]},
    \tag{P}\label{eq:primal}
\end{align}
where $ F: \mathbb{R}^{d_x} \times \mathbb{R}^{d_y} \to \mathbb{R} $
is \textit{nonconvex} and \textit{{possibly nonsmooth}} with respect to the primal variable $x$,
and \textit{smooth} but \textit{nonconcave} with respect to the dual variable $y$. Here, $\mathbb{P}$ denotes the probability distribution with support $\Xi$ from which the data points are sampled. In the \emph{online} setting, the distribution $\mathbb{P}$ is unknown, and the algorithm
accesses the objective only through stochastic samples
$\bm{\xi} \sim \mathbb{P}$. When $\mathbb{P}$ is the empirical distribution induced by $N$ i.i.d.\ samples,
problem \eqref{eq:primal} reduces to the \emph{finite-sum} setting, {i.e., 
\begin{equation}\label{eq:finite-sum-obj}
F(x,y) = \frac{1}{N}\sum_{i=1}^N  f(x,y;\bm{\xi}_i).  
\end{equation}}The feasible sets $ \mathcal{X} \subseteq \mathbb{R}^{d_x} $ and $ \mathcal{Y} \subseteq \mathbb{R}^{d_y} $
are assumed to be nonempty, 
{closed}
convex, and
$ f: \mathbb{R}^{d_x} \times \mathbb{R}^{d_y} \times \Xi \to \mathbb{R} $
represents the cost function associated with a random variable $ \bm{\xi} \in \Xi $. In this work, we {first} focus 
on a smooth formulation of problem \eqref{eq:primal}, which
serves as the cornerstone of our algorithmic development and theoretical analysis. To demonstrate the versatility of our framework, we also consider a \textit{nonsmooth} variant where the cost function has a structured form commonly arising in DRO. By applying a suitable smoothing transformation, this variant can be efficiently brought into the smooth regime, allowing our framework to be applied directly while providing rigorous convergence guarantees.
\subsection{Related Works}
\begin{table*}[t]
\caption{{Comparison of {structural} assumptions and sample complexities in finite--sum and/or online regimes for obtaining an $\varepsilon$-stationary solution of minimax optimization problems {considered in several relevant works} in the regime of $N=\cO(\varepsilon^{-4})$. The columns \textbf{Primal} and \textbf{Dual} specify the assumptions on the {objective function about} primal and dual {variables} 
respectively; the column \textbf{Type} indicates whether the method uses stochastic or deterministic gradient. Here, {$\cX$ and $\cY$ are nonempty convex sets.} We denote D: Deterministic, ST: Stochastic, CP: Compact, Cl: Closed, S: Smooth, NS: Nonsmooth, WC: Weakly-Convex, C: Concave, SC: Strongly-Concave, and  e-K\L{}: extended K\L{}. 
$^\dagger$They do not consider the finite sum setting separately, we infer the same sample complexity as the one obtained in the online case. Here, $\tilde\cO$ hides a logarithmic factor of $\frac{1}{\varepsilon}$.} 
$^\ddagger$They consider different settings including a two-sided P\L{} condition and discusses cases when condition number dominates the number of samples. Here for a fair comparison, we assume the underlying condition number is small. $^*$They do not exploit the finite-sum structure, and all its complexity results are inferred from its online setting. $^\star$They consider a more general setting where nonsmooth convex regularizers exist.}
\label{tab:comparison_assumptions}
\centering
\resizebox{\textwidth}{!}{\begin{tabular}{c | c c c | c c c c}
\toprule
\textbf{Source} & \textbf{Type} & $\cX$ & $\cY$ &\textbf{Primal} & \textbf{Dual}& \textbf{Finite-Sum} & \textbf{Online} \\[0.07in]
\midrule
 \multirow{2}{*}{\cite{rafique2018weakly}} & \multirow{2}{*}{ST} & \multirow{2}{*}{$\mathbb{R}^{d_x}$}  &
  \multirow{2}{*}{$\mathbb{R}^{d_y}$} 
  & {S}
  & {S}\&C &{$\tilde{\cO}\left(N\varepsilon^{-2}+\varepsilon^{-6}\right)$}  & N/A  \\[0.07in]
  & & & &  NS\&WC  & C  & $\tilde{\cO}\left(\varepsilon^{-6}\right)^{\dagger}$ & $\tilde{\cO}\left(\varepsilon^{-6}\right)$ \\
  \hline
  \cite{chen2023faster}$^\ddagger$ & ST & $\mathbb{R}^{d_x} $  &
  $\mathbb{R}^{d_y}$ & S  & S\&P\L{} & {$\tilde{\cO}\left( \sqrt{N}\varepsilon^{-2}\right)$} & N/A  \\[0.07in]
   \cite{JiangZhuSoCuiSun2024} & ST & $\mathbb{R}^{d_x} $  &
  $\mathbb{R}^{d_y}$ 
  & S   & S\&SC & $\cO\left(N\varepsilon^{-2}\right)$  &N/A  \\[0.07in]\hline
 \multirow{2}{*}{\cite{zhang2024sapdplus}$^\star$} 
 & {\multirow{2}{*}{ST}} & {\multirow{2}{*}{Cl}}  &
  {\multirow{2}{*}{Cl} }
  & {NS \& WC}
  & {NS}\&{SC} &{$\cO\left(\varepsilon^{-3.5}\right)$}  & {$\cO\left(\varepsilon^{-3.5}\right)$}  \\[0.07in]
  & {ST} 
  & {Cl} 
  & {CP}  &  {S}  & {NS\&C}  & {$\cO\left(\varepsilon^{-6}\right)$} & {$\cO\left(\varepsilon^{-6}\right)$} \\
  \hline
    \cite{ZhengSoLi2025} & D & Cl  & CP & S   & S\&K\L{} &{$ \cO\left(N\varepsilon^{- \max\{4\theta, 2\}}\right)$}  & N/A  \\[0.07in]
     \cite{jiang2025single} & ST & CP  & CP & S   & S\&C &{$\cO\left(\sqrt{N}\varepsilon^{-4}\right)$}  &  N/A  \\[0.07in] \hline
     \multirow{4}{*}{\cite{lin2025two_timescale_ttgda}$^*$} & \multirow{4}{*}{ST} & \multirow{4}{*}{$\mathbb{R}^{d_x}$}  & \multirow{4}{*}{{CP}} & {S}  & {S\&SC}  &$\cO\left(\varepsilon^{-4}\right)$& $\cO\left(\varepsilon^{-4}\right)$ \\ 
     & & & & {S}  & {S\&C}  &$\cO\left(\varepsilon^{-8}\right)$& $\cO\left(\varepsilon^{-8}\right)$ \\ 
     & & & & {NS\&WC}  & {SC}  &$\tilde\cO\left(\varepsilon^{-6}\right)$& $\tilde\cO\left(\varepsilon^{-6}\right)$ \\
     & & & & {NS\&WC}  & {C}  &$\cO\left(\varepsilon^{-8}\right)$& $\cO\left(\varepsilon^{-8}\right)$ \\\hline
     \textbf{This paper} (Sect.~\ref{Section_Smooth}) & ST & {Cl} 
     & CP & S
     & S\&e-K\L{}  & $\cO\left(\sqrt{N}\varepsilon^{-\max\left\{4\theta, 2\right\}}\right)$ &  $\cO\left(\varepsilon^{-\max\left\{6\theta, 3\right\}}\right)$ \\[0.07in]
    \textbf{This paper} (Sect.~\ref{Nonsmooth-section}) & ST & CP  & CP & NS\&WC  & S\&e-K\L{} & 
    $\cO\left(\sqrt{N}\varepsilon^{-\max\{3,5\theta,\frac{11\theta-3}{2\theta}\}}\right)$&  
    $\cO\left(\varepsilon^{-\max\{4,\frac{15\theta-1}{2},\frac{31\theta-9}{4\theta}\}}\right)$\\[0.07in]
\bottomrule
\end{tabular}
}
\end{table*}

Practical challenges in minimax optimization have sparked extensive research across a wide spectrum of problem settings, from convex–concave to fully nonconvex–nonconcave regimes. Early works concentrated on convex–concave problems, leading to the development of optimal primal–dual methods~\cite{ChenLanOuyang2014,HamedaniAybat2021}. Subsequent work addressed nonconvex–concave settings using gradient descent–ascent (GDA) based algorithms and near-optimal methods~\cite{LinJinJordan2020_onGDA,LinJinJordan2020_nearOptimal}, while weakly convex–weakly concave problems were explored in~\cite{LiuRafiqueLinYang2021}. More recent advances have introduced adaptive and universal algorithms under Polyak–Łojasiewicz (P\L{}) and smooth{ness} 
assumptions \cite{YangOrvietoLucchiHe2022, Huang2023}, as well as universal GDA-type algorithms that provide convergence guarantees across multiple problem structures~\cite{ZhengSoLi2023,ZhengSoLi2025}, and rates for nonsmooth nonconvex–nonconcave problems leveraging primal–dual balance and K\L{} properties \cite{li2025nonsmooth}. 

A key limitation of much of the existing literature is the focus on deterministic settings {or unconstrained smooth cases}. While some methods have been extended to stochastic problems, they often require large batch sizes for each update ~\cite{LinJinJordan2020_onGDA, lin2025two_timescale_ttgda} and thus high computational costs, limiting their practicality. This motivates the incorporation of variance reduction techniques, such as SVRG~\cite{johnson2013accelerating} and SPIDER~\cite{fang2018spider}, which have demonstrated 
{effectiveness} in classical minimization tasks, into the more challenging minimax setting. Recent efforts in this direction aim to bridge the gap between theoretical efficiency and practical scalability. 
For instance, 
\citet{chen2023faster} analyzed the SPIDER{-based} algorithm under the P\L{} condition and demonstrated improved convergence rates, while~\citet{JiangZhuSoCuiSun2024} proposed a shuffling gradient descent-ascent method with variance reduction for nonconvex-strongly concave stochastic problems. More recently, 
~\citet{jiang2025single} introduced {single loop algorithms, incorporating variance reduction,} tailored to the finite-sum regime and general nonconvex-concave stochastic minimax problems. 

Despite these advances, a common limitation is the reliance on strong assumptions, such as unconstrained smoothness, and (strong) concavity, or the P\L{} condition on the maximization part, which restricts their applicability to general nonconvex–nonconcave minimax problems. Addressing this limitation remains a major open challenge in the design of scalable and 
{reliable} minimax optimization algorithms.

{To position our work relative to the existing literature, Table~\ref{tab:comparison_assumptions} summarizes relevant studies and their key structural assumptions on the considered problems as well as the sample complexities for obtaining a certain $\varepsilon$-stationary solution.} 

{We emphasize that our results are new. In certain special cases, they are either strictly better than existing ones or match with the best-known ones but under strictly weaker conditions; see more discussions in Remarks~\ref{rm:smooth-comp} and \ref{rm:ns-comp}.}

\subsection{Contributions}
We summarize our key contributions as follows.

 First,  in contrast to prior methods, our approach addresses the fully {nonconvex-nonconcave} minimax problem under significantly \emph{weaker} and more realistic assumptions in both the finite-sum and stochastic online regimes. 
 Specifically, we 
assume the extended {K\L{} property} in $y$ in the sense that the exponent parameter $\theta$ can be any value on $[0,1]$; see Assumption~\ref{assumption_KL_for_smooth_case}. {This property allows for significantly broader applications as compared to the P\L{} condition assumed in \cite{chen2023faster}, which only applies to unconstrained smooth cases with $\theta=\frac{1}{2}$. Though \citet{ZhengSoLi2025} have also studied minimax problems with a classical K\L{} property in $y$, their algorithm relies on deterministic gradients and thus will be inefficient on solving a finite-sum structured or stochastic online minimax problem. In addition, their analysis does not cover the case of $\theta=1$ that is implied by the restricted (a.k.a. star-) concavity (see Remark~\ref{rm:ccv-KL}) and is strictly weaker than the concavity condition assumed in \cite{jiang2025single}.}

Second, designed for both the stochastic \emph{online} and  \emph{finite-sum} settings, our algorithm employs a framework integrated with {SPIDER}-type variance reduction, enabling  
{guaranteed} convergence while significantly reducing computational as well as the sample complexity. In the smooth finite-sum setting, the {sample complexity} of our algorithm  is 
$\mathcal{O}\big(\sqrt{N}\,\varepsilon^{-\max\{4\theta,2\}}\big)$ to produce an $\varepsilon$-stationary solution in expectation. {Compared to the deterministic method in \cite{ZhengSoLi2025} applied to the finite-sum case, our result achieves a factor $\sqrt{N}$ improvement.} 
 In the {online smooth setting}, we obtain a sample complexity of 
$\mathcal{O}\big(\varepsilon^{-\max\{6\theta,3\}}\big)$, which is new. 

{Third, by using the Moreau-envelope smoothing technique, we further extend our method and analysis to a structured nonsmooth nonconvex-nonconcave minimax problem,} for which 
the {sample complexity results of our method} are 
$\mathcal{O}\left(\sqrt{N}\,\max\Big\{\varepsilon^{-3}, \varepsilon^{-5\theta}, \varepsilon^{-\frac{11\theta-3}{2\theta}}\Big\}\right)$ 
and 
$\mathcal{O}\left(\max\left\{\varepsilon^{-4}, \varepsilon^{-\frac{15\theta-1}{2}}, \varepsilon^{-\frac{31\theta-9}{4\theta}}\right\}\right)$ 
respectively for the finite-sum and the online 
settings. Both results are new.


To the best of our knowledge, our method is the {first unified framework} that jointly accommodates weak convexity, the {extended} K\L{} property, and variance-reduced stochastic updates for nonconvex–nonconcave minimax optimization.


 

\subsection{Notations and Definitions} 
Throughout this paper, $\|\cdot\|$ denotes the 2-norm. For an arbitrary set $\mathcal S$, let $\iota_{\mathcal S}$ denote its $0$-$\infty$ indicator function
and $\mathcal N_{\mathcal S}(s)$ for the normal cone to $\mathcal S$ at $s \in \mathcal S$. The distance from a point $a$ to a set $\mathcal S$ is denoted by
$\mathrm{dist}(a,\mathcal S)$, {and {$\mathrm{proj}_\mathcal{S}$} is the projection operator onto {$\mathcal{S}$}.} 
A function $\Psi:\mathcal{X}\to\mathbb{R}$ is called {${\rho}$}-weakly convex for some $\rho\ge0$, if $\Psi(\cdot) + \frac{{\rho}}{2}\|\cdot\|^2$ is convex on $\mathcal{X}$, and 
its {${\delta}$}-subdifferential is defined as
\begin{equation}\label{eq:-delta-subdiff}
\partial^{{\delta}} \Psi(x) =
\left\{
\begin{aligned}
v : \Psi(x') &\ge \Psi(x) + v^\top (x'-x) \\
&\quad \textstyle - \frac{{\rho}}{2}\|x' - x\|^2 - {\delta}
\end{aligned}
\right\}.
\end{equation}
When ${\delta}=0$, the $\delta$-subdifferential reduces to the standard subdifferential.

For {constrained} smooth problems, we pursue an $\varepsilon$-game stationary point defined below.
\begin{definition}
For a given $\varepsilon \geq 0$, a random solution $(x, y) \in \mathcal{X} \times \mathcal{Y}$ is called an $\varepsilon$-Game Stationary (GS) point in expectation of Problem \eqref{eq:primal} if 
\begin{equation*}
\mathbb{E}\left[\mathrm{dist}\left(0, \nabla_x F(x, y) + \mathcal{N}_\mathcal{X}(x)\right)^{2}\right] \leq \varepsilon^{2} \quad \text{and}
\end{equation*} 
\begin{equation*}\mathbb{E}\left[\mathrm{dist}\left(0, -\nabla_y F(x, y) + \mathcal{N}_\mathcal{Y}(y)\right)^{2}\right] \leq \varepsilon^{2}.
\end{equation*} 
\end{definition}
When $F$ is nonsmooth about $x$, we pursue a nearly $\varepsilon$-game stationary point defined as follows.
\begin{definition}
For a given $\varepsilon \geq 0$, a random solution $(x, y) \in \mathcal{X} \times \mathcal{Y}$ is called a \textit{nearly} $\varepsilon$-GS point in expectation of Problem  \eqref{eq:primal} if 
\begin{equation*}
\mathbb{E}\left[\mathrm{dist}\big(0, \partial_{x}^{\varepsilon} (F(x, y) + \iota_\mathcal{X}(x))\big)^{2}\right] \leq \varepsilon^{2} \quad \text{and}
\end{equation*} 
\begin{equation*} \mathbb{E}\left[\mathrm{dist}\left(0, -\nabla_y F(x, y) + \mathcal{N}_\mathcal{Y}(y)\right)^{2}\right] \leq \varepsilon^{2}.
\end{equation*} 
\end{definition}
For convenience of stating our main results, we list several key notations in Tables~\ref{tab:mylabel}. 
\renewcommand{\arraystretch}{1}
\begin{table}[t]
\centering
\caption{A few key notations and definitions}
\label{tab:mylabel}
\begin{tabular}{
>{\centering\arraybackslash}p{2.3cm}
>{\centering\arraybackslash}p{5.3cm}}
\toprule
\textbf{Notation} & \textbf{Definition} \\
\midrule
$\begin{aligned}F_{r}(x,y,z)\end{aligned}$ & $F(x,y) + \frac{r}{2}\left\|x - z\right\|^{2}$ \vspace{0.05in}\\

$\begin{aligned}x_{r}\left(y, z\right)\end{aligned}$ & $\begin{aligned}\arg\min_{x \in \mathcal{X}} \; F_{r}(x, y, z)\end{aligned}$ \vspace{0.05in}\\

$\begin{aligned}d_{r}(y,z)\end{aligned}$ & $\begin{aligned}\min_{x \in \mathcal{X}} \; F_{r}(x,y,z)\end{aligned}$ \vspace{0.05in}\\

$\begin{aligned}p_{r}(z)\end{aligned}$ & $\begin{aligned}\max_{y \in \mathcal{Y}} \min_{x \in \mathcal{X}} \; F_{r}(x, y,z)\end{aligned}$ \vspace{0.05in}\\

$\begin{aligned}F^{\lambda}_{r}(x,y,z)\end{aligned}$ & $F^{\lambda}(x,y) + \frac{r}{2}\left\|x - z\right\|^{2}$ \vspace{0.05in}\\

$\begin{aligned}d_{r}^{\lambda}(y,z)\end{aligned}$ & $ \begin{aligned}\min_{x \in \mathcal{X}} \; F^{\lambda}_{r}(x,y,z)\end{aligned}$ \vspace{0.05in}\\

$\begin{aligned}p_{r}^{\lambda}(z)\end{aligned}$ & $ \begin{aligned}\max_{y \in \mathcal{Y}} \min_{x \in \mathcal{X}} \; F^{\lambda}_{r}(x, y,z)\end{aligned}$  \vspace{0.05in}\\

\bottomrule
\end{tabular}
\end{table}

\section{Constrained Smooth Minimax Problems}\label{Section_Smooth}
We first study the case of smooth minimax problems, which satisfy 
the following assumptions. 
\begin{assumption} \label{assumption_for_smooth_case}
The following statements hold: 
\begin{enumerate}
\item[i.] The set $\mathcal{Y}$ is compact with diameter $D_\mathcal{Y} = \max_{y_1, y_2 \in \mathcal{Y}}\|y_1 - y_2\| < +\infty$.
    \item[ii.] {There exists $\ell>0$ such that 
     $\forall\, (x_1, y_1), (x_2, y_2) \in \mathcal{X} \times \mathcal{Y},$} 
    \begin{equation*}
    \begin{aligned}
    &{\mathbb{E}}\left\|f(x_1, y_1; \bm{\xi}) - f(x_2, y_2; \bm{\xi})\right\| \\
    &\leq \ell \left(\|x_1 - x_2\| + \|y_1 - y_2\|\right). 
    \end{aligned}
    \end{equation*}

    \item [iii. ]
    {There are constants $L_x$ and $L_y$ such that 
    for all $x, x_1, x_2\in \cX$ and $y, y_1, y_2 \in \cY$} 
    \begin{equation*}
    \begin{aligned}
    &\hspace{-2mm}\mathbb{E}\left\|\nabla_{x} f(x_1, y; \bm{\xi}) - \nabla_{x} f(x_2, y; \bm{\xi})\right\|^{2} 
    \leq L_{x}^{2} \|x_1 - x_2\|^{2}, \\
    &\hspace{-2mm}\mathbb{E} \left\|\nabla_{x} f(x, y_{1}; \bm{\xi}) - \nabla_{x} f(x, y_{2}; \bm{\xi})\right\|^{2}
    \leq L_{y}^{2} \|y_1 - y_2\|^{2},\\
    &\hspace{-2mm}{\mathbb{E}}\left\|\nabla_{y} f(x_1, y_1; \bm{\xi}) - \nabla_{y} f(x_2, y_2; \bm{\xi})\right\|^{2} \\
    &\leq L_{y}^{2} \left(\|x_1 - x_2\|^{2} + \|y_1 - y_2\|^{2}\right).
    \end{aligned}
    \end{equation*}

\item[iv. ] {For any $y\in \mathcal{Y}$,} $F(\cdot, y)$ is $\rho$-weakly convex. 
\item[v. ] There exist constants $\sigma_x, \sigma_y\geq 0$ such that {for all} $ x\in \mathcal{X}$ and $ y \in \mathcal{Y}$, 
\begin{align*}
&{\mathbb{E}[\nabla f(x,y;\bm{\xi})\, |\, (x,y)] = \nabla F(x,y)};\\
&\mathbb{E}\left\|\nabla_{x} f(x, y; \bm{\xi}) - \nabla_{x} {F}(x, y)\right\|^{2} \leq \sigma_{x}^2;\\
&\mathbb{E}\left\|\nabla_{y} f(x, y; \bm{\xi}) - \nabla_{y} {F}(x, y)\right\|^{2} \leq \sigma_{y}^2.
\end{align*}
{\item[vi. ] There exists a constant $\underline{F} \in\mathbb{R}$ such that 
\begin{align}\label{eq:low-bound-F}\max_{y \in \mathcal{Y}}\min_{x \in \mathcal{X}}F(x,y) \geq \underline{F}.\end{align}
}
\end{enumerate}    
    \end{assumption} 
    
    
    \begin{assumption}        \label{assumption_KL_for_smooth_case} For all $x \in \mathcal{X}$, 
    $ F(x,\cdot) $ satisfies the 
    {extended K\L{} property:}
    there exist constants $ \mu > 0 $ and $ \theta \in [0,1] $ such that $\forall\, x\in\mathcal{X}, y\in\mathcal{Y}$. 
    \begin{equation}
        \begin{aligned}
    &\mathrm{dist}\left(0, -\nabla_y F(x, y) + \mathcal{N}_\mathcal{Y}(y)\right) \\
    &\geq \mu \left( \max_{y' \in \mathcal{Y}} F(x, y') - F(x, y) \right)^{\theta} \label{KL}.
    \end{aligned}
    \end{equation}
    \end{assumption}

\begin{remark}\label{rm:ccv-KL}
It follows directly from  Assumption~\ref{assumption_for_smooth_case}(ii) that $F(x,y)$ is also $\ell$-Lipschitz continuous. In addition, 
{from Assumption~\ref{assumption_for_smooth_case}(iii) and (v), it follows that for each $x\in\cX$ and $y\in\cY$, $\nabla_x F(\cdot,y)$ is $L_{x}$-Lipschitz continuous, $\nabla_x F(x,\cdot)$ is $L_{y}$-Lipschitz continuous, and  $\nabla_y F(\cdot,\cdot)$ is $L_{y}$-Lipschitz continuous.}

{The classical K\L{} property assumes $\theta \in [0,1)$ and is satisfied by a wide and practically relevant class of functions commonly encountered in ML \cite{Bolte2014PALM}. {An illustrative example showing that nonconcavity does not rule out the satisfaction of the K\L{} property is provided in subsection~\ref{final-section-kl} of the appendix}.

When $\theta=1$, the condition in \eqref{KL} is implied by the restricted (a.k.a. star-) concavity by using the compactness of $\cY$. The proof is given in Section~\ref{sec:proof-rm-ccv-KL} of the appendix.
}
\end{remark}

\subsection{Algorithm}
Our algorithm is motivated by the smoothed proximal linear descent ascent (PLDA) method in \cite{li2025nonsmooth}, which is a deterministic gradient-type method for solving composite nonsmooth minimax problems, as well as the Probabilistic Variance-Reduced Smoothed Gradient Descent-Ascent (PVR-SGDA) in \cite{jiang2025single}. Different from the \emph{deterministic} PLDA 
and PVR-SGDA, 
our method uses \emph{stochastic} gradients together with the SPIDER variance reduction technique that is originally developed for minimization problems in \cite{fang2018spider}. 
The pseudocode of our algorithm is given in Algorithm~\ref{alg:smoothed_primal_dual_spider}, where the stochastic gradient estimators are given as follows: 
\begin{equation}
\begin{aligned}
G_{x,\tau}^{k} 
&= \begin{cases}
  {\frac{1}{B}\sum_{i=1}^{B}\nabla_{x}f\left(x^{k}_{\tau},y^{k}_{\tau};\bm{\xi}^{k}_{\tau,i}\right)}\, ,& \hspace{-0.5cm}\text{ if }\tau = 0, \\
 \frac{1}{M}\sum_{i = 1}^{M} \nabla_{x}f\left(x^{k}_{\tau},y^{k}_{\tau};\bm{\xi}^{k}_{\tau,i}\right)\\
 - \frac{1}{M}\sum_{i=1}^{M} \nabla_{x}f\left(x^{k}_{\tau-1}, y_{\tau-1}^{k};\bm{\xi}^{k}_{\tau,i}\right)\\
 +  G_{x,\tau-1}^{k}\,, & \text{else}.
 \end{cases}
 \label{eq:Gx} 
 \end{aligned}
\end{equation}
\begin{equation}
\begin{aligned}
G_{y,\tau}^{k} 
&= \begin{cases}
  {\frac{1}{B}\sum_{i=1}^{B}\nabla_{y}f\left(x^{k}_{\tau},y^{k}_{\tau};\bm{\xi}^{k}_{\tau,i}\right)}, & \hspace{-0.5cm}\text{ if }\tau = 0, \\
 \frac{1}{M}\sum_{i = 1}^{M} \nabla_{y}f\left(x^{k}_{\tau},y^{k}_{\tau};\bm{\xi}^{k}_{\tau,i}\right)\\
 - \frac{1}{M}\sum_{i=1}^{M} \nabla_{y}f\left(x^{k}_{\tau-1}, y_{\tau-1}^{k};\bm{\xi}^{k}_{\tau,i}\right) \\
 +  G_{y,\tau-1}^{k},   & \text{else}.
 \end{cases}
 \label{eq:Gy}
\end{aligned}
\end{equation}
Here, when $\tau = 0$, we 
draw a mini-batch of $B$ i.i.d. samples $\{\boldsymbol{\xi}_{0,1}^k,\ldots,\boldsymbol{\xi}_{0,B}^k\}$ {from the distribution $\mathbb{P}$} {for the online setting, and take all $N$ samples (thus $B=N$) for the finite-sum setting {in \eqref{eq:finite-sum-obj}};} 
when $0<\tau < T$, we draw a mini-batch of $M$ 
i.i.d. samples $\left\{\bm{\xi}^{k}_{\tau,1}, \ldots , \bm{\xi}^{k}_{\tau,M}\right\}$ {from $\mathbb{P}$}, compute sampled gradients at the current iterate {$(x_\tau^k,y_\tau^k)$} and the previous iterate {$(x_{\tau-1}^k,y_{\tau-1}^k)$}, and then apply the momentum term to update the \emph{stochastic} gradient estimator to reduce variance. Though the algorithm appears to be double-loop, it does not require any subroutine but instead the double iteration indices $k$ and $\tau$ are used for the gradient estimator.
\begin{algorithm}[htbp]
\caption{Smoothed Stochastic Gradient Descent-Ascent Method with SPIDER Variance Reduction}
\label{alg:smoothed_primal_dual_spider}
\begin{algorithmic}[1]
    \STATE \textbf{Input:} Initial point $\left( x^{0}_{0}, y^{0}_{0}, z_{0}^{0}\right)$, 
    {positive integers} $K , T, M, B \geq 1$, {and parameters}  $\alpha_{x},\; \alpha_{y},\; \beta,\; r > 0$. 
\FOR{$k = 0, 1, 2, \ldots, K-1$}
	\FOR{$\tau = 0, 1, 2, \ldots,T - 1 $}
\STATE $x^{k}_{\tau + 1} \gets \text{proj}_\mathcal{X} \left( x^{k}_{\tau} - \alpha_{x} \left[G_{x, \tau}^{k}+r\left(x_{\tau}^{k} - z_{\tau}^{k}\right)\right] \right) $
\STATE $y^{k}_{\tau + 1} \gets \text{proj}_\mathcal{Y} \left( y^{k}_{\tau} + \alpha_{y} G_{y, \tau}^{k}\right)$
\STATE $z^{k}_{\tau + 1} \gets z^{k}_{\tau} + \beta \left(x^{k}_{\tau + 1} - z^{k}_{\tau}\right)$
\ENDFOR
\STATE $x^{k+1}_{0}  \gets x^{k}_{T}$,  $y^{k+1}_{0}  \gets y^{k}_{T}$, $z^{k+1}_{0}  \gets z^{k}_{T}$
\ENDFOR
\STATE \textbf{Output:} Sample $(\tilde{x}, \tilde{y})$ uniformly at random from $\left\{\left(x_{\tau+1}^{k}, y_{\tau+1}^{k}\right)\right\}_{\tau = 0,1, \dots , T-1}^{k = 0,1, \dots , K-1}$
\end{algorithmic}
\end{algorithm}
\subsection{Roadmap of Convergence Analysis}
Here, we 
{sketch the key steps} for establishing the convergence of Algorithm~\ref{alg:smoothed_primal_dual_spider} under Assumptions~\ref{assumption_for_smooth_case} and \ref{assumption_KL_for_smooth_case}. Specifically, we start off by deriving the error bounds for the stochastic gradient estimators employed in the algorithm (see Lemma \ref{lem:err-grad-est}). {These errors carry over the entire analysis, and with appropriate choices of batches size and variance reduction, we are able to control them in a user-specified level.} 
To further understand the behavior of our algorithm, we then 
bound the error between the 
\emph{actual} iterate $ x_{\tau+1}^{k} $ and the \emph{virtual} regularized solution $x_{r}\big(y_{\tau}^k, z_{\tau}^k\big)$ {defined in Table~\ref{tab:mylabel}} (see Lemma \ref{lemma:lemma_primal_error_bound}) . Also, we analyze how changes in the dual variable $ y $ affect the regularized solution, which helps us understand the interaction between the primal and dual updates (see Lemma \ref{dual_error_bound}). These results play a central role in characterizing algorithmic progress. 

{With all these error bounds established,} we then 
{show the one-iteration progress of our algorithm, measured based on} 
a Lyapunov function similar to that used in \cite{li2025nonsmooth}. Specifically, we define $\Phi_{r}: \mathbb{R}^{d_{x}} \times \mathbb{R}^{d_{y}} \times \mathbb{R}^{d_{x}} \to \mathbb{R}$ as follows:
\begin{align}
\Phi_r(x, y, z) := &~ \underbrace{F_r(x, y, z) - d_r(y, z)}_{\text{Primal Descent}} \nonumber\\
\quad &+ \underbrace{p_r(z) - d_r(y, z)}_{\text{Dual Ascent}} + \underbrace{p_r(z)}_{\text{Proximal Descent}}. \label{Lyapunov_Function_Main}
\end{align}
This Lyapunov function is designed to reflect the structure of the algorithm by combining the behavior of the primal, dual, and proximal variable updates. Instead of analyzing each component separately, this unified function captures their interaction in a single expression. In particular, it reflects descent in $x$, ascent in $y$, and the descent in $z$. Analyzing decrease in this function allows us to track convergence in a principled and coherent way, closely aligned with the actual algorithmic updates (see Lemma \ref{lemma:Sufficient_Decrease_Init}).

Finally, we establish auxiliary bounds (see Lemma \ref{lemma:scenario1}) associated with the adopted optimality conditions (see Lemmas~\ref{lemma_sc} and~\ref{dz_for_smooth_case}) and combine all components to derive the iteration and sample complexity guarantees for the smooth case in Theorem \ref{main:theorem2} (presented below). We note that our results encompass both the finite--sum and online settings. 
\subsection{Complexity Results for Smooth Problems}
\begin{theorem}[Iteration and sample complexity for smooth case] 
Under Assumptions~\ref{assumption_for_smooth_case} and \ref{assumption_KL_for_smooth_case}, let $\varepsilon > 0$ be given. Suppose $\rho = \mathcal{O}\big( \min\{L_x, L_y\} \big)$ and $\min\{L_x, L_y\} = \Omega(1)$ and denote ${\Delta\Phi = (\Phi_r(x^{0}_{0}, y^{0}_{0}, z_{0}^{0})-\underline{F})}$, where $\underline{F}$ is the constant in Assumption~\ref{assumption_KL_for_smooth_case}[vi]. 
{For the finite-sum setting, take all $N$ samples to obtain $G_{x,\tau}^k$ and $G_{y,\tau}^k$ when $\tau=0$ in Algorithm \ref{alg:smoothed_primal_dual_spider}, thus $B=N$; for the online setting, let}
$$\begin{aligned}
B =
\begin{cases}
\Theta\bigg(\left(\frac{{L_y^2}(\sigma_x^2+\sigma_y^2)}{(L_y^2+L_x)\varepsilon^2}\right)\max\Big\{L_x + L_y^2,\frac{L_y+\sqrt{L_x}}{\mu^2}  
\Big\}\bigg),\\
  &\hspace{-2.2cm} \text{ if 
}\theta \in \left[0,\frac{1}{2}\right],\\
\Theta\Bigg(\left(\frac{\sigma_x^2+\sigma_y^2}{L_y^2+L_x}\right)\max\Bigg\{
\frac{L_{y}^2\left(L_{y}^2+L_{x}\right)}{\varepsilon^{2}},\\
~\frac{(L_y^2+L_y\sqrt{L_x})^{2\theta}L_y}{\mu^2\varepsilon^{4\theta}},
\frac{(L_x+L_y^2)^\frac{2\theta-1}{2\theta}L_y^\frac{1}{2\theta}(L_y^2+L_y\sqrt{L_x})}{\mu^\frac{1}{\theta}\varepsilon^{\frac{4\theta - 1}{\theta}}}, \\
~\frac{
{L_y^2(L_{y} +\sqrt{L_x})^\frac{3\theta-1}{\theta}}}{\mu^{\frac{1}{\theta}} \varepsilon^{\frac{4\theta-1}{\theta}}}
\Bigg\}\Bigg), &\hspace{-2.2cm} \text{ if 
}\theta \in \left(\frac{1}{2},1\right].
\end{cases}
\end{aligned}$$
In addition, take {$T=M=\left\lceil \sqrt{\frac{B}{2}}\,\right\rceil$}, choose 
$$\begin{aligned}
r = \max\Bigg\{ & 2\rho + 325(L_y+1) + 12\sqrt{L_x}\sqrt{2(L_y+1) },\\[-2mm] 
&\hspace{-0.8cm}2\rho + 54 L_y (L_y + 1) + 4\sqrt{L_x}\sqrt{3 L_y (L_y + 1)}\Bigg\},
\end{aligned}$$
$\alpha_x=\min\Bigg\{
\frac{1}{12(r+L_{x} + 2L_{y})},
\frac{(r-\rho)^{2}}{24(r+L_{x})^2(L_y+1)},  
\frac{r - (\rho + 2 L_y)}{2 L_y (L_x + r)}
\Bigg\}$  and {$$\alpha_y = \min\left\{\alpha_x, \frac{1}{40L_y},\frac{1}{4(2L_y+1)}\right\}.$$} Moreover, let
\begin{align*}
\beta=  \begin{cases}
\min\left\{\frac{1}{30},\; \frac{1}{30r},  \frac{L_{y}}{20r\varpi}\right\}, & \hspace{-0.3cm}\text{ if }\theta\in [0,\frac{1}{2}],\\[4mm]
\Theta\Bigg( \min\Bigg\{
 \frac{1}{r},\; \alpha_x^{\frac{2\theta - 1}{2\theta}} L_y^{-\frac{1}{2\theta}} \mu^\frac{1}{\theta}\varepsilon^{\frac{2\theta - 1}{\theta}}, \;\\
 \hspace{1.0cm}r^{-(2\theta - 1)} \mu^2 L_{y}^{-1} \varepsilon^{4\theta - 2}
,\;\\
\hspace{1.0cm}r^{-\frac{2\theta-1}{\theta}}\;L_{y}^{\frac{\theta-1}{\theta}} \mu^{\frac{1}{\theta}}\varepsilon^{\frac{2\theta-1}{\theta}}\Bigg\}\Bigg),  & \hspace{-0.3cm}\text{ if }\theta \in \left(\frac{1}{2},1\right].
\end{cases}
\end{align*}
Then Algorithm \ref{alg:smoothed_primal_dual_spider} can output an $\varepsilon$-GS point $(\tilde x, \tilde y)$ in expectation of \eqref{eq:primal} and have $\mathbb{E}\left\|\nabla_{z}d_{r}\left(\tilde{y}, \tilde{x}\right)\right\|^{2} =\mathcal{O}\left(\varepsilon^2\right)$ in 
\begin{equation*}
KT=  \begin{cases}
\mathcal{O}\left(\frac{\Delta\Phi\max\left\{L_x + L_y^2, \frac{(L_y+\sqrt{L_x})}{\mu^2} 
\right\}}
{\varepsilon^{2}}\right),
&\hspace{-1.8cm} \text{ if }\theta\in [0,\frac{1}{2}],\\[4mm]
\mathcal{O}\Bigg(\Delta\Phi\max\Bigg\{
\frac{L_{y}^2\left(L_{y}^2+L_{x}\right)}{\varepsilon^{2}}, 
\frac{(L_y^2+L_y\sqrt{L_x})^{2\theta}L_y}{\mu^2\varepsilon^{4\theta}}, \\
\hspace{0.8cm}\frac{(L_x+L_y^2)^\frac{2\theta-1}{2\theta}L_y^\frac{1}{2\theta}(L_y^2+L_y\sqrt{L_x})}{\mu^\frac{1}{\theta}\varepsilon^{\frac{4\theta - 1}{\theta}}},\; \\[1mm]
\hspace{0.8cm}\frac{
{L_y^2(L_{y} +\sqrt{L_x})^\frac{3\theta-1}{\theta}}}{\mu^{\frac{1}{\theta}} \varepsilon^{\frac{4\theta-1}{\theta}}}\Bigg\}\Bigg), & \hspace{-1.8cm}\text{ if }\theta \in \left(\frac{1}{2},1\right],
\end{cases}
\end{equation*}
iterations and by {$\cO(\lceil \sqrt{B} \rceil KT )$} sampled gradients.
\label{main:theorem2}
\end{theorem}
\begin{remark}\label{rm:smooth-comp}
Notice $4\theta \ge \frac{4\theta-1}{\theta}$ for all $\theta>0$. Hence, to only show the dependence on $\varepsilon$, we have the overall iteration complexity $\mathcal{O}\big(\varepsilon^{-\max\{4\theta, 2\}}\big)$ in both the finite-sum and  online settings. 
Consequently, the sample complexity in the finite--sum setting 
is $\mathcal{O}\big(\sqrt{N}\varepsilon^{-\max\{4\theta, 2\}}\big)$ while for the online setting, since $B$ has the same order dependence on $\varepsilon$ as $KT$, the overall sample complexity is $\mathcal{O}\big(\varepsilon^{-\max\{6\theta, 3\}}\big)$. {When $\theta\in[0,\frac{1}{2}]$, our results are optimal, matching the established lower bound of stochastic first-order methods for the special nonconvex minimization problem \cite{arjevani2023lower}.}

{Let us elaborate the superiority of our results over existing ones in a few special cases. We first note that when $\theta = \frac{1}{2}$, our sample complexity is $\cO(\sqrt{N}\varepsilon^{-2})$. It  matches the result in \cite{chen2023faster}, which assumes the P\L{} condition for unconstrained problems. In addition, under the same K\L{} condition, our 
sample complexity 
is lower by a factor of $\sqrt{N}$ 
than that in \cite{ZhengSoLi2025} 
 on solving the finite-sum structured problem. Moreover, when $\theta=1$, our result is $\mathcal{O}\big(\sqrt{N}\varepsilon^{-4}\big)$ for the finite-sum case. 
 It matches that of \cite{jiang2025single}, which assume concavity about $y$, a 
 stronger assumption than ours. 
 In the online setting, when $\theta=1$, our 
 sample complexity 
 is $\mathcal{O}\big(\varepsilon^{-6}\big)$. It matches the result in \cite{rafique2018weakly} but requires a weaker condition, and it improves over the result in  
 \cite{lin2025two_timescale_ttgda} by a factor of $\varepsilon^{-2}$.
 }
\end{remark}
\section{Composite Nonsmooth Minimax Problems}\label{Nonsmooth-section}

In this section, we consider a structured problem setting where the cost function $f$ in \eqref{eq:primal} {has the form of} 
\begin{align}\label{eq:composite-f}
f(x, y; \bm{\xi})
= \varphi\big(h(c(x; \bm{\xi})),\, y;\, \bm{\xi}\big),
\end{align}
with functions
$ \varphi: \mathbb{R}^{d_h} \times \mathbb{R}^{d_y} \times \Xi \to \mathbb{R} $,
$ h: \mathbb{R}^{d_c} \to \mathbb{R}^{d_h} $, and
$ c: \mathbb{R}^{d_x} \times \Xi \to \mathbb{R}^{d_c} $.
Here, $ \varphi $ and $ c $ are assumed to be \textit{smooth}, while $ h $ can be \textit{nonsmooth}.
Consequently, we note that the 
{population function} $F(x,y)$ is possibly \textit{nonsmooth} in $x$ and \textit{smooth} in $y$. 

\subsection{Motivating Applications}
{Several applications in machine learning naturally fall within this framework. 
Notable examples include $\phi$-Divergence distributionally robust optimization (DRO) \cite{levy2020large} as well as Group DRO~\cite{SagawaKohHashimotoLiang2020} which focuses on maintaining strong performance across different subgroups in the data to improve model fairness and robustness. In subsection \ref{application sections} of the appendix, we elaborate on each of these applications and show that, under suitable choices of $\varphi$, $h$, and $c$, the corresponding objective functions can be expressed in the form of \eqref{eq:composite-f}.}

\subsection{Convergence Results}
In this section, we make the following assumptions.
\begin{assumption}\label{assumption:F-nonsmooth}
The sets $\mathcal{X}, \mathcal{Y}$ along with the functions $c$, $h$, and $\varphi$ in \eqref{eq:composite-f} satisfy: 
\begin{itemize}
\item[i.] The set $\mathcal{X}$ is compact with diameter $D_\mathcal{X} = \max_{x_1, x_2 \in \mathcal{X}}\|x_1 - x_2\| < +\infty$.
\item[ii.] The set $\mathcal{Y}$ is compact with diameter $D_\mathcal{Y} = \max_{y_1, y_2 \in \mathcal{Y}}\|y_1 - y_2\| < +\infty$.
    {\item[iii.] 
    {For each {$\bm{\xi} \in \Xi$}, $c(\cdot\,; \bm{\xi})$ 
    } is $\ell_{c}$-Lipschitz continuous.;
    \item[iv.] Each component of $h$ 
    is convex, potentially nonsmooth, and $\ell_h$-Lipschitz continuous; 
    \item[v.] {For each {$\bm{\xi} \in \Xi$}, $\varphi(\cdot, \cdot\,; \bm{\xi})$,} 
    is nondecreasing in its first argument, and $\ell_{\varphi}$-Lipschitz continuous;
    \item[vi.] There exists $L_c$ such that $\forall x_1,x_2 \in \mathbb{R}^{d_x}$
    \begin{align}
    &\hspace{-5mm}\mathbb{E}\left\|\nabla_{x} c(x_1; \bm{\xi}) - \nabla_{x} c(x_2; \bm{\xi})\right\|^{2}
    \leq L_{c}^{2} \|x_1 - x_2\|^{2}. \label{nabla_x_for_c}
    \end{align}
    \item[vii.] {There exists $L_{\varphi}>0$ such that for any $u_1, u_2 \in \mathbb{R}^{d_h}$, $y_1, y_2 \in \mathbb{R}^{d_y}$, and $\bm{\xi} \in \Xi$,}
    \begin{align}
    &~\left\|\nabla_{1} \varphi(u_1, y_{1}; \bm{\xi}) - \nabla_{1} \varphi(u_2, y_{2}; \bm{\xi})\right\|^{2}\nonumber\\ 
    &~\leq L_{\varphi}^{2} \left(\|u_1 - u_2\|^{2} +\|y_1 - y_2\|^{2} \right),\label{nabla_1_phi}\\
    &~ \left\|\nabla_{y} \varphi(u_1, y_{1}; \bm{\xi}) - \nabla_{y} \varphi(u_2, y_{2}; \bm{\xi})\right\|^{2}\nonumber\\ 
    &~\leq L_{\varphi}^{2} \left(\|u_1 - u_2\|^{2} +\|y_1 - y_2\|^{2} \right),\label{nabla_y_phi}
    \end{align}
    where $\nabla_{1}\varphi$ denotes the gradient wrt the first argument. 
    }
\end{itemize}
\end{assumption}

\begin{assumption}\label{assumption:g-kl} 
There exist $\tilde{\delta}>0$, $\mu>0$, and $\theta\in[0,1]$ such that for all $x\in\mathcal{X}$ and all $u: \Xi \to \mathbb{R}^{d_h}$ satisfying $\| u({\bm{\xi}}) - h(c(x;\bm{\xi})) \| \leq \tilde{\delta}, \forall\, \bm{\xi} \in \Xi$, it holds {for \(
    g(u,y) := \mathbb{E}_{\bm{\xi} \sim \mathbb{P} }\left[\varphi(u({\bm{\xi}}), y; \bm{\xi})\right]
    \) that }
\begin{align}
    &\mathrm{dist}\left(0, -\nabla_y g(u, y) + \mathcal{N}_\mathcal{Y}(y)\right)\nonumber\\
    &\geq \mu \left[ \max_{y' \in \mathcal{Y}} g(u, y') - g(u, y) \right]^{\theta} \label{KL-g}, \forall\, y\in\mathcal{Y}.
    \end{align}
    \end{assumption}

\subsubsection{Smoothing by {Moreau Envelope} and Key Properties}
The potential nonsmoothness of $h$ implies that 
$F(x,y)$ is typically nonsmooth in $x$, although it remains smooth in $y$. Hence, Algorithm~\ref{alg:smoothed_primal_dual_spider} cannot be directly applied. To address the challenge caused by the nonsmoothness, we propose to apply the Moreau envelope smoothing technique to each component of $h$, yielding a smooth approximation. The resulting smoothed objective function is given by
\begin{align}\label{eq:smoothed-func}
F^{\lambda}(x, y) 
&:= \mathbb{E}_{\bm{\xi} \sim \mathbb{P}} \left[f^{\lambda}(x, y; \bm{\xi})\right] \nonumber\\
&:= \mathbb{E}_{\bm{\xi} \sim \mathbb{P}} \left[\varphi\big( h^{\lambda}(c(x; \bm{\xi})), y; \bm{\xi} \big)\right] .
\end{align}
{Here,} $h^\lambda: \mathbb{R}^{d_c} \to \mathbb{R}^{d_h}$ is the component-wise Moreau envelope {of $h$}, defined as 
\[
h^\lambda(w) := \Big( h^{\lambda}_{1}(w), \dots, h^{\lambda}_{d_{h}}(w) \Big), \quad \forall w \in \mathbb{R}^{d_c},
\]
where
\[
h^{\lambda}_{j}(w) := \min_{q \in \mathbb{R}^{d_c}} \left\{ h_{j}(q) + \frac{1}{2\lambda} \| w - q \|^2 \right\}, \; j = 1, \dots, d_h.
\]
{\begin{remark}\label{remark-f_lambda_compact}
Since $\cX$ and $\cY$ are both compact, there must exist a constant $\underline{F} \in\mathbb{R}$ such that \eqref{eq:low-bound-F} holds.
Moreover, denote $\underline{F^{\lambda}} = \underline{F} - \frac{\lambda L_{\varphi}L_{h}^{2}\sqrt{d_{h}}}{2} < \infty$. 
\end{remark}}

This smoothing step enables us to apply Algorithm \ref{alg:smoothed_primal_dual_spider} developed for smooth problems while carefully controlling the approximation error introduced by the Moreau envelope. 


%








However, before applying Algorithm~\ref{alg:smoothed_primal_dual_spider} and deriving iteration and sample complexity guarantees in the nonsmooth regime, we 
{need to} establish several key properties of the smoothed objective $F^{\lambda}$ that are required for the algorithm’s convergence analysis.

{Details on showing these properties are given in the appendix. Here we simply state what properties $F^\lambda$ has.} We begin by establishing the Lipschitz continuity of $F^{\lambda}(x,y)$ (see Lemma~\ref{lem:Lip-f-lambda}). We then show that the gradient $\nabla F^{\lambda}$ is Lipschitz continuous (see Lemma~\ref{lemma for smooth f lambda}). 
{Moreover,} we characterize the weak convexity moduli of both the original objective $F$ and its smoothed counterpart $F^{\lambda}$ (see Lemma~\ref{lem:weak-f-lambda}). Finally, we verify that for all $x\in\cX$, 
the function $F^{\lambda}(x,\cdot)$ satisfies 
the {extended} K\L{} property (see Lemma~\ref{lemma:kl-regularized}). 

Having established the necessary properties, we are consequently in a position to apply Algorithm \ref{alg:smoothed_primal_dual_spider} to the smoothed objective $F^{\lambda}(x,y)$. While smoothing allows us to use the algorithm designed in the smooth regime, it introduces an additional approximation error. This error affects the stationarity measures of the original nonsmooth problem, as $F^{\lambda}(x,y)$ only approximates $F(x,y)$, introducing a residual in the stationarity conditions. As a result, the algorithm guarantees convergence only to a \textit{nearly} $\varepsilon$-GS point of the original problem, rather than an exact $\varepsilon$-GS point. In Lemma \ref{lemma:for_application}, we formalize the relationship between the stationarity measures of the original nonsmooth problem and the corresponding measures of the smoothed problem.

\subsubsection{Complexity Results -- Nonsmooth regime}
We now apply Algorithm \ref{alg:smoothed_primal_dual_spider} developed in Section~\ref{Section_Smooth} to the smoothed problem 
$\min_{x\in\mathcal{X}}\max_{y\in\mathcal{Y}} F^{\lambda}(x,y)$. 
As shown in Lemmas~\ref{lem:Lip-f-lambda}, \ref{lemma for smooth f lambda}, and \ref{lem:weak-f-lambda}, the smoothed function $F^{\lambda}(x,y)$ satisfies 
Assumption~\ref{assumption_for_smooth_case}
with smoothness constants and weak convexity constant
\[
\ell = \ell_{\lambda},\; L_x = L_{\lambda,x}, \; L_y = L_{\lambda,y}, \;\rho = \rho_{\lambda}.
\] 
%
Consequently, all relevant constants in Section~\ref{Section_Smooth} that depend on these parameters need to be updated accordingly. 
Additionally, when $\lambda \ll 1$ and $d_h = \Omega(1)$,
we have $\rho_{\lambda} = \mathcal{O}\big( \min\left\{L_{\lambda,x}, L_{\lambda,y} \right\}\big)$ and $\min\left\{L_{\lambda,x}, L_{\lambda,y} \right\}=\Omega(1)$. 
Furthermore, assume there exist constants $\sigma_{\lambda,x},\, \sigma_{\lambda,y}\geq 0$ such that $\forall x \in \mathbb{R}^{d_x},\, y \in \mathbb{R}^{d_y},\, \bm{\xi} \in \Xi$,
\begin{subequations}
\begin{align}
&{\mathbb{E}[\nabla f^{\lambda}(x,y;\bm{\xi})\, |\, (x,y)] = \nabla F^{\lambda}(x,y)};\label{eq:unbias-ns}\\
&\mathbb{E}\left\|\nabla_{x} f^{\lambda}(x, y; \bm{\xi}) - \nabla_{x} {F^{\lambda}}(x, y)\right\|^{2} \leq \sigma_{\lambda, x}^2\;,\;\label{eq:var-x-ns}\\
&\mathbb{E}\left\|\nabla_{y} f^{\lambda}(x, y; \bm{\xi}) - \nabla_{y} {F^{\lambda}}(x, y)\right\|^{2} \leq \sigma_{\lambda, y}^2.\label{eq:var-y-ns}
\end{align}
\end{subequations}
Finally, similar to Section \ref{Section_Smooth}, we define the Lyapunov function as {$\hat{\Phi}_{r}: \mathbb{R}^{d_{x}} \times \mathbb{R}^{d_{y}} \times \mathbb{R}^{d_{x}} \to \mathbb{R}$ as follows:} 
\begin{align*}\hat{\Phi}_{r}(x,y,z)=&~F^{\lambda}_r(x, y, z) - d^{\lambda}_r(y, z)  + p^{\lambda}_r(z)\\
&~- d^{\lambda}_r(y, z) + p^{\lambda}_r(z).\end{align*}
{With all the established properties and updated constants,} we can now directly obtain
the following convergence result for the nonsmooth problem from Theorem~\ref{main:theorem2}.
\begin{theorem}[Iteration and sample complexity for nonsmooth case]
Under Assumptions~\ref{assumption:F-nonsmooth} and \ref{assumption:g-kl} {and conditions in \eqref{eq:unbias-ns}-\eqref{eq:var-y-ns},} let $\varepsilon > 0$ be given. Choose $\lambda=\Theta(\varepsilon)$ and {$T=M=\left\lceil \sqrt{\frac{B}{2}}\,\right\rceil$}. Let $\hat{\Phi}_r^0 = \hat{\Phi}_{r}\left(x^{0}_{0},y^{0}_{0}, z^{0}_{0}\right)$
and denote $\Delta \hat{\Phi} = \hat{\Phi}_{r}^0 - \underline{F}^{\lambda}$.
Apply Algorithm~\ref{alg:smoothed_primal_dual_spider} to the smoothed problem $\min_{x\in\cX}\max_{y\in\cY} F^{\lambda}(x,y)$ with 
{$$\begin{aligned}r = \max\Big\{  &2\rho_{\lambda} + 325(L_{\lambda,y}+1) \\
&+ 12\sqrt{L_{\lambda,x}}\sqrt{2(L_{\lambda,y}+1) }\;, \\
&2\rho + 54 L_{\lambda,y} (L_{\lambda,y} + 1) \\
&+ 4\sqrt{L_{\lambda,x}}\sqrt{3 L_{\lambda,y} (L_{\lambda,y} + 1)}\Big\},
\end{aligned}$$ }
$$\begin{aligned}\alpha_x=\min\Bigg\{&
\frac{1}{12(r+L_{\lambda,x} + 2L_{\lambda,y})},\\
&\hspace{-3mm}\frac{(r-\rho_{\lambda})^{2}}{24(r+L_{\lambda,x})^2(L_{\lambda,y}+1)},\
\frac{r - (\rho_{\lambda} + 2 L_{\lambda,y})}{2 L_{\lambda,y} (L_{\lambda,x} + r)}
\Bigg\}
\end{aligned}$$ and $$\alpha_y = \min\left\{{\alpha_x}, \frac{1}{40L_{\lambda,y}}, \; \frac{1}{4(2L_{\lambda,y}+1)}\right\}.$$ {Additionally, {for the finite-sum setting, take all $N$ samples to obtain the gradient estimators when $\tau=0$, thus $B=N$,} 
while for the online setting, let} 
\begin{equation*}
B =
\begin{cases}
\cO\Bigg(\left(\frac{\sigma_{\lambda,x}^2+\sigma_{\lambda,y}^2}{\left(L_{\lambda,y}^2+L_{\lambda,x}\right)\varepsilon^2}\right)\max\Bigg\{L_{\lambda,x} + L_{\lambda,y}^2, \\
\hspace{3mm}\frac{(L_{\lambda,y}+\sqrt{L_{\lambda,x}})}{\mu^2},
\;L_{\lambda,y}\left(L_{\lambda,y}+\sqrt{L_{\lambda,x}}\right)^2  
\Bigg\}\Bigg),\\
&\hspace{-2.0cm} \text{ if }\theta \in \left[0,\frac{1}{2}\right],\\[2mm]
\cO\Bigg(\left(\frac{\sigma_{\lambda,x}^2+\sigma_{\lambda,y}^2}{L_{\lambda,y}^2+L_{\lambda,x}}\right)\max\Bigg\{
\frac{L_{\lambda,y}^2\left(L_{\lambda,y}^2+L_{\lambda,x}\right)}{\varepsilon^{2}}, \\
\hspace{3mm}\frac{(L_{\lambda,y}^2+L_{\lambda,y}\sqrt{L_{\lambda,x}})^{2\theta}L_{\lambda,y}}{\mu^2\varepsilon^{4\theta}}, \\
\hspace{3mm}\frac{(L_{\lambda,x}+L_{\lambda,y}^2)^\frac{2\theta-1}{2\theta}L_{\lambda,y}^\frac{1}{2\theta}\big(L_{\lambda,y}^2+L_{\lambda,y}\sqrt{L_{\lambda,x}}\big)}{\mu^\frac{1}{\theta}\varepsilon^{\frac{4\theta - 1}{\theta}}},\\
\hspace{3mm}\frac{
{L_{\lambda,y}^2(L_{\lambda,y} +\sqrt{L_{\lambda,x}})^\frac{3\theta-1}{\theta}}}{\mu^{\frac{1}{\theta}} \varepsilon^{\frac{4\theta-1}{\theta}}}\Bigg\}\Bigg),&\hspace{-2.0cm} \text{ if   
}\theta \in \left(\frac{1}{2},1\right].
\end{cases}
\end{equation*}
Moreover, let
\begin{equation*}
\beta=  \begin{cases}
\min\left\{\frac{1}{30},\; \frac{1}{30r},  \frac{L_{\lambda,y}}{20r\varpi_{\lambda}}\right\}, &\hspace{-1.5cm} \text{ if }\theta\in [0,\frac{1}{2}],\\[2mm]
\Theta\Bigg( \min\Bigg\{
 \frac{1}{r},\; \alpha_x^{\frac{2\theta - 1}{2\theta}} L_{\lambda,y}^{-\frac{1}{2\theta}} \mu^\frac{1}{\theta}\varepsilon^{\frac{2\theta - 1}{\theta}}, \\
 \hspace{2mm}r^{-(2\theta - 1)} \mu^2 L_{\lambda, y}^{-1} \varepsilon^{4\theta - 2}
,\\
r^{-\frac{2\theta-1}{\theta}}\;L_{\lambda, y}^{\frac{\theta-1}{\theta}} \mu^{\frac{1}{\theta}}\varepsilon^{\frac{2\theta-1}{\theta}}\Bigg\}\Bigg),
&\hspace{-1.5cm} \text{ if }\theta \in \left(\frac{1}{2},1\right].
\end{cases}
\end{equation*}
Then the algorithm can output a nearly $\varepsilon$-GS point in expectation of \eqref{eq:primal} in
\begin{equation*}
KT=  \begin{cases}
\cO\Bigg(\left(\frac{\Delta \hat{\Phi}  L_{\lambda,y}^2}{\varepsilon^2}\right)\max\Bigg\{L_{\lambda,x} + L_{\lambda,y}^2,\\
\hspace{3mm}\frac{(L_{\lambda,y}+\sqrt{L_{\lambda,x}})}{\mu^2}
\Bigg\}\Bigg),
&\hspace{-1.5cm} \text{ if }\theta\in [0,\frac{1}{2}],\\[2mm]
\cO\Bigg(\Delta \hat{\Phi}\max\Bigg\{
\frac{L_{\lambda,y}^2\left(L_{\lambda,y}^2+L_{\lambda,x}\right)}{\varepsilon^{2}},\\
\hspace{2mm}\frac{(L_{\lambda,y}^2+L_{\lambda,y}\sqrt{L_{\lambda,x}})^{2\theta}L_{\lambda,y}}{\mu^2\varepsilon^{4\theta}}, \\
\hspace{2mm}\frac{(L_{\lambda,x}+L_{\lambda,y}^2)^\frac{2\theta-1}{2\theta}L_{\lambda,y}^\frac{1}{2\theta}(L_{\lambda,y}^2+L_{\lambda,y}\sqrt{L_{\lambda,x}})}{\mu^\frac{1}{\theta}\varepsilon^{\frac{4\theta - 1}{\theta}}}, \\[2mm]
\hspace{2mm}\frac{
{L_{\lambda,y}^2(L_{\lambda,y} +\sqrt{L_{\lambda,x}})^\frac{3\theta-1}{\theta}}}{\mu^{\frac{1}{\theta}} \varepsilon^{\frac{4\theta-1}{\theta}}}\Bigg\}\Bigg),&\hspace{-1.5cm} \text{ if }\theta \in \left(\frac{1}{2},1\right],
\end{cases}
\end{equation*}
iterations and by 
{$\cO(\lceil \sqrt{B} \rceil KT )$} sampled gradients.
\label{last:theorem}
\end{theorem}
\begin{remark}\label{rm:ns-comp} 
Given $\varepsilon \in (0,1)$, note that when $\lambda = \Theta\left(\varepsilon\right)$, we have from from Lemma \ref{lemma for smooth f lambda} that $L_{\lambda,x} = \Theta\left(\lambda^{-1}\right) = \Theta\left(\varepsilon^{-1}\right)$, and $L_{\lambda,y}$ is independent of $\varepsilon$. Thus, 
without specifying the dependence on other quantities, we have the overall iteration complexity $KT=\cO\left(
\max\left\{
\varepsilon^{-3},
\varepsilon^{-5\theta},
\varepsilon^{-\frac{11\theta-3}{2\theta}}
\right\}
\right)$. Consequently, the sample complexity in the finite-sum setting 
is $\cO\left(\sqrt{N}\max\left\{
\varepsilon^{-3},
\varepsilon^{-5\theta},
\varepsilon^{-\frac{11\theta-3}{2\theta}}
\right\}\right)$. For the online setting, since $L_{\lambda,x} = \Theta(\varepsilon^{-1})$ and $L_{\lambda,y}$ is independent of $\varepsilon$, we have $B = \cO\left(\max\{\varepsilon^{-2}, \varepsilon^{-5\theta+1}, \varepsilon^{-\frac{9\theta-3}{2\theta}}\}\right)$, thus the overall sample complexity is $\cO\left(\max\big\{\varepsilon^{-4}, \varepsilon^{-\frac{15\theta-1}{2}}, \varepsilon^{-\frac{31\theta-9}{4\theta}}\big\}\right)$.

{These results are novel. We discuss a few special cases below. In the finite-sum setting, when $N=\Omega(\varepsilon^{-2})$, our result is better than that of the deterministic method in \cite{li2025nonsmooth}, which assumes a stronger proximal-linear-type oracle. When $\theta=1$, our result reduces to $\cO(\sqrt{N}\varepsilon^{-5})$. Recall that in this case, the e-K\L{} condition is implied by the concavity on $y$. Thus our result is better than the $\cO(\varepsilon^{-8})$ complexity in \cite{lin2025two_timescale_ttgda} if $N = \cO(\varepsilon^{-6})$ under a weaker condition. Moreover, in the online setting, when $\theta\in [0,\frac{3}{5}]$, our result becomes $\cO(\varepsilon^{-4})$, which matches the $\cO(\varepsilon^{-4})$ complexity in \cite{lin2025two_timescale_ttgda} for the smooth nonconvex-strongly-concave case and strictly better than $\tilde\cO(\varepsilon^{-6})$ complexity in \cite{lin2025two_timescale_ttgda} for the nonsmooth nonconvex-strongly-concave case. When $\theta=1$, our result reduces to $\mathcal{O}(\varepsilon^{-7})$; it is worse than the $\tilde\cO(\varepsilon^{-6})$ complexity in \cite{rafique2018weakly}, but notice that our assumption of e-K\L{} is weaker than the concavity assumed by the latter.}


\end{remark}
\section{Concluding Remarks}
We have presented a variance-reduced framework for solving fully nonconvex-nonconcave stochastic minimax optimization problems under weak convexity and the {extended} K\L{} property --- assumptions which are substantially weaker than those commonly adopted in literature. By integrating SPIDER-type variance reduction, our method achieves guaranteed convergence with improved iteration and sample complexity in both the \emph{online} and \emph{finite-sum} settings. By Moreau-envelope smoothing, we are able to further extend 
our framework by applying it to a structured nonsmooth problem and establish rigorous complexity results. 

We observe that, in the nonsmooth case, the sample complexity becomes $\cO(\varepsilon^{-7})$ in the extreme case of $\theta=1$, which is implied by a restricted concavity condition in the dual variable. Although this result is new, we conjecture that this complexity can be further reduced, which is worth exploring in future work. 



\bibliography{minimax}
\bibliographystyle{icml2026}

\newpage
\appendix
\onecolumn

{We provide complete proofs of all theorems claimed in the main body of the paper. Beyond the key notations in Table~\ref{tab:mylabel}, we list a few more in Table~\ref{tab:mylabel-key} for convenience of our analysis.}


\renewcommand{\arraystretch}{2.0}
\begin{table}[htpb]
\centering
\caption{More Notations and Definitions}
\label{tab:mylabel-key}
\begin{tabular}{
>{\centering\arraybackslash}p{2.3cm}
>{\centering\arraybackslash}p{5.3cm}}
\toprule
\textbf{Notation} & \textbf{Definition} \\
\midrule

$\begin{aligned}y_{+}(z)\end{aligned}$ & $\begin{aligned}\text{proj}_{\mathcal{Y}}\left(y + \alpha_{y} \nabla_{y}F\left(x_{r}\left(y,z\right),y\right)\right)\end{aligned}$ \\

$\begin{aligned}Y(z)\end{aligned}$ & $\begin{aligned}\arg\max_{y \in \mathcal{Y}} \; d_{r}(y,z)\text{; } y(z) \in Y(z)\end{aligned}$ \\

$\begin{aligned}x_{\tau , +}^{k}\left(y, z\right)\end{aligned}$ & $\begin{aligned}\text{proj}_\mathcal{X} \left( x_{\tau}^{k} - \alpha_{x} \nabla_{x}F_{r}(x_{\tau}^{k},y,z) \right)\end{aligned}$ \\

$\begin{aligned}y_{\tau , +}^{k}(z)\end{aligned}$ & $\begin{aligned}\text{proj}_{\mathcal{Y}}\left(y_{\tau}^{k} + \alpha_{y} \nabla_{y}F\left(x_{r}\left(y_{\tau}^{k},z\right),y_{\tau}^{k}\right)\right)\end{aligned}$ \\

$\begin{aligned}x_{r}^{\lambda}\left(y, z\right)\end{aligned}$ & $\begin{aligned}\arg\min_{x \in \mathcal{X}} \; F^{\lambda}_{r}(x, y, z)\end{aligned}$ \\
\bottomrule
\end{tabular}
\end{table}

\section{Proofs for the smooth case}

We first provide the claim in Remark~\ref{rm:ccv-KL} that our assumed e-K\L{} condition is implied by the star-concavity in the case of $\theta=1$.
\subsection{Proof of the claim in Remark~\ref{rm:ccv-KL}}\label{sec:proof-rm-ccv-KL}To show this claim, for any $y\in \mathcal{Y}$,
%
let us define 
\begin{equation*}
\begin{aligned}
&y^* \in \text{argmax}_{y' \in \mathcal{Y}} F(x, y'),\\ 
&g^* \in \text{argmin}_{g \in -\nabla_y F(x, y) + \mathcal{N}_\mathcal{Y}(y)}\left\|g\right\|.
\end{aligned}
\end{equation*}
 When $F(x,\cdot)$ is {restricted concave on $\cY$, i.e.,
 \begin{equation*}
 F(x,y^*) \le F(x,y) + \langle \nabla_y F(x, y), y^*-y\rangle , \forall\, y\in \cY,
 \end{equation*}}we have from Cauchy-Schwarz inequality that
\begin{equation*}
\begin{aligned}
&\max_{y' \in \mathcal{Y}} F(x, y') = F(x,y^*) \\
&\leq F(x, y) + \left\langle -g^*,\; y^* - y\right\rangle\\
&\leq  F(x, y) + \left\|g^*\right\|\left\| y^* - y\right\|\\
&\leq F(x,y) + D_{\mathcal{Y}}\;\mathrm{dist}\left(0, -\nabla_y F(x, y) + \mathcal{N}_\mathcal{Y}(y)\right).
\end{aligned}
\end{equation*}
Thus $F(x,\cdot)$ satisfies \eqref{KL} with parameters $\mu = \frac{1}{D_{\mathcal{Y}}}$ and $\theta = 1$.

\subsection{Proofs for the smooth case} {To establish our main complexity results in the smooth case, we start from showing a few error bound inequalities.}

\begin{lemma}[Error Bounds for Gradient Estimators]
\label{lem:err-grad-est}
Let $G_{x, \tau}^{k}$ and $G_{y, \tau}^{k}$ be defined in \eqref{eq:Gx} and \eqref{eq:Gy} and {denote
\begin{align*}
C_{\sigma,x} =  C_{\sigma,y} = 0,  & \text{ for finite-sum setting}, \\
C_{\sigma,x} = \frac{\sigma_x^2}{B}, \ C_{\sigma,y} = \frac{\sigma_y^2}{B}, & \text{ for online setting}.
\end{align*}
}
Then under Assumption \ref{assumption_for_smooth_case}, for any $k \geq 0$ and any  $0 \leq \tau \leq T - 1$, 
the following error bounds hold:
\begin{align}
\mathbb{E}\left\|G_{y,\tau}^{k} -\nabla_{y}F\left(x_{\tau}^{k},y_{\tau}^{k}\right)\right\|^{2} \leq  \frac{L_{y}^{2}}{M} \sum_{b=0}^{\tau-1}\mathbb{E}\left\| x^{k}_{b+1} - x^{k}_{b}\right\|^{2} + \frac{L_{y}^{2}}{M} \sum_{b=0}^{\tau-1}\mathbb{E}\left\| y^{k}_{b+1} - y^{k}_{b}\right\|^{2}{+C_{\sigma,y}},
\label{eq:lemma0_y}\\
\mathbb{E}\left\|G_{x, \tau}^{k}  - \nabla_{x}F\left(x_{\tau}^{k},y_{\tau}^{k}\right)\right\|^{2} \leq \frac{2L_{x}^{2}}{M} \sum_{b=0}^{\tau-1}\mathbb{E}\left\| x^{k}_{b+1} - x^{k}_{b}\right\|^{2} + \frac{2L_{y}^{2}}{M} \sum_{b=0}^{\tau-1}\mathbb{E}\left\| y^{k}_{b+1} - y^{k}_{b}\right\|^{2} {+C_{\sigma,x}}.
\label{eq:lemma0_x}
\end{align}
\end{lemma}
\begin{proof} For ease of notation, denote the variance of a random variable $\zeta$ by $\mathrm{Var}(\zeta) = \mathbb{E}\left\|\zeta - \mathbb{E}\left(\zeta\right)\right\|^{2}$. Then following the proof of \citep[Lemma 1]{ZhangXiao1906.10186}, we have 
\begin{align}
&\mathbb{E}\left\|G_{y,\tau}^{k} -\nabla_{y}F\left(x_{\tau}^{k},y_{\tau}^{k}\right)\right\|^{2}\nonumber\\
\leq &~ \mathbb{E}\left\|G_{y,\tau-1}^{k} - \nabla_{y}F\left(x^{k}_{\tau-1},y^{k}_{\tau-1}\right) \right\|^{2} +  \mathrm{Var}\left( \frac{1}{M}\sum_{i=1}^{M}\left(\nabla_{y}f\left(x^{k}_{\tau},y^{k}_{\tau};\bm{\xi}^{k}_{\tau,i}\right) -  \nabla_{y}f\left(x^{k}_{\tau-1}, y_{\tau-1}^{k};\bm{\xi}^{k}_{\tau,i}\right)\right)\right)\nonumber\\
\overset{(i)}{\leq} &~ \mathbb{E}\left\|G_{y,\tau-1}^{k} - \nabla_{y}F\left(x^{k}_{\tau-1},y^{k}_{\tau-1}\right) \right\|^{2} + \frac{1}{M} \mathbb{E}\left\| \nabla_{y}f\left(x^{k}_{\tau},y^{k}_{\tau};\bm{\xi}^{k}_{\tau,i}\right) -  \nabla_{y}f\left(x^{k}_{\tau-1}, y_{\tau-1}^{k};\bm{\xi}^{k}_{\tau,i}\right)\right\|^{2}\nonumber\\
\overset{(ii)}{\leq} &~\mathbb{E}\left\|G_{y,\tau-1}^{k} - \nabla_{y}F\left(x^{k}_{\tau-1},y^{k}_{\tau-1}\right) \right\|^{2} + \frac{L_{y}^{2}}{M} \left(\mathbb{E}\left\| x^{k}_{\tau} - x^{k}_{\tau-1}\right\|^{2} + \mathbb{E}\left\| y^{k}_{\tau} - y^{k}_{\tau-1}\right\|^{2}\right)\nonumber\\
\overset{(iii)}{\leq} & ~\underbrace{\mathbb{E}\left\|G_{y,0}^{k} - \nabla_{y}F\left(x^{k}_{0},y^{k}_{0}\right) \right\|^{2}}_{\textcircled{1}} + \frac{L_{y}^{2}}{M} \sum_{b=1}^{\tau}\left(\mathbb{E}\left\| x^{k}_{b} - x^{k}_{b-1}\right\|^{2} + \mathbb{E}\left\| y^{k}_{b} - y^{k}_{b-1}\right\|^{2}\right)\label{first_eq_for_error_bounds:x}
\end{align}
where $(i)$ follows since the samples are i.i.d and $\mathbb{E}\left\|\zeta - \mathbb{E}\left(\zeta\right)\right\|^{2} \leq \mathbb{E}\left\|\zeta\right\|^{2}$, $(ii)$ follows from the $L_y$-Lipschitz continuity of $\nabla_{y}f(\cdot,\cdot; \bm{\xi})$, and $(iii)$ follows from applying the inequality recursively. 

{
We now bound \textcircled{1} in \eqref{first_eq_for_error_bounds:x}. 
In the finite-sum setting, since {we take all $N$ samples,} 
we have
\[
G_{y,0}^k = \frac{1}{N}\sum_{i=1}^{N} \nabla_y f(x_0^k,y_0^k;\bm{\xi}_{0,i}^{k}) = \nabla_y F(x_0^k,y_0^k),
\]
and thus 
$\textcircled{1} = 0$. 
In the online setting, since the samples are i.i.d., 
Assumption~\ref{assumption_for_smooth_case}$(v.)$ implies
$\textcircled{1} \leq \frac{\sigma_y^2}{B}$. 
Plugging these bounds back into \eqref{first_eq_for_error_bounds:x} yields the desired result in \eqref{eq:lemma0_y}.
}

Similarly, we have 
\begin{align}
&\mathbb{E}\left\|G_{x,\tau}^{k} -\nabla_{x}F\left(x_{\tau}^{k},y_{\tau}^{k}\right)\right\|^{2}\nonumber\\
\leq &~ \mathbb{E}\left\|G_{x,\tau-1}^{k} - \nabla_{x}F\left(x^{k}_{\tau-1},y^{k}_{\tau-1}\right) \right\|^{2} + \frac{1}{M} \mathbb{E}\left\| \nabla_{x}f\left(x^{k}_{\tau},y^{k}_{\tau};\bm{\xi}^{k}_{\tau,i}\right) -  \nabla_{x}f\left(x^{k}_{\tau-1}, y_{\tau-1}^{k};\bm{\xi}^{k}_{\tau,i}\right)\right\|^{2}.\label{first-eq}
\end{align}
Now, by applying Young's inequality and using the fact that 
$\nabla_x F(x,y)$ is $L_{x}$-Lipschitz in $x$ and $L_{y}$-Lipschitz in $y$, 
we can bound the second term in \eqref{first-eq} by
\begin{align*}
&\mathbb{E}\left\| \nabla_{x}f\left(x^{k}_{\tau},y^{k}_{\tau};\bm{\xi}^{k}_{\tau,i}\right) -  \nabla_{x}f\left(x^{k}_{\tau-1}, y_{\tau-1}^{k};\bm{\xi}^{k}_{\tau,i}\right)\right\|^{2}\\
\leq &~ 2\mathbb{E}\left\| \nabla_{x}f\left(x^{k}_{\tau},y^{k}_{\tau};\bm{\xi}^{k}_{\tau,i}\right) -  \nabla_{x}f\left(x^{k}_{\tau-1}, y_{\tau}^{k};\bm{\xi}^{k}_{\tau,i}\right)\right\|^{2} + 2\mathbb{E}\left\| \nabla_{x}f\left(x^{k}_{\tau-1},y^{k}_{\tau};\bm{\xi}^{k}_{\tau,i}\right) -  \nabla_{x}f\left(x^{k}_{\tau-1}, y_{\tau-1}^{k};\bm{\xi}^{k}_{\tau,i}\right)\right\|^{2}\\
\leq &~ 2L_{x}^{2}\left\|x_{\tau}^{k} - x_{\tau-1}^{k}\right\|^{2} + 2L_{y}^{2}\left\|y_{\tau}^{k} - y_{\tau-1}^{k}\right\|^{2}
\end{align*}
Plugging this bound back into \eqref{first-eq} and utilizing a line of reasoning similar to that in the proof for \eqref{eq:lemma0_y}, we retrieve \eqref{eq:lemma0_x}.
\end{proof}

\begin{lemma}\label{lem:lip-cont-x-r} For any $y, y' \in \mathcal{Y}$ and $z, z' \in \mathbb{R}^{d_{x}}$, we have
\begin{align}
&\left\|x_{r}(y,z) - x_{r}(y,z')\right\|^{2} \leq \sigma_{1}^{2}\left\|z - z'\right\|^{2}, 
\label{eq:sc1}\\
&\left\|x_{r}(y,z) - x_{r}(y',z)\right\|^{2} \leq \sigma_{2}^{2}\left\|y - y'\right\|^{2},
\label{eq:sc2}
\end{align}
where $\sigma_{1} = \frac{r}{r-\rho}$ and $\sigma_{2} = 2+\frac{L_{y}}{r-\rho}$.  
\end{lemma}
\begin{proof}
First, we note that the proof of \eqref{eq:sc1} in \cite{li2025nonsmooth} fundamentally relies on the weak convexity of the function in $x$ and the Lipschitz smoothness of the objective in $y$. Since our assumptions on $F(x, y)$ satisfy these same conditions, we can directly adopt their proof without modification. 

We now turn to the proof of \eqref{eq:sc2}. The overall structure of our proof closely follows that of \cite{li2025nonsmooth}, with minor adjustments to accommodate our specific setting. In particular, since $F(\cdot, y)$ is $\rho$-weakly convex in $x$ for any fixed $y$, it follows that
\begin{align}
F_{r}(x_{r}(y', z), y, z) - F_{r}(x_{r}(y, z), y, z) \geq \frac{r-\rho}{2}\left\|x_{r}(y',z) - x_{r}(y, z)\right\|^{2}, \label{lemma3_1}
\\
F_{r}(x_{r}(y, z), y', z) - F_{r}(x_{r}(y', z), y', z) \geq \frac{r-\rho}{2}\left\|x_{r}(y,z) - x_{r}(y', z)\right\|^{2}, \label{lemma3_2}
\end{align}
Additionally, from the $L_{y}$-Lipschitz continuity of $\nabla_{y}F(\cdot,\cdot)$, we have 
\begin{align}
F_{r}(x_{r}(y, z), y', z) - F_{r}(x_{r}(y, z), y, z) \leq \langle \nabla_{y}F(x_{r}(y,z),y,z), y'-y \rangle + \frac{L_{y}}{2}\left\|y - y'\right\|^{2}, \label{lemma3_3}
\\
F_{r}(x_{r}(y', z), y, z) - F_{r}(x_{r}(y', z), y', z) \leq \langle \nabla_{y}F(x_{r}(y',z),y',z), y-y' \rangle + \frac{L_{y}}{2}\left\|y - y'\right\|^{2}, \label{lemma3_4}
\end{align}
Adding \eqref{lemma3_1} and \eqref{lemma3_2}, plugging in the sum of \eqref{lemma3_3} and \eqref{lemma3_4}, and using the Lipschitz continuity of $\nabla_{y}F$ along with the Cauchy–Schwarz inequality, we get
\begin{align}
&(r - \rho)\|x_r(y, z) - x_r(y', z)\|^2 \nonumber\\
\leq &~ \langle \nabla_y F_r(x_r(y, z), y, z) - \nabla_y F_r(x_r(y', z), y', z), y' - y \rangle + L_{y}\|y - y'\|^2 \nonumber\\
\leq &~ \left\| \nabla_y F_r(x_r(y, z), y, z) - \nabla_y F_r(x_r(y', z), y', z)\right\| \left\| y' - y \right\| + L_{y}\|y - y'\|^2\nonumber\\
\leq &~ \Bigg(L_{y}\|x_r(y, z) - x_r(y', z)\| +  L_{y}\left\| y' - y \right\| \Bigg) \left\| y' - y \right\| + L_{y}\|y - y'\|^2    \nonumber\\ 
= &~ L_{y}\|x_r(y, z) - x_r(y', z)\| \left\| y' - y \right\| + 2L_{y}\|y - y'\|^2 \label{lemma3_5}
\end{align}
Letting $\varrho = \frac{\|x_r(y, z) - x_r(y', z)\|}{\left\| y' - y \right\|}$ and rearranging \eqref{lemma3_5}, we get
\begin{align*}
\varrho^2 
\leq &~ \frac{L_{y}}{r - \rho}\varrho + \frac{2L_{y}}{r - \rho} \\
\overset{(i)}{\leq} &~ \frac{1}{2}\varrho^2 + \frac{L_{y}^2}{2(r - \rho)^2} + \frac{2L_{y}}{r - \rho} \\
= &~ \frac{1}{2}\varrho^2 + \frac{L_{y}^2 + 4L_{y}(r - \rho)}{2(r - \rho)^2} \\
\leq &~ \frac{1}{2}\varrho^2 + \frac{(L_{y} + 2(r - \rho))^2}{2(r - \rho)^2},
\end{align*}
where $(i)$ follows from Young's inequality. The desired result follows immediately from the above and this completes the proof.
\end{proof}

The next lemma is taken directly from \cite{li2025nonsmooth}. Though our algorithm is stochastic, the definition of $x_r$ and $d_r$ is based on the deterministic function. Thus this lemma also holds for our case.

\begin{lemma}\label{lem:smooth-dr-y} For any $y, y', y'' \in \mathcal{Y}$ and $z, z', z'' \in \mathbb{R}^{d_{x}}$, we have 
\begin{align}
\left\|\nabla_{y}d_{r}(y',z) - \nabla_{y}d_{r}(y'',z)\right\| \leq L_{d_{r}}\left\|y' - y''\right\|,
\label{eq:sc3}
\end{align}
where $L_{d_{r}} = (\sigma_{2}+1)L_{y}$.
\end{lemma}

\begin{lemma}\label{eq:primal:init}
Under Assumption \ref{assumption_for_smooth_case}, for any $k\geq 0$ and $0 \leq \tau \leq T-1$, we have
\begin{align}
\left\|x_{\tau}^{k} - x_{r}(y_{\tau}^{k},z_{\tau}^{k})\right\| \leq  \eta\left\|x_{\tau}^{k} - x_{\tau, +}^{k}(y_{\tau}^{k},z_{\tau}^{k})\right\|,
\end{align}
where $\eta = \frac{\alpha_{x}L_{x}+\alpha_{x}r+1}{\alpha_{x}r - \alpha_{x}\rho}$.
\end{lemma}

\begin{proof}
The proof follows \cite{Pang1987}, with modifications to fit our setting. Let us first define 
\[u_{\tau}^{k} = x_{\tau}^{k} - x_{\tau, +}^{k}(y_{\tau}^{k},z_{\tau}^{k}) \implies x_{\tau}^{k} - u_{\tau}^{k} = x_{\tau, +}^{k}(y_{\tau}^{k},z_{\tau}^{k}) \in \cX.\] By the definition of $x_{r}(y_{\tau}^{k},z_{\tau}^{k})$, the first-order optimality condition yields
\begin{align}
\left \langle \nabla_{x}F_{r}\left(x_{r}(y_{\tau}^{k},z_{\tau}^{k}),y_{\tau}^{k},z_{\tau}^{k}\right), \left(x_{\tau}^{k} - u_{\tau}^{k}\right) - x_{r}(y_{\tau}^{k},z_{\tau}^{k})  \right\rangle \geq 0 \nonumber\\
\implies \left \langle {\alpha_{x}}\nabla_{x}F_{r}\left(x_{r}(y_{\tau}^{k},z_{\tau}^{k}),y_{\tau}^{k},z_{\tau}^{k}\right), \left(x_{\tau}^{k} - u_{\tau}^{k}\right) - x_{r}(y_{\tau}^{k},z_{\tau}^{k})  \right\rangle \geq 0. \label{eq-supp-2}
\end{align}
Similarly, by the definition of the projected gradient step $x_{\tau, +}^{k}(y_{\tau}^{k},z_{\tau}^{k})$ we obtain
\begin{align}
\left \langle \left(x_{\tau}^{k} - u_{\tau}^{k}\right) -  \left(x_{\tau}^{k} - \alpha_{x} \nabla_{x}F_{r}(x_{\tau}^{k},y_{\tau}^{k},z_{\tau}^{k})\right), x_{r}(y_{\tau}^{k},z_{\tau}^{k}) - \left(x_{\tau}^{k} - u_{\tau}^{k}\right)   \right\rangle \geq 0. \label{eq-supp-3}
\end{align}
Adding \eqref{eq-supp-2} and \eqref{eq-supp-3} leads to
\[\left \langle {\alpha_{x}}\nabla_{x}F_{r}\left(x_{r}(y_{\tau}^{k},z_{\tau}^{k}),y_{\tau}^{k},z_{\tau}^{k}\right) - \alpha_{x} \nabla_{x}F_{r}(x_{\tau}^{k},y_{\tau}^{k},z_{\tau}^{k}) + u_{\tau}^{k}\;, \left(x_{\tau}^{k} - x_{r}(y_{\tau}^{k},z_{\tau}^{k})\right) - u_{\tau}^{k}   \right\rangle \geq 0, \]
which, after rearranging terms, implies
\begin{align}
&\left \langle \alpha_{x} \nabla_{x}F_{r}(x_{\tau}^{k},y_{\tau}^{k},z_{\tau}^{k}) - {\alpha_{x}}\nabla_{x}F_{r}\left(x_{r}(y_{\tau}^{k},z_{\tau}^{k}),y_{\tau}^{k},z_{\tau}^{k}\right)  \;, x_{\tau}^{k} - x_{r}(y_{\tau}^{k},z_{\tau}^{k}) \right\rangle \nonumber\\
\leq &~ \left \langle \alpha_{x} \nabla_{x}F_{r}(x_{\tau}^{k},y_{\tau}^{k},z_{\tau}^{k}) - {\alpha_{x}}\nabla_{x}F_{r}\left(x_{r}(y_{\tau}^{k},z_{\tau}^{k}),y_{\tau}^{k},z_{\tau}^{k}\right)  \;, u_{\tau}^{k}   \right\rangle + \left \langle u_{\tau}^{k}, x_{\tau}^{k} - x_{r}(y_{\tau}^{k},z_{\tau}^{k}) \right\rangle  - \left\|u_{\tau}^{k}\right\|^{2}\nonumber\\
\leq &~ \left\| \alpha_{x} \nabla_{x}F_{r}(x_{\tau}^{k},y_{\tau}^{k},z_{\tau}^{k}) - {\alpha_{x}}\nabla_{x}F_{r}\left(x_{r}(y_{\tau}^{k},z_{\tau}^{k}),y_{\tau}^{k},z_{\tau}^{k}\right) \right\|\left\|u_{\tau}^{k}\right\| + \left\|x_{\tau}^{k} - x_{r}(y_{\tau}^{k},z_{\tau}^{k})\right\|\left\|u_{\tau}^{k}\right\|\nonumber\\
\overset{(i)}{\leq} &~ \left({\alpha_{x}}\left(r+L_{x}\right)+1\right)\left\|x_{\tau}^{k} - x_{r}(y_{\tau}^{k},z_{\tau}^{k})\right\|\left\|x_{\tau}^{k} - x_{\tau, +}^{k}(y_{\tau}^{k},z_{\tau}^{k})\right\|, \label{eq-supp-4}
\end{align}
where $(i)$ follows from the $(r+L_{x})$-Lipschitz smoothness of $\nabla_{x}F_{r}(\cdot,y,z)$ and the definition of $u_{\tau}^{k}$.

Now, since $F_{r}(\cdot,y,z)$ is $(r-\rho)$-strongly convex, we have
\begin{align}
 \left\langle  {\alpha_{x}}\nabla_{x}F_{r}\left(x_{\tau}^{k},y_{\tau}^{k},z_{\tau}^{k}\right) - {\alpha_{x}}\nabla_{x}F_{r}\left(x_{r}(y_{\tau}^{k},z_{\tau}^{k}),y_{\tau}^{k},z_{\tau}^{k}\right), x_{\tau}^{k} - x_{r}(y_{\tau}^{k},z_{\tau}^{k})  \right \rangle \geq {\alpha_{x}}(r - \rho)\left\|x_{\tau}^{k} - x_{r}(y_{\tau}^{k},z_{\tau}^{k})\right\|^{2}  \label{eq-supp-1}
\end{align}
Finally, combining \eqref{eq-supp-1} and \eqref{eq-supp-4} yields the desired inequality, completing the proof.
%
\end{proof}

{We are now ready to establish the primal error bound.}

\begin{lemma}[Primal Error Bound] Under Assumption \ref{assumption_for_smooth_case}, for any $k\geq 0$ and $0 \leq \tau \leq T-1$, we have

\begin{align*} 
\mathbb{E}\left\|x^{k}_{\tau + 1} - x_{r}\left(y_{\tau}^{k}, z_{\tau}^{k}\right)\right\|^{2} \leq &~ \left(\frac{5\eta^{2}}{2}+5\right)\mathbb{E}\left\|x^{k}_{\tau+1} - x^{k}_{\tau}  \right\|^{2} + \frac{5\eta^{2}\alpha_{x}^{2}L_{x}^{2}}{M}\sum_{b=0}^{\tau-1}\mathbb{E}\left\| x^{k}_{b+1} - x^{k}_{b}\right\|^{2}\\
\quad & +\frac{5\eta^{2}\alpha_{x}^{2}L_{y}^{2}}{M}\sum_{b=0}^{\tau-1}\mathbb{E}\left\| y^{k}_{b+1} - y^{k}_{b}\right\|^{2}{+\frac{5\alpha_x^2\eta^2C_{\sigma,x}}{2}}.
\end{align*}
\label{lemma:lemma_primal_error_bound}
\end{lemma}
\begin{proof} Utilizing Young's inequality, we have
\begin{align}
\mathbb{E}\left\|x^{k}_{\tau + 1} - x_{r}\left(y_{\tau}^{k}, z_{\tau}^{k}\right)\right\|^{2} 
\leq &~ 5\mathbb{E}\left\|x^{k}_{\tau + 1} - x^{k}_{\tau}\right\|^{2} + \frac{5}{4}\mathbb{E}\left\|x^{k}_{\tau} - x_{r}\left(y_{\tau}^{k}, z_{\tau}^{k}\right)\right\|^{2}\nonumber\\
\overset{(i)}
{\leq} &~ 5\mathbb{E}\left\|x^{k}_{\tau + 1} - x^{k}_{\tau}\right\|^{2} + \frac{5\eta^{2}}{4}\mathbb{E}\left\|x^{k}_{\tau}  - x_{\tau , +}^{k}\left(y_{\tau}^{k}, z_{\tau}^{k}\right) \right\|^{2}\nonumber\\
\overset{(i)}{\leq}&~\left(\frac{5\eta^{2}}{2}+5\right)\mathbb{E}\left\|x^{k}_{\tau+1} - x^{k}_{\tau} \right\|^{2} + \frac{5\eta^{2}}{2}\mathbb{E}\left\| x^{k}_{\tau+1} - x_{\tau , +}^{k}\left(y_{\tau}^{k}, z_{\tau}^{k}\right) \right\|^{2},
\label{eq:lemma1_init_1}
\end{align}
where $(i)$ follows from 
Lemma \ref{eq:primal:init}. 
We now bound the second term in \eqref{eq:lemma1_init_1} by
\begin{align}
\mathbb{E}\left\|x^{k}_{\tau + 1} - x_{\tau , +}^{k}\left(y_{\tau}^{k}, z_{\tau}^{k}\right)  \right\|^{2} 
= &~ \mathbb{E}\left\| \text{proj}_\mathcal{X} \left( x^{k}_{\tau} - \alpha_{x} \left[G_{x,\tau}^{k} + r\left(x^{k}_{\tau} - z^{k}_{\tau}\right)\right] \right) - \text{proj}_\mathcal{X} \left( x^{k}_{\tau} - \alpha_{x} \nabla_{x}F_{r}(x^{k}_{\tau},y_{\tau}^{k},z_{\tau}^{k}) \right)  \right\|^{2} \nonumber \nonumber \\
\overset{(i)}{\leq} &~ \alpha_{x}^{2}\mathbb{E}\left\| \nabla_{x}F(x^{k}_{\tau},y_{\tau}^{k}) - G_{x,\tau}^{k} \right\|^{2} \nonumber \\
\overset{(ii)}{\leq} &~ \frac{2\alpha_{x}^{2}L_{x}^{2}}{M}\sum_{b=0}^{\tau-1}\mathbb{E}\left\| x^{k}_{b+1} - x^{k}_{b}\right\|^{2} + \frac{2\alpha_{x}^{2}L_{y}^{2}}{M}\sum_{b=0}^{\tau-1}\mathbb{E}\left\| y^{k}_{b+1} - y^{k}_{b}\right\|^{2}{+\alpha_x^2C_{\sigma,x}}\,,
\label{eq:lastline}
\end{align}
where $(i)$ follows from non-expansiveness of the projection operator, and $(ii)$ holds by  \eqref{eq:lemma0_x}. Consequently, plugging \eqref{eq:lastline} into \eqref{eq:lemma1_init_1}, we get the desired inequality.
\end{proof}
The next lemma is adapted directly from \cite{li2025nonsmooth}. The only modification is that we assume $\rho$-weak convexity, and apart from this, all other arguments in its proof remain unchanged.
\begin{lemma} Under Assumption \ref{assumption_for_smooth_case}, for any $y \in \mathcal{Y}$ and $z \in \mathbb{R}^{d_{x}}$, we have 
\begin{align}
\frac{r - \rho}{2} \left\| x_r(y(z), z) - x_r(y_+(z), z) \right\|^2 
\leq \max_{y' \in \mathcal{Y}} F\left(x_r(y_+(z), z), y'\right) - F\left(x_r(y_+(z), z), y_+(z)\right). \label{dual_bound_intermediate}
\end{align}
\label{dual_lemma_init}
\end{lemma}
We are now in a position to formally establish the dual error bound. Its proof adapts from that in \cite{li2025nonsmooth} but tightens the bound for the case of $\theta\in \left(0,\frac{1}{2}\right]$.
\begin{lemma}[Dual Error Bound]\label{dual_error_bound} Under Assumptions \ref{assumption_for_smooth_case} and \ref{assumption_KL_for_smooth_case}, for any $y \in \mathcal{Y}$ and $z \in \mathbb{R}^{d_{x}}$, we have 
\begin{align}
&\left\|x_{r}(y(z),z) - x_{r}(y_{+}(z),z)\right\|^{2} \leq 
\varpi\left\|y - y_{+}(z)\right\|^{2},
\text{ $ for \ \textstyle \theta \in \left[0,\; \frac{1}{2}\right]$},
\label{eq:dual_theta_half}\\
&\left\|x_{r}(y(z),z) - x_{r}(y_{+}(z),z)\right\|^{2} \leq 
\kappa \left\|y - y_{+}(z)\right\|^{\frac{1}{\theta}},
\text{ $ for \  \textstyle \theta \in \left(\frac{1}{2},\; 1\right]$},
\label{eq:dual_theta_1}
\end{align}
where 
$\varpi = \frac{2\left(\ell D_\cY\right)^{1 - 2\theta}}{r-\rho}\left(\frac{\frac{2}{\alpha_{y}^{2}} + 2L_{y}^{2}\sigma_{2}^{2} + 2L_{y}^{2} }{\mu^{2}}\right)$, and $\kappa = \frac{2}{r-\rho}\left(\frac{\sqrt{\frac{2}{\alpha_{y}^{2}} + 2L_{y}^{2}\sigma_{2}^{2}+ 2L_{y}^{2}} }{\mu}\right)^{\frac{1}{\theta}}
$, with $\sigma_{2} = 2+\frac{L_{y}}{r-\rho}$ the same as that in Lemma~\ref{lem:lip-cont-x-r}.

\end{lemma}
\begin{proof} 
We first discuss the case when $\theta = 0$. We note that if $\max_{y' \in \mathcal{Y}} F(x_r(y_+(z),z), y') - F(x_r(y_+(z),z), y_+(z)) = 0$, then from \eqref{dual_bound_intermediate}, it follows trivially that 
\[\frac{r - \rho}{2} \left\| x_r(y(z), z) - x_r(y_+(z), z) \right\|^2 
\leq 0 \leq \frac{\ell \cdot D_\cY\left(\frac{2}{\alpha_{y}^{2}} + 2L_{y}^{2}\sigma_{2}^{2}+ 2L_{y}^{2}\right)}{\mu^{2}} \left\|y - y_{+}(z)\right\|^{2},\] which renders the desired inequality in \eqref{eq:dual_theta_half} for $\theta=0$.  Hence, we suppose $\max_{y' \in \mathcal{Y}} F(x_r(y_+(z),z), y') - F(x_r(y_+(z),z), y_+(z)) \neq 0$. Then from \eqref{KL}, it holds
\begin{align}
&\mathrm{dist}\left(0, -\nabla_y F(x_r(y_+(z),z), y_+(z)) + \cN_\cY(y_+(z))\right) \geq \mu. 
\label{dual_for_theta_0}
\end{align}
By the definition of $y_{+}(z)$ in Table~\ref{tab:mylabel}, 
we have
\begin{align*}
0 &\in \left(y_{+}(z)- y\right) - \alpha_{y} \nabla_{y}F\left(x_{r}(y,z), y\right) + \cN_\cY\left(y_{+}(z)\right)\\
\implies & \frac{1}{\alpha_{y}}\left(y - y_{+}(z)\right) + \nabla_{y}F\left(x_{r}(y,z), y\right) - \nabla_{y}F\left(x_{r}(y_+(z),z), y_{+}(z)\right) \in - \nabla_{y}F\left(x_{r}(y_+(z),z), y_{+}(z)\right) + \cN_\cY\left(y_{+}(z)\right)
\end{align*}
Hence, by Young's inequality, the $L_{y}$-Lipschitz continuity of $\nabla_{y}F$ and \eqref{eq:sc2}, it follows
\begin{align}
&\Bigg(\mathrm{dist}\left(0, -\nabla_y F(x_r(y_+(z),z), y_+(z)) + \cN_\cY(y_+(z))\right)\Bigg)^{2}\nonumber\\
\leq &~ \left\|\frac{1}{\alpha_{y}}\left(y - y_{+}(z)\right) + \nabla_{y}F\left(x_{r}(y,z), y\right) - \nabla_{y}F\left(x_{r}(y_+(z),z), y_{+}(z)\right)  \right\|^{2}\nonumber\\
{\leq} &~\frac{2}{\alpha_{y}^{2}}\left\|y - y_{+}(z)\right\|^{2} + 2\left\|\nabla_{y}F\left(x_{r}(y,z), y\right) - \nabla_{y}F\left(x_{r}(y_+(z),z), y_{+}(z)\right)  \right\|^{2}\nonumber\\
{\leq} &~ \left(\frac{2}{\alpha_{y}^{2}} + 2L_{y}^{2}\sigma_{2}^{2}+ 2L_{y}^{2}\right)\left\|y - y_{+}(z)\right\|^{2} \label{dist-inequal-for-dual}, 
\end{align}
which together with \eqref{dual_for_theta_0} gives
\begin{equation}\label{dual_for_theta_0_3}
\frac{1}{\mu^2} \left(\frac{2}{\alpha_{y}^{2}} + 2L_{y}^{2}\sigma_{2}^{2}+ 2L_{y}^{2}\right)\left\|y - y_{+}(z)\right\|^{2}  \ge 1. 
\end{equation}
In addition, by the $\ell$-Lipschitz continuity of $F$ from Assumption~\ref{assumption_for_smooth_case}, it holds that 
\begin{align}
&\max_{y' \in \mathcal{Y}} F\left(x_r(y_+(z), z), y'\right) - F\left(x_r(y_+(z), z), y_+(z)\right) \leq \ell \cdot D_\cY. 
\label{dual_for_theta_0_second}
\end{align}
Now by \eqref{dual_bound_intermediate}, \eqref{dual_for_theta_0_3}, and \eqref{dual_for_theta_0_second}, we get \eqref{eq:dual_theta_half} for $\theta=0$.

Second, we discuss the case when $\theta \in \left(0,\frac{1}{2}\right]$. 
We have
\begin{align}
&\mu \left( 
\max_{y' \in \mathcal{Y}} F\left(x_r(y_{+}(z), z), y'\right) - F\left(x_r(y_{+}(z), z), y_{+}(z)\right) 
\right)^{\frac{1}{2}} \nonumber\\
= &~ \mu \left( 
\max_{y' \in \mathcal{Y}} F\left(x_r(y_{+}(z), z), y'\right) - F\left(x_r(y_{+}(z), z), y_{+}(z)\right) 
\right)^{\theta} \nonumber\\
\quad & \left( 
\max_{y' \in \mathcal{Y}} F\left(x_r(y_{+}(z), z), y'\right) - F\left(x_r(y_{+}(z), z), y_{+}(z)\right) 
\right)^{\frac{1}{2} - \theta} \nonumber\\
{\leq} &~\left(\ell D_\cY\right)^{\frac{1}{2} - \theta}\left(\sqrt{\frac{2}{\alpha_{y}^{2}} + 2L_{y}^{2}\sigma_{2}^{2}+ 2L_{y}^{2}} \right)\left\|y - y_{+}(z)\right\|,  \nonumber
\end{align}
where the inequality follows from \eqref{KL}, \eqref{dist-inequal-for-dual}, and \eqref{dual_for_theta_0_second}. 
The above inequality together with \eqref{dual_bound_intermediate} implies \eqref{eq:dual_theta_half}. 

Third, we discuss the case when $\theta \in \left(\frac{1}{2},1\right]$. 
We have from \eqref{KL} and \eqref{dist-inequal-for-dual} that
\begin{align}
&\mu \left( 
\max_{y' \in \mathcal{Y}} F\left(x_r(y_{+}(z), z), y'\right) - F\left(x_r(y_{+}(z), z), y_{+}(z)\right) 
\right)^{\theta} \nonumber\\
\leq &~ \left(\sqrt{\frac{2}{\alpha_{y}^{2}} + 2L_{y}^{2}\sigma_{2}^{2}+ 2L_{y}^{2}} \right)\left\|y - y_{+}(z)\right\| \nonumber
\end{align}
which together with \eqref{dual_bound_intermediate} gives \eqref{eq:dual_theta_1}. This completes the proof. 
\end{proof}

{In preparation for establishing a sufficient decrease condition for the Lyapunov function defined in \eqref{Lyapunov_Function_Main}, we first present several auxiliary lemmas that will be vital in the subsequent analysis.}

\begin{lemma}
Under Assumption \ref{assumption_for_smooth_case}, let $\{x_\tau^k, y_\tau^k,z_\tau^k \}$ be generated from Algorithm \ref{alg:smoothed_primal_dual_spider}. Then it holds for any $k\ge0$ and any $0\le \tau \le T-1$ that
\begin{align*}
&\mathbb{E}\left[F_r(x^{k}_{\tau}, y^{k}_{\tau}, z^{k}_{\tau}) -  F_r(x^{k}_{\tau + 1}, y^{k}_{\tau + 1}, z^{k}_{\tau + 1}) \right]\\
\geq &~  \left(\frac{1}{2\alpha_{x}} - \frac{r+L_{x}}{2}\right)\mathbb{E}\left\|x^{k}_{\tau+1} - x^{k}_{\tau}\right\|^{2} - \frac{\alpha_{x}L_{x}^{2}}{M}\sum_{b=0}^{\tau-1}\mathbb{E}\left\| x^{k}_{b+1} - x^{k}_{b}\right\|^{2} - \frac{\alpha_{x}L_{y}^{2}}{M}\sum_{b=0}^{\tau-1}\mathbb{E}\left\| y^{k}_{b+1} - y^{k}_{b}\right\|^{2}    \\
& +  \mathbb{E}\left\langle \nabla_y F_r\left(x^{k}_{\tau + 1}, y^{k}_{\tau}, z^{k}_{\tau}\right), y^{k}_{\tau} - y^{k}_{\tau + 1} \right\rangle - \frac{L_{y}}{2} \mathbb{E}\left\|y^{k}_{\tau + 1} - y^{k}_{\tau} \right\|^2 + \frac{(2 - \beta)r\beta}{2} \mathbb{E}\left\|x^{k}_{\tau+1} - z^{k}_{\tau}\right\|^2{-\frac{\alpha_xC_{\sigma,x}}{2}}. 
\end{align*}
\label{lemma:sufficient_decrease_primal}
\end{lemma}
\begin{proof} Our proof technique 
is similar to that in \cite{jiang2025single}, upon which we build to accommodate our setting and extend the result. We split the change of function to three parts as follows
\begin{align}
&\mathbb{E}\left[F_r(x^{k}_{\tau}, y^{k}_{\tau}, z^{k}_{\tau}) -  F_r(x^{k}_{\tau + 1}, y^{k}_{\tau + 1}, z^{k}_{\tau + 1}) \right] \nonumber\\
= &~ \underbrace{\mathbb{E}\left[F_r(x^{k}_{\tau}, y^{k}_{\tau}, z^{k}_{\tau}) - F_r(x^{k}_{\tau+1}, y^{k}_{\tau}, z^{k}_{\tau})\right]}_{\textcircled{1}} + \underbrace{\mathbb{E}\left[F_r(x^{k}_{\tau+1}, y^{k}_{\tau}, z^{k}_{\tau}) - F_r(x^{k}_{\tau+1}, y^{k}_{\tau+1}, z^{k}_{\tau})\right]}_{\textcircled{2}} \nonumber\\
\quad & + \underbrace{\mathbb{E}\left[F_r(x^{k}_{\tau+1}, y^{k}_{\tau+1}, z^{k}_{\tau}) -  F_r(x^{k}_{\tau + 1}, y^{k}_{\tau + 1}, z^{k}_{\tau + 1}) \right]}_{\textcircled{3}}
\label{sufficient_decrease_x_only}
\end{align}
We first bound \textcircled{1} in \eqref{sufficient_decrease_x_only}. From the update for $x_{\tau+1}^k$ in Algorithm \ref{alg:smoothed_primal_dual_spider}, we have that for any $x \in \mathcal{X}$, it holds
\begin{align*}
\left\langle x_{\tau+1}^{k} - \left(x^{k}_{\tau} - \alpha_{x} \left[G_{x,\tau}^{k} + r\left(x^{k}_{\tau} - z^{k}_{\tau}\right)\right]\right), x - x_{\tau+1}^{k}\right\rangle \geq 0.  
\end{align*}
Letting $x=x_\tau^k$ in the above inequality and rearranging terms yields
\begin{align}
\left\langle G_{x,\tau}^{k} + r\left(x^{k}_{\tau} - z^{k}_{\tau}\right), x_{\tau+1}^{k} - x_{\tau}^{k}\right\rangle \leq -\frac{1}{\alpha_{x}}\left\|x_{\tau+1}^{k} - x_{\tau}^{k}\right\|^{2}.
\label{eq:primal_descent_init}
\end{align}
In addition, by the $(r+L_{x})$-smoothness of $F_{r}(.,y,z)$ , we have 
\begin{equation}\label{eq:ineq-smooth-Fr}
F_r\left(x^{k}_{\tau+1}, y^{k}_{\tau}, z^{k}_{\tau}\right) 
\leq F_r\left(x^{k}_{\tau}, y^{k}_{\tau}, z^{k}_{\tau}\right) + \left\langle \nabla_{x}F_{r}\left(x_{\tau}^{k}, y^{k}_{\tau}, z^{k}_{\tau}\right) , x^{k}_{\tau+1} - x^{k}_{\tau}\right\rangle + \frac{r+L_{x}}{2}\left\|x^{k}_{\tau+1} - x^{k}_{\tau}\right\|^{2}.   
\end{equation}
For the inner product term in \eqref{eq:ineq-smooth-Fr}, we split it and then bound each term by
\begin{align}\label{eq:ineq-smooth-Fr-prod}
&\left\langle \nabla_{x}F_{r}\left(x_{\tau}^{k}, y^{k}_{\tau}, z^{k}_{\tau}\right) , x^{k}_{\tau+1} - x^{k}_{\tau}\right\rangle \nonumber\\
= &~ \left\langle \nabla_{x}F\left(x_{\tau}^{k}, y^{k}_{\tau}\right) - G_{x,\tau}^{k}  , x^{k}_{\tau+1} - x^{k}_{\tau}\right\rangle + \left\langle G_{x,\tau}^{k} + r\left(x^{k}_{\tau} - z^{k}_{\tau}\right) , x^{k}_{\tau+1} - x^{k}_{\tau}\right\rangle \nonumber \\
\leq &~ \frac{\alpha_x}{2} \|\nabla_{x}F\left(x_{\tau}^{k}, y^{k}_{\tau}\right) - G_{x, \tau}^{k}\|^2 + \frac{1}{2\alpha_{x}}\left\|x_{\tau+1}^{k} - x_{\tau}^{k}\right\|^{2} - \frac{1}{\alpha_{x}}\left\|x_{\tau+1}^{k} - x_{\tau}^{k}\right\|^{2},
\end{align}
where the inequality follows from the Young's inequality and \eqref{eq:primal_descent_init}. Now taking expectation on both sides of \eqref{eq:ineq-smooth-Fr-prod} and  using \eqref{eq:lemma0_x}, we obtain
\begin{align*}
\EE\big[\left\langle \nabla_{x}F_{r}\left(x_{\tau}^{k}, y^{k}_{\tau}, z^{k}_{\tau}\right) , x^{k}_{\tau+1} - x^{k}_{\tau}\right\rangle\big]  \leq &~  \frac{\alpha_{x} L_{x}^{2}}{M}\sum_{b=0}^{\tau-1}\mathbb{E}\left\| x^{k}_{b+1} - x^{k}_{b}\right\|^{2} + \frac{\alpha_{x} L_{y}^{2}}{M}\sum_{b=0}^{\tau-1}\mathbb{E}\left\| y^{k}_{b+1} - y^{k}_{b}\right\|^{2} \\
\quad & -\frac{1}{2\alpha_{x}}\EE\left\|x_{\tau+1}^{k} - x_{\tau}^{k}\right\|^{2} {+\frac{\alpha_xC_{\sigma,x}}{2}}. 
\end{align*}
Taking expectation on both sides of \eqref{eq:ineq-smooth-Fr}, plugging the above inequality, and rearranging terms, we obtain 
\begin{align}
\mathbb{E}\left[ F_r\left(x^{k}_{\tau}, y^{k}_{\tau}, z^{k}_{\tau}\right) - F_r\left(x^{k}_{\tau+1}, y^{k}_{\tau}, z^{k}_{\tau}\right)\right] 
\geq &~   \left(\frac{1}{2\alpha_{x}} - \frac{r+L_{x}}{2}\right)\mathbb{E}\left\|x^{k}_{\tau+1} - x^{k}_{\tau}\right\|^{2} - \frac{\alpha_{x}L_{x}^{2}}{M}\sum_{b=0}^{\tau-1}\mathbb{E}\left\| x^{k}_{b+1} - x^{k}_{b}\right\|^{2}  \nonumber\\
\quad & - \frac{\alpha_{x}L_{y}^{2}}{M}\sum_{b=0}^{\tau-1}\mathbb{E}\left\| y^{k}_{b+1} - y^{k}_{b}\right\|^{2}{-\frac{\alpha_xC_{\sigma,x}}{2}}.
\label{eq:suff_decrease_1}
\end{align}

Second, for \textcircled{2} in \eqref{sufficient_decrease_x_only}, since $ \nabla_y F_r(x, \cdot, z) $ is $ L_{y} $-Lipschitz continuous, we have
\begin{align}
\mathbb{E}\left[F_r(x^{k}_{\tau + 1}, y^{k}_{\tau}, z^{k}_{\tau}) - F_r(x^{k}_{\tau + 1}, y^{k}_{\tau + 1}, z^{k}_{\tau})\right] \geq &~ \mathbb{E}\left[\langle \nabla_y F_r(x^{k}_{\tau + 1}, y^{k}_{\tau}, z^{k}_{\tau}), y^{k}_{\tau} - y^{k}_{\tau + 1} \rangle\right] - \frac{L_{y}}{2} \mathbb{E}\left\|y^{k}_{\tau + 1} - y^{k}_{\tau} \right\|^2. 
\label{eq:suff_decrease_2}
\end{align}

Third, for \textcircled{3} in \eqref{sufficient_decrease_x_only}, based on the update for $z$, we get
\begin{align}
\mathbb{E}\left[F_r(x^{k}_{\tau + 1}, y^{k}_{\tau + 1}, z^{k}_{\tau}) - F_r(x^{k}_{\tau + 1}, y^{k}_{\tau + 1}, z^{k}_{\tau + 1})\right] 
=\frac{(2 - \beta)r\beta}{2} \mathbb{E}\left\|x^{k}_{\tau+1} - z^{k}_{\tau}\right\|^2.
\label{eq:suff_decrease_3}
\end{align}
Plugging \eqref{eq:suff_decrease_1}, \eqref{eq:suff_decrease_2}, \eqref{eq:suff_decrease_3} into \eqref{sufficient_decrease_x_only} gives the desired inequality. 
\end{proof}

{The next two lemmas can be shown by the same proofs of those in {\citep[Lemma 6 and Lemma 7]{li2025nonsmooth}}, which fundamentally rely on weak convexity in $x$ and smoothness 
in $y$. Our assumptions on $F(x,y)$ align with these conditions, thus the proofs in \cite{li2025nonsmooth} 
apply directly to our setting.}

\begin{lemma}
Under Assumption \ref{assumption_for_smooth_case}, let $\{x_\tau^k, y_\tau^k,z_\tau^k \}$ be generated from Algorithm \ref{alg:smoothed_primal_dual_spider}. Then it holds for any $k\ge0$ and any $0\le \tau \le T-1$ that
\begin{align*}
d_{r}\left(y_{\tau+1}^{k}, z_{\tau+1}^{k}\right) - d_{r}\left(y_{\tau}^{k}, z_{\tau}^{k}\right)
\geq &~ \left\langle \nabla_y F_r(x_r(y^{k}_{\tau}, z^{k}_{\tau}), y^{k}_{\tau}, z^{k}_{\tau}), y^{k}_{\tau + 1} - y^{k}_{\tau} \right\rangle - \frac{L_{d_r}}{2} \left\|y^{k}_{\tau + 1} - y^{k}_{\tau} \right\|^2\\ 
\quad & + \frac{r}{2}\left\langle z^{k}_{\tau + 1} - z^{k}_{\tau}, z^{k}_{\tau + 1} + z^{k}_{\tau} - 2 x_r(y^{k}_{\tau + 1}, z^{k}_{\tau + 1})\right\rangle,
\end{align*}
where $L_{d_r}$ is given in Lemma~\ref{lem:smooth-dr-y}.
\label{lemma:sufficient_decrease_dual}
\end{lemma}
\begin{lemma}\label{lemma:sufficient_decrease_proximal}
Under Assumption \ref{assumption_for_smooth_case}, let $\{x_\tau^k, y_\tau^k,z_\tau^k \}$ be generated from Algorithm \ref{alg:smoothed_primal_dual_spider}, and $y\left(z^{k}_{\tau + 1}\right) \in Y\left(z^{k}_{\tau + 1}\right)$. Then it holds for any $k\ge0$ and any $0\le \tau \le T-1$ that
\begin{align*}
p_{r}\left(z_{\tau}^{k}\right) - p_{r}\left(z_{\tau+1}^{k}\right)
\geq &~ \frac{r}{2}\left\langle z^{k}_{\tau + 1} - z^{k}_{\tau}, 2 x_r\left(y\left(z^{k}_{\tau + 1}\right), z^{k}_{\tau}\right) - z^{k}_{\tau} - z^{k}_{\tau+1}\right\rangle.
\end{align*}
\end{lemma}
We still need the following lemma that upper bounds the actual $y$-iterate with the virtual iterate $y_{\tau,+}^k$.
\begin{lemma} 
Under Assumption \ref{assumption_for_smooth_case}, let $\{x_\tau^k, y_\tau^k,z_\tau^k \}$ be generated from Algorithm \ref{alg:smoothed_primal_dual_spider} with $r \geq 2 \rho$.
Then, it holds   
\begin{align}
&\mathbb{E} \left\| y_{\tau , +}^{k}(z^{k}_{\tau + 1}) - y^{k}_{\tau + 1}  \right\|^{2} \nonumber\\ 
\leq &~  24\alpha_{y}^2 \beta^{2} L_{y}^{2}\mathbb{E}\left\|x^{k}_{\tau + 1} - z^{k}_{\tau}\right\|^{2}  + 6 \alpha_{y}^2L_{y}^{2}\mathbb{E} \left\|x^{k}_{\tau + 1} - x_{r}(y^{k}_{\tau},z^{k}_{\tau})\right\|^{2} + 3\alpha_{y}^2 L_{y}^{2}\mathbb{E}\left\| x^{k}_{\tau+1} - x^{k}_{\tau}\right\|^{2}\nonumber\\
\quad &  + \frac{3\alpha_{y}^{2}L_{y}^{2}}{M}\sum_{b=0}^{\tau-1}\mathbb{E}\left\| y^{k}_{b+1} - y^{k}_{b}\right\|^{2} + \frac{3\alpha_{y}^{2}L_{y}^{2}}{M}\sum_{b=0}^{\tau-1}\mathbb{E}\left\| x^{k}_{b+1} - x^{k}_{b}\right\|^{2} {+3\alpha_y^2C_{\sigma,y}}\;.
\label{eq:y_minus_y_2}
\end{align}
\end{lemma}

\begin{proof}
By the definition of $y_{\tau , +}^{k}(z)$ and the update formula of $y$, it follows
\begin{align}
&\mathbb{E} \left\| y_{\tau , +}^{k}(z^{k}_{\tau + 1}) - y^{k}_{\tau + 1}  \right\|^{2} \nonumber\\ 
= &~ \mathbb{E}\left\| \text{proj}_{\mathcal{Y}}\left(y^{k}_{\tau} + \alpha_{y} \nabla_{y}F(x_{r}(y^{k}_{\tau},z^{k}_{\tau + 1}),y^{k}_{\tau})  \right) - \text{proj}_{\mathcal{Y}}\left(y^{k}_{\tau} + \alpha_{y} G_{y, \tau}^{k} \right)   \right\|^{2} \nonumber \\
\overset{(i)}{\leq} &~ \mathbb{E}\left\| y^{k}_{\tau} + \alpha_{y} \nabla_{y}F(x_{r}(y^{k}_{\tau},z^{k}_{\tau + 1}),y^{k}_{\tau})  - y^{k}_{\tau} - \alpha_{y} G_{y, \tau}^{k} \right\|^{2} \nonumber\\
\overset{(ii)}{\leq} &~  3\alpha_{y}^2\mathbb{E}\left\|\nabla_{y}F(x_{r}(y^{k}_{\tau},z^{k}_{\tau + 1}),y^{k}_{\tau}) - \nabla_{y}F(x^{k}_{\tau + 1},y^{k}_{\tau}) \right\|^{2}+ 3\alpha_{y}^2\mathbb{E} \left\|\nabla_{y}F(x^{k}_{\tau+1},y^{k}_{\tau}) -  \nabla_{y}F(x^{k}_{\tau},y^{k}_{\tau}) \right\|^{2}\nonumber\\
\quad & +3\alpha_{y}^2\mathbb{E} \left\|\nabla_{y}F(x^{k}_{\tau},y^{k}_{\tau}) -  G_{y, \tau}^{k} \right\|^{2} \nonumber \\
\overset{(iii)}{\leq} &~  3\alpha_{y}^2 L_{y}^{2}\mathbb{E} \left\|x_{r}\left(y^{k}_{\tau},z^{k}_{\tau + 1}\right) - x^{k}_{\tau + 1}\right\|^{2} + 3\alpha_{y}^2 L_{y}^{2}\mathbb{E}\left\| x^{k}_{\tau+1} - x^{k}_{\tau}\right\|^{2}  + \frac{3\alpha_{y}^{2}L_{y}^{2}}{M}\sum_{b=0}^{\tau-1}\mathbb{E}\left\| y^{k}_{b+1} - y^{k}_{b}\right\|^{2}\nonumber \\
\quad & +\frac{3\alpha_{y}^{2}L_{y}^{2}}{M}\sum_{b=0}^{\tau-1}\mathbb{E}\left\| x^{k}_{b+1} - x^{k}_{b}\right\|^{2}  {+3\alpha_y^2C_{\sigma,y}}\nonumber \\
\overset{(iv)}{\leq} &~  6\alpha_{y}^2 L_{y}^{2}\mathbb{E} \left\|x_{r}(y^{k}_{\tau},z^{k}_{\tau + 1}) - x_{r}(y^{k}_{\tau},z^{k}_{\tau})\right\|^{2} + 6 \alpha_{y}^2L_{y}^{2}\mathbb{E} \left\|x^{k}_{\tau + 1} - x_{r}(y^{k}_{\tau},z^{k}_{\tau})\right\|^{2} + 3\alpha_{y}^2 L_{y}^{2}\mathbb{E}\left\| x^{k}_{\tau+1} - x^{k}_{\tau}\right\|^{2} \nonumber\\
\quad &  + \frac{3\alpha_{y}^{2}L_{y}^{2}}{M}\sum_{b=0}^{\tau-1}\mathbb{E}\left\| y^{k}_{b+1} - y^{k}_{b}\right\|^{2} + \frac{3\alpha_{y}^{2}L_{y}^{2}}{M}\sum_{b=0}^{\tau-1}\mathbb{E}\left\| x^{k}_{b+1} - x^{k}_{b}\right\|^{2}{+3\alpha_y^2C_{\sigma,y}}, 
\label{eq:y_minus_y}
\end{align}
where $(i)$ follows from the non-expansiveness of the projection operator, $(ii)$ and $(iv)$ are by Young's inequality, and $(iii)$ holds from the $L_y\text{-Lipschitz continuity of }\nabla_{y}F(\cdot,y) $ and \eqref{eq:lemma0_y}. 
Moreover, by \eqref{eq:sc1}, it holds
\begin{align}
\mathbb{E} \left\|x_{r}(y^{k}_{\tau},z^{k}_{\tau + 1}) - x_{r}(y^{k}_{\tau},z^{k}_{\tau})\right\|^{2} 
{\leq} \sigma_{1}^{2}\mathbb{E} \left\|z^{k}_{\tau + 1} - z^{k}_{\tau}\right\|^{2}
{\leq}4\beta^{2}\mathbb{E} \left\|x^{k}_{\tau + 1} - z^{k}_{\tau}\right\|^{2},
\label{intermediate_y}
\end{align}
where 
the second inequality follows from the update of $z$ in Algorithm \ref{alg:smoothed_primal_dual_spider} and $\sigma_{1} \leq 2$ (since $r\geq 2\rho)$. Plugging \eqref{intermediate_y} into \eqref{eq:y_minus_y} yields the desired result.
\end{proof}
{We are now ready to establish the sufficient decrease condition for the Lyapunov function defined in \eqref{Lyapunov_Function_Main}.}
\begin{lemma}\label{lemma:Sufficient_Decrease_Init}
Under Assumption \ref{assumption_for_smooth_case}, let $\{x_\tau^k, y_\tau^k,z_\tau^k \}$ be generated from Algorithm \ref{alg:smoothed_primal_dual_spider} with $r\geq \max\{2\rho, L_{y}+\rho\}$. Then it holds for any $k\ge0$ and any $0\le \tau \le T-1$ that
\begin{align}\label{eq:Sufficient_Decrease_Init}
&\mathbb{E}\left[\Phi_r(x^{k}_{\tau}, y^{k}_{\tau}, z^{k}_{\tau}) - \Phi_r(x^{k}_{\tau + 1}, y^{k}_{\tau + 1}, z^{k}_{\tau + 1})\right]\\
\geq &~ \left(\frac{1}{2\alpha_{x}} - \frac{r+L_{x}}{2}-16\alpha_{y}L_{y}^{2}-540r\beta\alpha_{y}^{2}L_{y}^{2}\right)\mathbb{E}\left\|x^{k}_{\tau+1} - x^{k}_{\tau}\right\|^{2} + \left(\frac{11}{16\alpha_{y}} - 5L_{y}\right)\mathbb{E}\left\|y^{k}_{\tau + 1} - y^{k}_{\tau} \right\|^2 \nonumber\\
\quad & + \left(\frac{(2-\beta)r\beta}{2} - 4r\beta^{2}-\frac{r\beta}{10} -4320r\beta^{3}\alpha_{y}^{2}L_{y}^{2}\right)\mathbb{E} \left\|x^{k}_{\tau + 1} -z^{k}_{\tau}\right\|^{2}\nonumber\\
\quad & - \left(1080r\beta\alpha_{y}^{2}L_{y}^{2} + 4\alpha_{y}L_{y}^{2}\right)\mathbb{E} \left\|x^{k}_{\tau + 1} -x_{r}(y_{\tau}^{k},z_{\tau}^{k})\right\|^{2} - \left(\frac{16\alpha_{y}L_{y}^{2}+540r\beta\alpha_{y}^{2}L_{y}^{2}+\alpha_{x}L_{x}^{2}}{M}\right)\sum_{b=0}^{\tau-1}\mathbb{E}\left\| x^{k}_{b+1} - x^{k}_{b}\right\|^{2}\nonumber\\
\quad & -\left(\frac{16\alpha_{y}L_{y}^{2}+540r\beta\alpha_{y}^{2}L_{y}^{2}+\alpha_{x}
{L_{y}^{2}}}{M}\right)\sum_{b=0}^{\tau-1}\mathbb{E}\left\| y^{k}_{b+1} - y^{k}_{b}\right\|^{2}\nonumber\\
\quad & -20r\beta\mathbb{E}\left\|x_{r}\left( y(z^{k}_{\tau + 1}),z^{k}_{\tau + 1}  \right)-x_{r}\left( {y_{\tau,+}^k}(z^{k}_{\tau + 1}),z^{k}_{\tau + 1}  \right)\right\|^{2}{-\frac{\alpha_xC_{\sigma,x}}{2} - \left(16\alpha_y+540r\beta\alpha_y^2\right)C_{\sigma,y}}\;.\nonumber
\end{align}
\end{lemma}
\begin{proof} Our proof technique 
is similar to that in \cite{li2025nonsmooth}, upon which we build to accommodate our setting and extend the result. By the definition of $\Phi_r$ in \eqref{Lyapunov_Function_Main}, we utilize $\sigma_{2} \leq 3$ (since $r\geq L_{y} + \rho)$,  Lemmas \ref{lemma:sufficient_decrease_primal}, \ref{lemma:sufficient_decrease_dual}, and \ref{lemma:sufficient_decrease_proximal} to have 
\begin{align}
&\mathbb{E}\left[\Phi_r(x^{k}_{\tau}, y^{k}_{\tau}, z^{k}_{\tau}) - \Phi_r(x^{k}_{\tau + 1}, y^{k}_{\tau + 1}, z^{k}_{\tau + 1})\right] \nonumber\\
{=} &~ \mathbb{E}\left[F_r\left(x_{\tau}^{k}, y_{\tau}^{k}, z_{\tau}^{k}\right) - F_r\left(x_{\tau+1}^{k}, y_{\tau+1}^{k}, z_{\tau+1}^{k}\right)\right] + 2\mathbb{E}\left[d_r\left(y_{\tau+1}^{k}, z_{\tau+1}^{k}\right) - d_r\left(y_{\tau}^{k}, z_{\tau}^{k}\right)\right]  + 2\mathbb{E}\left[p_r\left(z_{\tau}^{k}\right) - p_r\left(z_{\tau+1}^{k}\right) \right]\nonumber\\
\geq &~ \left(\frac{1}{2\alpha_{x}} - \frac{r+L_{x}}{2}\right)\mathbb{E}\left\|x^{k}_{\tau+1} - x^{k}_{\tau}\right\|^{2} - \frac{\alpha_{x}L_{x}^{2}}{M}\sum_{b=0}^{\tau-1}\mathbb{E}\left\| x^{k}_{b+1} - x^{k}_{b}\right\|^{2} - \frac{\alpha_{x}L_{y}^{2}}{M}\sum_{b=0}^{\tau-1}\mathbb{E}\left\| y^{k}_{b+1} - y^{k}_{b}\right\|^{2}   \nonumber \\
\quad & +  \mathbb{E}\left\langle \nabla_y F_r\left(x^{k}_{\tau + 1}, y^{k}_{\tau}, z^{k}_{\tau}\right), y^{k}_{\tau} - y^{k}_{\tau + 1} \right\rangle- \frac{L_{y}}{2} \mathbb{E}\left\|y^{k}_{\tau + 1} - y^{k}_{\tau} \right\|^2 + \frac{(2 - \beta)r\beta}{2} \mathbb{E}\left\|x^{k}_{\tau+1} - z^{k}_{\tau}\right\|^2 \nonumber\\ 
\quad &  + \mathbb{E}\left[\left\langle 2\nabla_y F_r(x_r(y^{k}_{\tau}, z^{k}_{\tau}), y^{k}_{\tau}, z^{k}_{\tau}), y^{k}_{\tau + 1} - y^{k}_{\tau} \right\rangle - L_ {d_r} \|y^{k}_{\tau} - y^{k}_{\tau + 1}\|^2\right]\nonumber\\ 
\quad & + r\mathbb{E}\left[\left\langle z^{k}_{\tau + 1} - z^{k}_{\tau}, z^{k}_{\tau + 1} + z^{k}_{\tau} - 2 x_r(y^{k}_{\tau + 1}, z^{k}_{\tau + 1})\right\rangle \right] + r\mathbb{E}\left[\left\langle z^{k}_{\tau + 1} - z^{k}_{\tau}, 2 x_r(y(z^{k}_{\tau + 1}), z^{k}_{\tau}) - z^{k}_{\tau} - z^{k}_{\tau + 1}\right\rangle \right]{-\frac{\alpha_xC_{\sigma,x}}{2}}\nonumber\\
{\geq} &~  \left(\frac{1}{2\alpha_{x}} - \frac{r+L_{x}}{2}\right)\mathbb{E}\left\|x^{k}_{\tau+1} - x^{k}_{\tau}\right\|^{2} - 5L_{y} \mathbb{E}\left\|y^{k}_{\tau + 1} - y^{k}_{\tau} \right\|^2 + \frac{(2 - \beta)r\beta}{2} \mathbb{E}\left\|x^{k}_{\tau+1} - z^{k}_{\tau}\right\|^2   \nonumber\\
\quad & + \underbrace{\mathbb{E}\left[\left\langle 2\nabla_y F_r(x_r(y^{k}_{\tau}, z^{k}_{\tau}), y^{k}_{\tau}, z^{k}_{\tau}), 
-\nabla_y F_r\left(x^{k}_{\tau + 1}, y^{k}_{\tau}, z^{k}_{\tau}\right), y^{k}_{\tau+1} - y^{k}_{\tau } \right\rangle\right]}_{\textcircled{1}}  \nonumber\\
\quad &  + \underbrace{2r\mathbb{E}\left[\left\langle z^{k}_{\tau + 1} - z^{k}_{\tau}, x_r(y(z^{k}_{\tau + 1}), z^{k}_{\tau}) - x_{r}(y^{k}_{\tau + 1},z^{k}_{\tau + 1})\right\rangle \right]}_{\textcircled{2}} - \frac{\alpha_{x}L_{x}^{2}}{M}\sum_{b=0}^{\tau-1}\mathbb{E}\left\| x^{k}_{b+1} - x^{k}_{b}\right\|^{2} \nonumber\\
\quad & - \frac{\alpha_{x}L_{y}^{2}}{M}\sum_{b=0}^{\tau-1}\mathbb{E}\left\| y^{k}_{b+1} - y^{k}_{b}\right\|^{2}{-\frac{\alpha_xC_{\sigma,x}}{2}}.\label{eq:orig_Lyapunov_decrease}
\end{align}
Below we bound \textcircled{1} and \textcircled{2} in \eqref{eq:orig_Lyapunov_decrease}.
First, we have 
\begin{align}
&\mathbb{E}\left\langle 2\nabla_y F_r(x_r(y^{k}_{\tau}, z^{k}_{\tau}), y^{k}_{\tau}, z^{k}_{\tau}) - \nabla_y F_r(x^{k}_{\tau + 1}, y^{k}_{\tau}, z^{k}_{\tau}), y^{k}_{\tau + 1} - y^{k}_{\tau} \right\rangle \nonumber\\ 
= &~ 2\mathbb{E}\left\langle \nabla_y F_r(x_r(y^{k}_{\tau}, z^{k}_{\tau}), y^{k}_{\tau}, z^{k}_{\tau}) - \nabla_y F_r(x^{k}_{\tau + 1}, y^{k}_{\tau}, z^{k}_{\tau}), y^{k}_{\tau + 1} - y^{k}_{\tau} \right\rangle + \mathbb{E}\left\langle\nabla_y F_r(x^{k}_{\tau + 1}, y^{k}_{\tau}, z^{k}_{\tau}), y^{k}_{\tau + 1} - y^{k}_{\tau} \right\rangle. \label{eq:decrease_Lyapunov_init}
\end{align}
For the first inner product term in \eqref{eq:decrease_Lyapunov_init}, we bound it by
\begin{align}
&2\mathbb{E}\left\langle \nabla_y F_r(x_r(y^{k}_{\tau}, z^{k}_{\tau}), y^{k}_{\tau}, z^{k}_{\tau}) - \nabla_y F_r(x^{k}_{\tau + 1}, y^{k}_{\tau}, z^{k}_{\tau}), y^{k}_{\tau + 1} - y^{k}_{\tau} \right\rangle \nonumber\\
\overset{(i)}{\geq} & -4\alpha_{y}\mathbb{E}\left\|\nabla_y F_r(x_r(y^{k}_{\tau}, z^{k}_{\tau}), y^{k}_{\tau}, z^{k}_{\tau}) - \nabla_y F_r(x^{k}_{\tau + 1}, y^{k}_{\tau}, z^{k}_{\tau})\right\|^{2} -\frac{1}{4\alpha_{y}}\mathbb{E}\left\|y^{k}_{\tau + 1} - y^{k}_{\tau}\right\|^{2}\nonumber\\
\overset{(ii)}{\geq} & -4\alpha_{y} L_{y}^{2}\mathbb{E} \left\|x^{k}_{\tau + 1} -  x_r(y^{k}_{\tau}, z^{k}_{\tau})  \right\|^{2} - \frac{1}{4\alpha_{y}} \mathbb{E}\left\|y^{k}_{\tau + 1} - y^{k}_{\tau}\right\|^{2},  \label{eq:Lyapunov_2}
\end{align}
where $(i)$ follows from Young's inequality, 
and $(ii)$ is by the $L_{y}$ Lipschitz continuity of $\nabla_{y}F_{r}(\cdot,y,z)$.
For the second inner product term in \eqref{eq:decrease_Lyapunov_init}, we split and bound it by
\begin{align}
&\mathbb{E}\left\langle\nabla_y F_r(x^{k}_{\tau + 1}, y^{k}_{\tau}, z^{k}_{\tau}), y^{k}_{\tau + 1} - y^{k}_{\tau} \right\rangle \nonumber\\
= &~ \frac{1}{\alpha_{y}}\mathbb{E}\left\langle  y^{k}_{\tau} + \alpha_{y} G_{y, \tau}^{k} - y^{k}_{\tau + 1}, y^{k}_{\tau + 1} - y^{k}_{\tau} \right\rangle - \mathbb{E} \left\langle G_{y, \tau}^{k}  - \nabla_y F_r(x^{k}_{\tau + 1}, y^{k}_{\tau}, z^{k}_{\tau}) , y^{k}_{\tau + 1} - y^{k}_{\tau} \right\rangle  + \frac{1}{\alpha_{y}}\mathbb{E}\left\|y^{k}_{\tau + 1} - y^{k}_{\tau}\right\|^{2}\nonumber\\
\ge & ~- \mathbb{E} \left\langle G_{y, \tau}^{k}  - \nabla_y F_r(x^{k}_{\tau + 1}, y^{k}_{\tau}, z^{k}_{\tau}) , y^{k}_{\tau + 1} - y^{k}_{\tau} \right\rangle  + \frac{1}{\alpha_{y}}\mathbb{E}\left\|y^{k}_{\tau + 1} - y^{k}_{\tau}\right\|^{2},  \label{intermediate_1}
\end{align}
where the inequality holds because $\left\langle y^{k}_{\tau} + \alpha_{y}G_{y, \tau}^{k} - y^{k}_{\tau+1}, y_{\tau+1}^{k} - y_{\tau}^{k} \right\rangle \geq 0$ by the update of $y$.
The first inner product term in \eqref{intermediate_1} is bounded by
\begin{align}
&\mathbb{E} \left\langle G_{y, \tau}^{k}  - \nabla_y F_r(x^{k}_{\tau + 1}, y^{k}_{\tau}, z^{k}_{\tau}) , y^{k}_{\tau + 1} - y^{k}_{\tau} \right\rangle \nonumber\\
= &~\mathbb{E} \left\langle G_{y, \tau}^{k}  - \nabla_y F(x^{k}_{\tau + 1}, y^{k}_{\tau}) , y^{k}_{\tau + 1} - y^{k}_{\tau} \right\rangle \nonumber\\
\leq &~ 8\alpha_{y}\mathbb{E}\left\| G_{y, \tau}^{k}  - \nabla_y F(x^{k}_{\tau + 1}, y^{k}_{\tau}) \right\|^{2} + \frac{1}{16\alpha_{y}}\mathbb{E}\left\|y^{k}_{\tau + 1} - y^{k}_{\tau} \right\|^{2} \nonumber\\
\leq &~ 16\alpha_{y}\mathbb{E}\left\| G_{y, \tau}^{k}  - \nabla_y F(x^{k}_{\tau}, y^{k}_{\tau}) \right\|^{2} + 16\alpha_{y}\mathbb{E}\left\| \nabla_y F(x^{k}_{\tau}, y^{k}_{\tau})  - \nabla_y F(x^{k}_{\tau + 1}, y^{k}_{\tau}) \right\|^{2} + \frac{1}{16\alpha_{y}}\mathbb{E}\left\|y^{k}_{\tau + 1} - y^{k}_{\tau} \right\|^{2} \nonumber\\
\leq &~ \frac{16\alpha_{y}L_{y}^{2}}{M}\sum_{b=0}^{\tau-1}\mathbb{E}\left\| y^{k}_{b+1} - y^{k}_{b}\right\|^{2} + \frac{16\alpha_{y}L_{y}^{2}}{M}\sum_{b=0}^{\tau-1}\mathbb{E}\left\| x^{k}_{b+1} - x^{k}_{b}\right\|^{2}+ 16\alpha_{y}L_{y}^{2}\mathbb{E}\left\|x^{k}_{\tau + 1} - x^{k}_{\tau} \right\|^{2} + \frac{1}{16\alpha_{y}}\mathbb{E}\left\|y^{k}_{\tau + 1} - y^{k}_{\tau} \right\|^{2}\nonumber\\
\quad & {+16\alpha_yC_{\sigma,y}}\;,
\label{intermediate_1_second}
\end{align}
where 
the last inequality follows from \eqref{eq:lemma0_y} and the $L_{y}$-Lipschitz continuity of $\nabla_{y}F(\cdot,y)$. Plugging \eqref{intermediate_1_second} into \eqref{intermediate_1} we get
\begin{align}
&\mathbb{E}\left\langle\nabla_y F_r(x^{k}_{\tau + 1}, y^{k}_{\tau}, z^{k}_{\tau}), y^{k}_{\tau + 1} - y^{k}_{\tau} \right\rangle \nonumber\\
\geq &~ \frac{15}{16\alpha_{y}}\mathbb{E}\left\| y_{\tau+1}^{k} - y_{\tau}^{k} \right\|^{2} - \frac{16\alpha_{y}L_{y}^{2}}{M}\sum_{b=0}^{\tau-1}\mathbb{E}\left\| y^{k}_{b+1} - y^{k}_{b}\right\|^{2} - 16\alpha_{y}L_{y}^{2}\mathbb{E}\left\| x_{\tau+1}^{k} - x_{\tau}^{k} \right\|^{2} - \frac{16\alpha_{y}L_{y}^{2}}{M}\sum_{b=0}^{\tau-1}\mathbb{E}\left\| x^{k}_{b+1} - x^{k}_{b}\right\|^{2}\nonumber\\
\quad & {-16\alpha_yC_{\sigma,y}}\;. \label{eq:Lyapunov_1}
\end{align}
Plugging \eqref{eq:Lyapunov_2} and \eqref{eq:Lyapunov_1}  into \eqref{eq:decrease_Lyapunov_init} yields
\begin{align}
&\mathbb{E}\left\langle 2\nabla_y F_r(x_r(y^{k}_{\tau}, z^{k}_{\tau}), y^{k}_{\tau}, z^{k}_{\tau}) - \nabla_y F_r(x^{k}_{\tau + 1}, y^{k}_{\tau}, z^{k}_{\tau}), y^{k}_{\tau + 1} - y^{k}_{\tau} \right\rangle \nonumber\\
\geq &~ \frac{11}{16\alpha_{y}}\mathbb{E}\left\|y^{k}_{\tau + 1} - y^{k}_{\tau}\right\|^{2}- \frac{16\alpha_{y}L_{y}^{2}}{M}\sum_{b=0}^{\tau-1}\mathbb{E}\left\| y^{k}_{b+1} - y^{k}_{b}\right\|^{2}- 16\alpha_{y}L_{y}^{2}\mathbb{E}\left\|x^{k}_{\tau+1} - x^{k}_{\tau}\right\|^{2}   - \frac{16\alpha_{y}L_{y}^{2}}{M}\sum_{b=0}^{\tau-1}\mathbb{E}\left\| x^{k}_{b+1} - x^{k}_{b}\right\|^{2}\nonumber\\
\quad & -4\alpha_{y} L_{y}^{2}\mathbb{E} \left\|x^{k}_{\tau + 1} -  x_r(y^{k}_{\tau}, z^{k}_{\tau})  \right\|^{2}{-16\alpha_yC_{\sigma,y}}\;.\label{eq:final_term_2}
\end{align}

Second, we bound \textcircled{2} in \eqref{eq:orig_Lyapunov_decrease} by
\begin{align}
&2r\mathbb{E}\left[\left\langle z^{k}_{\tau + 1} - z^{k}_{\tau}, x_r(y(z^{k}_{\tau + 1}), z^{k}_{\tau}) - x_{r}(y^{k}_{\tau + 1},z^{k}_{\tau + 1})\right\rangle \right] \nonumber\\
= &~ 2r\mathbb{E}\left[\left\langle  z^{k}_{\tau+1}-z^{k}_{\tau}, x_r(y(z^{k}_{\tau + 1}), z^{k}_{\tau}) - x_r(y(z^{k}_{\tau + 1}), z^{k}_{\tau+1})\right\rangle \right] \nonumber\\
\quad & +2r\mathbb{E}\left[\left\langle  z^{k}_{\tau+1}-z^{k}_{\tau }, x_r(y(z^{k}_{\tau + 1}), z^{k}_{\tau+1}) - x_{r}(y^{k}_{\tau + 1},z^{k}_{\tau + 1})\right\rangle \right]  \nonumber\\
\overset{(i)}{\geq} &~ -2r\mathbb{E}\big(\|z^{k}_{\tau + 1}-z^{k}_{\tau}\|\left\|x_r(y(z^{k}_{\tau + 1}), z^{k}_{\tau}) - x_r(y(z^{k}_{\tau + 1}), z^{k}_{\tau+1})\right\|\big) \nonumber\\
\quad & -2r\mathbb{E}\big(\|z^{k}_{\tau + 1}-z^{k}_{\tau}\|\left\|x_{r}\left( y(z^{k}_{\tau + 1}),z^{k}_{\tau + 1}  \right)-x_{r}\left( y^{k}_{\tau + 1},z^{k}_{\tau + 1}  \right)\right\|\big)   \nonumber\\
\overset{(ii)}{\geq} &~ -2r\sigma_{1}\mathbb{E}\|z^{k}_{\tau + 1}-z^{k}_{\tau}\|^{2} - \frac{r}{10\beta}\mathbb{E}\|z^{k}_{\tau + 1}-z^{k}_{\tau}\|^{2}  - 10r\beta\mathbb{E}\left\|x_{r}\left( y(z^{k}_{\tau + 1}),z^{k}_{\tau + 1}  \right)-x_{r}\left( y^{k}_{\tau + 1},z^{k}_{\tau + 1}  \right)\right\|^{2} \nonumber\\
\overset{(iii)}{\geq} &~ -\left(4r\beta^{2}+\frac{r\beta}{10}\right)\mathbb{E}\left\|x^{k}_{\tau + 1}-z^{k}_{\tau}\right\|^{2} - 10r\beta\mathbb{E}\left\|x_{r}\left( y(z^{k}_{\tau + 1}),z^{k}_{\tau + 1}  \right)-x_{r}\left( y^{k}_{\tau + 1},z^{k}_{\tau + 1}  \right)\right\|^{2},
\label{eq:third_term_decrease}
\end{align}
where $(i)$ holds by Cauchy-Schwarz inequality, $(ii)$ follows from \eqref{eq:sc1} and Young's inequality, and $(iii)$ is obtained from the update of $z$ and  $\sigma_{1} \leq 2$ {(since $r \geq 2\rho$)}. For the second term in \eqref{eq:third_term_decrease}, we bound it by
\begin{align}
&\mathbb{E}\left\|x_{r}\left( y(z^{k}_{\tau + 1}),z^{k}_{\tau + 1}  \right)-x_{r}\left( y^{k}_{\tau + 1},z^{k}_{\tau + 1}  \right)\right\|^{2}\nonumber\\ 
\overset{(i)}{\leq} &~\mathbb{E} \left(  2\left\|x_{r}\left( y(z^{k}_{\tau + 1}),z^{k}_{\tau + 1}  \right)-x_{r}\left( y_{\tau, +}^k(z^{k}_{\tau + 1}),z^{k}_{\tau + 1}  \right)\right\|^{2}  + 2\left\| x_{r}\left( y_{\tau, +}^k(z^{k}_{\tau + 1}),z^{k}_{\tau + 1}  \right)-x_{r}\left( y^{k}_{\tau + 1},z^{k}_{\tau + 1}  \right)  \right\|^{2}  \right)\nonumber\\
\overset{(ii)}{\leq} &~ 2\mathbb{E}   \left\|x_{r}\left( y(z^{k}_{\tau + 1}),z^{k}_{\tau + 1}  \right)-x_{r}\left( y_{\tau, +}^k(z^{k}_{\tau + 1}),z^{k}_{\tau + 1}  \right)\right\|^{2} + 2\sigma_{2}^{2} \mathbb{E}\left\|y_{\tau, +}^k(z^{k}_{\tau + 1}) - y^{k}_{\tau + 1}   \right\|^{2} \nonumber\\
\overset{(iii)}{\leq} &~ 2\mathbb{E}   \left\|x_{r}\left( y(z^{k}_{\tau + 1}),z^{k}_{\tau + 1}  \right)-x_{r}\left( y_{\tau, +}^k(z^{k}_{\tau + 1}),z^{k}_{\tau + 1}  \right)\right\|^{2}  + 18 \Bigg(  24\alpha_{y}^2 \beta^{2} L_{y}^{2}\mathbb{E}\left\|x^{k}_{\tau + 1} - z^{k}_{\tau}\right\|^{2}   \nonumber\\
\quad & + 6 \alpha_{y}^2L_{y}^{2}\mathbb{E} \left\|x^{k}_{\tau + 1} - x_{r}(y^{k}_{\tau},z^{k}_{\tau})\right\|^{2} + 3\alpha_{y}^2 L_{y}^{2}\mathbb{E}\left\| x^{k}_{\tau+1} - x^{k}_{\tau}\right\|^{2}  + \frac{3\alpha_{y}^{2}L_{y}^{2}}{M}\sum_{b=0}^{\tau-1}\mathbb{E}\left\| y^{k}_{b+1} - y^{k}_{b}\right\|^{2}  \nonumber\\
\quad & + \frac{3\alpha_{y}^{2}L_{y}^{2}}{M}\sum_{b=0}^{\tau-1}\mathbb{E}\left\| x^{k}_{b+1} - x^{k}_{b}\right\|^{2} {+3\alpha_y^2C_{\sigma,y}} \Bigg) \nonumber\\
= &~2\mathbb{E}   \left\|x_{r}\left( y(z^{k}_{\tau + 1}),z^{k}_{\tau + 1}  \right)-x_{r}\left( y_{\tau, +}^k(z^{k}_{\tau + 1}),z^{k}_{\tau + 1}  \right)\right\|^{2}+ 432\beta^{2}\alpha_{y}^{2}L_{y}^{2}\mathbb{E} \left\|x^{k}_{\tau + 1} -z^{k}_{\tau}\right\|^{2}   \nonumber\\ 
\quad & + 108\sigma_{2}^{2}\alpha_{y}^{2}L_{y}^{2}\mathbb{E} \left\|x^{k}_{\tau + 1} - x_{r}(y^{k}_{\tau},z^{k}_{\tau})\right\|^{2} + 54\alpha_{y}^{2}L_{y}^{2}\mathbb{E}\left\| x^{k}_{\tau+1} - x^{k}_{\tau}\right\|^{2} + \frac{54\alpha_{y}^{2}L_{y}^{2}}{M}\sum_{b=0}^{\tau-1}\mathbb{E}\left\| y^{k}_{b+1} - y^{k}_{b}\right\|^{2} \nonumber\\ 
\quad & + \frac{54\alpha_{y}^{2}L_{y}^{2}}{M}\sum_{b=0}^{\tau-1}\mathbb{E}\left\| x^{k}_{b+1} - x^{k}_{b}\right\|^{2}{+54\alpha_y^2C_{\sigma,y}},
\label{third_term_decrease_2}
\end{align}
where the $(i)$ holds by Young's inequality, $(ii)$ follows from \eqref{eq:sc2}, while $(iii)$ is obtained from $\sigma_{2} \leq 3$ (since $r \geq L_{y} + \rho$) and \eqref{eq:y_minus_y_2}. Plugging \eqref{third_term_decrease_2} into \eqref{eq:third_term_decrease} gives
\begin{align}
&2r\mathbb{E}\left[\left\langle z^{k}_{\tau + 1} - z^{k}_{\tau}, x_r(y(z^{k}_{\tau + 1}), z^{k}_{\tau}) - x_{r}(y^{k}_{\tau + 1},z^{k}_{\tau + 1})\right\rangle \right] \nonumber\\  
\geq &~ -\left(4r\beta^{2}+\frac{r\beta}{10} + 4320r\beta^{3}\alpha_{y}^{2}L_{y}^{2}\right)\mathbb{E} \left\|x^{k}_{\tau + 1} -z^{k}_{\tau}\right\|^{2}  -\left(1080r\beta\alpha_{y}^{2}L_{y}^{2}\right)\mathbb{E} \left\|x^{k}_{\tau + 1} - x_{r}(y^{k}_{\tau},z^{k}_{\tau})\right\|^{2}\nonumber\\
\quad &  -540r\beta\alpha_{y}^{2}L_{y}^{2}\mathbb{E}\left\| x^{k}_{\tau+1} - x^{k}_{\tau}\right\|^{2} -\frac{540r\beta\alpha_{y}^{2}L_{y}^{2}}{M}\sum_{b=0}^{\tau-1}\mathbb{E}\left\| y^{k}_{b+1} - y^{k}_{b}\right\|^{2}    - \frac{540r\beta\alpha_{y}^{2}L_{y}^{2}}{M}\sum_{b=0}^{\tau-1}\mathbb{E}\left\| x^{k}_{b+1} - x^{k}_{b}\right\|^{2}\nonumber\\
\quad &-20r\beta\mathbb{E}\left\|x_{r}\left( y(z^{k}_{\tau + 1}),z^{k}_{\tau + 1}  \right)-x_{r}\left( y_{\tau, +}^k(z^{k}_{\tau + 1}),z^{k}_{\tau + 1}  \right)\right\|^{2}{-540r\beta\alpha_y^2C_{\sigma,y}}. \label{eq:final_third_term}
\end{align}
Now we obtain the desired result by 
plugging \eqref{eq:final_term_2} and \eqref{eq:final_third_term} into \eqref{eq:orig_Lyapunov_decrease}.
\end{proof}

\begin{lemma}\label{lemma:dual:final}
Under Assumptions \ref{assumption_for_smooth_case} and \ref{assumption_KL_for_smooth_case}, let 
$\{x_{\tau}^k, y_{\tau}^k, z_{\tau}^k\}$ be generated by Algorithm \ref{alg:smoothed_primal_dual_spider} with $\beta \leq \frac{L_{y}}{20r\varpi}$ if $\theta \in [0,\frac{1}{2}]$, where $\varpi$ is given in Lemma~\ref{dual_error_bound}. Then
for any $k \geq 0$ and any $0 \leq \tau \leq T-1$, it holds
\begin{align}
&20r\beta   \left\|x_{r}\left( y(z^{k}_{\tau + 1}),z^{k}_{\tau + 1}  \right)-x_{r}\left( y_{\tau , +}^k(z^{k}_{\tau + 1}),z^{k}_{\tau + 1}  \right)\right\|^{2} \nonumber\\
\leq &~ 
L_{y} \left\|y^{k}_{\tau} - y_{\tau , +}^{k}(z^{k}_{\tau + 1})\right\|^{2} + \chi_{\theta}C_{\beta},
\label{eq:dual_theta_applied_1}
\end{align}
where $C_{\beta} = \left(\left(\frac{2\theta-1}{2\theta}\right)\left(\frac{20r\kappa}{\left(2\theta L_{y}\right)^{\frac{1}{2\theta}}}\right)^{\frac{2\theta}{2\theta-1}}\right)\beta^{\frac{2\theta}{2\theta-1}}$ with $\kappa$ defined in Lemma~\ref{dual_error_bound}.
\end{lemma}
\begin{proof} 
When $\theta \in [0,\frac{1}{2}]$, since $\beta \leq \frac{L_{y}}{20r\varpi}$, we have from \eqref{eq:dual_theta_half} that
\begin{align}
&20r\beta   \left\|x_{r}\left( y(z^{k}_{\tau + 1}),z^{k}_{\tau + 1}  \right)-x_{r}\left( y_{\tau , +}^k(z^{k}_{\tau + 1}),z^{k}_{\tau + 1}  \right)\right\|^{2}\nonumber\\ 
\leq &~ 20r\beta\varpi \left\|y^{k}_{\tau} - y_{\tau , +}^{k}(z^{k}_{\tau + 1})\right\|^{2}  \nonumber\\ 
\leq &~ L_{y} \left\|y^{k}_{\tau} - y_{\tau , +}^{k}(z^{k}_{\tau + 1})\right\|^{2}. \label{dual_final_bound_1}
\end{align}
When $\theta \in \left(\frac{1}{2},1\right]$, 
we have from \eqref{eq:dual_theta_1} that
\begin{align}
&20r\beta   \left\|x_{r}\left( y(z^{k}_{\tau + 1}),z^{k}_{\tau + 1}  \right)-x_{r}\left( y_{\tau , +}^k(z^{k}_{\tau + 1}),z^{k}_{\tau + 1}  \right)\right\|^{2}\nonumber\\
\leq &~ 20r\beta\kappa \left\|y^{k}_{\tau} - y_{\tau , +}^{k}(z^{k}_{\tau + 1})\right\|^{\frac{1}{\theta}}  \nonumber\\
= &~\left(\frac{20r\kappa\beta}{\left(2\theta L_{y}\right)^{\frac{1}{2\theta}}}\right)\left(\left(2\theta L_{y}\right)^{\frac{1}{2\theta}}   \left\|y^{k}_{\tau} - y_{\tau , +}^{k}(z^{k}_{\tau + 1})\right\|^{\frac{1}{\theta}}\right) \nonumber\\
\overset{(i)}{\leq} &~ \left(\frac{2\theta-1}{2\theta}\right)\left(\frac{20r\kappa\beta}{\left(2\theta L_{y}\right)^{\frac{1}{2\theta}}}\right)^{\frac{2\theta}{2\theta-1}} + L_{y} \left(   \left\|y^{k}_{\tau} - y_{\tau , +}^{k}(z^{k}_{\tau + 1})\right\|^{\frac{1}{\theta}}\right)^{2\theta},  \label{dual_final_bound_2}
\end{align}
where $(i)$ follows from Young's inequality for products, i.e., $ab \leq \frac{a^{p}}{p} + \frac{b^{q}}{q} \text{ where } a,b \geq 0, \;p, \;q > 1 \text{ such that }\linebreak \frac{1}{p}+\frac{1}{q} = 1$.

Combining \eqref{dual_final_bound_1} and \eqref{dual_final_bound_2} and utilizing the definition of $\chi_{\theta}$ yields the desired result.
\end{proof}
\begin{lemma}\label{lemma:dual:final:2}
Suppose Assumption \ref{assumption_for_smooth_case} holds and let  $
 \alpha_y \leq \frac{1}{10 L_{y}}, \; \beta \leq \frac{1}{30}.
$ Then for any $k\ge0$ and any $0\le \tau \le T-1$, we have
\begin{align*}
\mathbb{E}\left\|y^{k}_{\tau} - y_{\tau , +}^{k}(z^{k}_{\tau + 1})\right\|^{2} 
\leq &~ 2\mathbb{E}\left\|y^{k}_{\tau + 1} - y^{k}_{\tau}\right\|^{2} 
 + \frac{2\beta}{125}\mathbb{E}\left\|x^{k}_{\tau + 1} - z^{k}_{\tau}\right\|^{2}  + \frac{3}{25}\mathbb{E} \left\|x^{k}_{\tau + 1} - x_{r}(y^{k}_{\tau},z^{k}_{\tau})\right\|^{2} \nonumber\\
\quad &  + \frac{3}{50}\mathbb{E}\left\| x^{k}_{\tau+1} - x^{k}_{\tau}\right\|^{2} + \frac{3}{50M}\sum_{b=0}^{\tau-1}\mathbb{E}\left\| y^{k}_{b+1} - y^{k}_{b}\right\|^{2} + \frac{3}{50M}\sum_{b=0}^{\tau-1}\mathbb{E}\left\| x^{k}_{b+1} - x^{k}_{b}\right\|^{2}{+6\alpha_{y}^2C_{\sigma,y}}
\end{align*}
\end{lemma}
\begin{proof}Utilizing Young's inequality, we get
\begin{align*}
&~\mathbb{E}\left\|y^{k}_{\tau} - y_{\tau , +}^{k}(z^{k}_{\tau + 1})\right\|^{2}\\
\leq &~ 2\mathbb{E}\left\|y^{k}_{\tau + 1} - y^{k}_{\tau}\right\|^{2} + 2\mathbb{E}\left\| y^{k}_{\tau + 1} - y_{\tau , +}^{k}(z^{k}_{\tau + 1}) \right\|^{2}\nonumber \\
\overset{(i)}{\leq} &~  2\mathbb{E}\left\|y^{k}_{\tau + 1} - y^{k}_{\tau}\right\|^{2} 
 + 48\alpha_{y}^2 \beta^{2} L_{y}^{2}\mathbb{E}\left\|x^{k}_{\tau + 1} - z^{k}_{\tau}\right\|^{2}  + 12 \alpha_{y}^2L_{y}^{2}\mathbb{E} \left\|x^{k}_{\tau + 1} - x_{r}(y^{k}_{\tau},z^{k}_{\tau})\right\|^{2}  + 6\alpha_{y}^2 L_{y}^{2}\mathbb{E}\left\| x^{k}_{\tau+1} - x^{k}_{\tau}\right\|^{2} \nonumber\\
\quad & + \frac{6\alpha_{y}^{2}L_{y}^{2}}{M}\sum_{b=0}^{\tau-1}\mathbb{E}\left\| y^{k}_{b+1} - y^{k}_{b}\right\|^{2} + \frac{6\alpha_{y}^{2}L_{y}^{2}}{M}\sum_{b=0}^{\tau-1}\mathbb{E}\left\| x^{k}_{b+1} - x^{k}_{b}\right\|^{2}  {+6\alpha_{y}^2C_{\sigma,y}} \\
\overset{(ii)}{\leq} &~2\mathbb{E}\left\|y^{k}_{\tau + 1} - y^{k}_{\tau}\right\|^{2} 
 + \frac{2\beta}{125}\mathbb{E}\left\|x^{k}_{\tau + 1} - z^{k}_{\tau}\right\|^{2}  + \frac{3}{25}\mathbb{E} \left\|x^{k}_{\tau + 1} - x_{r}(y^{k}_{\tau},z^{k}_{\tau})\right\|^{2} + \frac{3}{50}\mathbb{E}\left\| x^{k}_{\tau+1} - x^{k}_{\tau}\right\|^{2} \nonumber\\
\quad &   + \frac{3}{50M}\sum_{b=0}^{\tau-1}\mathbb{E}\left\| y^{k}_{b+1} - y^{k}_{b}\right\|^{2} + \frac{3}{50M}\sum_{b=0}^{\tau-1}\mathbb{E}\left\| x^{k}_{b+1} - x^{k}_{b}\right\|^{2}{+6\alpha_{y}^2C_{\sigma,y}}\;,
\end{align*}
where $(i)$ follows from \eqref{eq:y_minus_y_2}, and $(ii)$ holds by $\alpha_{y} \leq \frac{1}{10L_{y}}$ and $\beta\leq \frac{1}{30}$. Rearranging the above gives the desired bound.
\end{proof}
{The following lemma provides auxiliary bounds that will be used in the later lemmas to establish the optimality of the algorithm’s output.}
\begin{lemma}
Under Assumptions \ref{assumption_for_smooth_case} and \ref{assumption_KL_for_smooth_case}, let 
$\{x_{\tau}^k, y_{\tau}^k, z_{\tau}^k\}$ be generated by Algorithm \ref{alg:smoothed_primal_dual_spider} with $T=M$ and other parameters satisfying  
\begin{align}
&r\geq \max\{2\rho, L_{y}+\rho\}, \label{eq:cond-r}\\
&\frac{24(L_{y}+1)}{(r-\rho)^2} \leq \alpha_{x} \leq  
\bar{\alpha}_x:=\min\Bigg\{
\frac{1}{12(r+L_{x} + 2L_{y})},\;
\frac{(r-\rho)^{2}}{24(r+L_{x})^2(L_y+1)},\;  
\frac{r - (\rho + 2 L_y)}{2 L_y (L_x + r)}
\Bigg\},\label{eq:cond-alpha-x}\\
&\alpha_{y} \leq \min\left\{\frac{1}{40L_y}, \; \frac{1}{4(2L_y+1)}\right\}, \label{eq:cond-alpha-y}\\
&\beta \leq  \begin{cases}
    \min\left\{\frac{1}{30}, \frac{1}{30r}, \frac{L_{y}}{20r\varpi}\right\}  & \text{ if }\theta \in \left[0, \frac{1}{2}\right],\\
     \min\left\{\frac{1}{30}, \frac{1}{30r}\right\} & \text{ if }\theta \in \left(\frac{1}{2},1\right]. 
\end{cases}\label{eq:cond-beta}
\end{align}
Additionally, 
denote $\chi_\theta :=
\begin{cases}
0, & \theta \in [0,\frac{1}{2}],\\
1, & \theta \in (\frac{1}{2},1]
\end{cases}$, $\Phi_r^0 = \Phi_{r}\left(x^{0}_{0},y^{0}_{0}, z^{0}_{0}\right)$ {along with $$C_\sigma = C_{\sigma,x}\left(2(1+L_y) \left( \frac{\alpha_x^2(r+L_{x})^2}{(r-\rho)^2} + \frac{1}{(r-\rho)^{2}} \right) + \frac{\alpha_x}{2}\right)+C_{\sigma,y}\left(16\alpha_y + 6\alpha_y^2\left(L_y+3\right)\right).$$} Then for any $k \geq 0$ and any $0 \leq \tau \leq T-1$, we have
\begin{align*}
\sum_{k=0}^{K-1}\sum_{\tau=0}^{T-1}\mathbb{E}\left\|x^{k}_{\tau+1} - x^{k}_{\tau}\right\|^{2} \leq 8\alpha_{x}\left((\Phi_0^0 - \underline{F}) +  \chi_{\theta}C_{\beta}KT{+C_{\sigma}KT}\right),
\end{align*}
\begin{align*}
\sum_{k=0}^{K-1}\sum_{\tau=0}^{T-1}\mathbb{E}\left\|y^{k}_{\tau+1} - y^{k}_{\tau}\right\|^{2} \leq 4\alpha_{y}\left((\Phi_0^0 - \underline{F}) +  \chi_{\theta}C_{\beta}KT{+C_{\sigma}KT}\right),
\end{align*}
\begin{align*}
\sum_{k=0}^{K-1}\sum_{\tau=0}^{T-1}\mathbb{E}\left\|x^{k}_{\tau+1} - z^{k}_{\tau}\right\|^{2} \leq \frac{2}{r\beta}\left((\Phi_0^0 - \underline{F}) +  \chi_{\theta}C_{\beta}KT{+C_{\sigma}KT}\right),
\end{align*}
\label{lemma:scenario1}
where $C_\beta$ is defined in Lemma~\ref{lemma:dual:final}.
\end{lemma}
\begin{proof} By substituting the bounds established in Lemmas \ref{lemma:dual:final} and \ref{lemma:dual:final:2} into \eqref{eq:Sufficient_Decrease_Init} and rearranging terms, we obtain
\begin{align}
&\mathbb{E}\left[\Phi_r(x^{k}_{\tau}, y^{k}_{\tau}, z^{k}_{\tau}) - \Phi_r(x^{k}_{\tau + 1}, y^{k}_{\tau + 1}, z^{k}_{\tau + 1})\right]\nonumber\\
\geq &~ \left(\frac{1}{2\alpha_{x}} - \frac{r+L_{x}}{2}-16\alpha_{y}L_{y}^{2}-540r\beta\alpha_{y}^{2}L_{y}^{2}-\frac{3L_{y}}{50}\right)\mathbb{E}\left\|x^{k}_{\tau+1} - x^{k}_{\tau}\right\|^{2} + \left(\frac{11}{16\alpha_{y}} - 7L_{y}\right)\mathbb{E}\left\|y^{k}_{\tau + 1} - y^{k}_{\tau} \right\|^2 \nonumber\\
\quad & + \left(\frac{(2-\beta)r\beta}{2} - 4r\beta^{2}-\frac{r\beta}{10} -4320r\beta^{3}\alpha_{y}^{2}L_{y}^{2} - \frac{2\beta L_{y}}{125}\right)\mathbb{E} \left\|x^{k}_{\tau + 1} -z^{k}_{\tau}\right\|^{2} \nonumber\\
\quad & - \left(1080r\beta\alpha_{y}^{2}L_{y}^{2} + 4\alpha_{y}L_{y}^{2}+\frac{3L_{y}}{25}\right)\mathbb{E} \left\|x^{k}_{\tau + 1} -x_{r}(y_{\tau}^{k},z_{\tau}^{k})\right\|^{2} \nonumber\\
\quad & - \left(\frac{16\alpha_{y}L_{y}^{2}+540r\beta\alpha_{y}^{2}L_{y}^{2}+\alpha_{x}L_{x}^{2}}{M} + \frac{3L_{y}}{50M}\right)\sum_{b=0}^{\tau-1}\mathbb{E}\left\| x^{k}_{b+1} - x^{k}_{b}\right\|^{2}
\nonumber\\
\quad & - \left(\frac{16\alpha_{y}L_{y}^{2}+540r\beta\alpha_{y}^{2}L_{y}^{2}+\alpha_{x}
{L_{y}^{2}}}{M} + \frac{3L_{y}}{50M}\right)\sum_{b=0}^{\tau-1}\mathbb{E}\left\| y^{k}_{b+1} - y^{k}_{b}\right\|^{2}-\frac{\alpha_xC_{\sigma,x}}{2} \nonumber\\
\quad & - \left(16\alpha_y+540r\beta\alpha_y^2+6\alpha_y^2L_y\right)C_{\sigma,y} - \chi_{\theta}C_{\beta}.\nonumber
\end{align}
Plugging the bound established in Lemma \ref{lemma:lemma_primal_error_bound} into the above inequality, noting $L_{y}\leq r$ and rearranging terms renders
\begin{align}\label{sufficient_decrease_init_0}
&\mathbb{E}\left[\Phi_r(x^{k}_{\tau}, y^{k}_{\tau}, z^{k}_{\tau}) - \Phi_r(x^{k}_{\tau + 1}, y^{k}_{\tau + 1}, z^{k}_{\tau + 1})\right]\\
\geq &~ \left(\frac{1}{2\alpha_x} - \underbrace{\Bigg(\begin{aligned}&\frac{r+L_x}{2}  +5940  r \beta \alpha_y^2 L_y^2  + 2700  r \beta \alpha_y^2 L_y^2 \eta^2 + 36  \alpha_y L_y^2  + 10  \alpha_y L_y^2 \eta^2 \\
&+ \frac{33}{50} L_y  + \frac{15}{50} L_y \eta^2\end{aligned}\Bigg)}_{\textcircled{1}}\right)\mathbb{E}\left\|x^{k}_{\tau+1} - x^{k}_{\tau}\right\|^{2} \nonumber\\
\quad & + \underbrace{\left(\frac{11}{16\alpha_{y}} - 7L_{y}\right)}_{\textcircled{2}}\mathbb{E}\left\|y^{k}_{\tau + 1} - y^{k}_{\tau} \right\|^2 + \underbrace{\left(\frac{(2-\beta)r\beta}{2} - 4r\beta^{2}-\frac{r\beta}{10} -4320r\beta^{3}\alpha_{y}^{2}L_{y}^{2} - \frac{2r\beta}{125}\right)}_{\textcircled{3}}\mathbb{E} \left\|x^{k}_{\tau + 1} -z^{k}_{\tau}\right\|^{2} \nonumber\\
\quad & - \underbrace{\Bigg(\begin{aligned}&\frac{16 \alpha_y L_y^2 + 540 r \beta \alpha_y^2 L_y^2 + \alpha_x L_x^2+5400 r \beta \alpha_x^2 \alpha_y^2 L_x^2 L_y^2 \eta^2 + 20 \alpha_x^2 \alpha_y L_x^2 L_y^2 \eta^2}{M} \\
&+ \frac{3 L_y}{50 M} 
+ \frac{3 \alpha_x^2 L_x^2 L_y \eta^2}{5 M}\end{aligned}\Bigg)}_{\textcircled{4}}\sum_{b=0}^{\tau-1}\mathbb{E}\left\| x^{k}_{b+1} - x^{k}_{b}\right\|^{2}
\nonumber\\
\quad & - \underbrace{\Bigg(\begin{aligned}&\frac{16 \alpha_y L_y^2 + 540 r \beta \alpha_y^2 L_y^2 + \alpha_x L_y^2+5400 r \beta \alpha_x^2 \alpha_y^2 L_y^4 \eta^2+20 \alpha_x^2 \alpha_y L_y^4 \eta^2}{M} \\
&+ \frac{3 L_y}{50 M} 
+ \frac{3 \alpha_x^2 L_y^3 \eta^2}{5 M}\end{aligned}\Bigg)}_{\textcircled{5}}\sum_{b=0}^{\tau-1}\mathbb{E}\left\| y^{k}_{b+1} - y^{k}_{b}\right\|^{2}\nonumber\\
 \quad &  
 {-\underbrace{\left(2700r\beta\alpha_{y}^{2}L_{y}^{2}\alpha_x^2\eta^2 + 10\alpha_{y}L_{y}^{2}\alpha_x^2\eta^2+\frac{3L_{y}\alpha_x^2\eta^2}{10}+\frac{\alpha_x}{2}\right)}_{\textcircled{6}}C_{\sigma,x} -\underbrace{\left(16\alpha_y+540r\beta\alpha_y^2+6\alpha_y^2L_y\right)  }_{\textcircled{7}}C_{\sigma,y}}    - \chi_{\theta}C_{\beta}. \nonumber
\end{align}
Below we bound the seven underbraced terms in \eqref{sufficient_decrease_init_0}. 
First, noting that $\alpha_{y} \leq \frac{1}{40L_y}$ and $\beta \leq \min\left\{\frac{1}{30}, \frac{1}{30r}\right\}$, we  have 
\begin{align}
\textcircled{1}
\leq &~ \frac{r+L_x}{2} +\frac{5940r}{30\cdot1600}+ \frac{2700\eta^{2}}{30\cdot1600} + \frac{36L_{y}}{40} + \frac{10L_{y}\eta^{2}}{40} + \frac{33L_{y}}{50}+\frac{15\eta^{2}L_{y}}{50}\nonumber\\
\leq &~ \frac{r+L_x}{2} + \frac{r}{2}+\eta^{2}+L_{y} +\frac{\eta^2L_{y}}{2}+L_{y}+\frac{\eta^2L_{y}}{2}\nonumber\\
\leq &~ r + L_x + 2L_y + \eta^2(L_y+1).\label{eq:for:constants:1}
\end{align}
By Young's inequality, it holds
\begin{equation}\label{eq:bound-eta-2}\eta^{2} = \frac{\left(\alpha_{x}(L_{x}+r)+1\right)^{2}}{\alpha_{x}^{2}\left(r - \rho\right)^{2}} \leq \frac{2(r+L_{x})^2}{(r-\rho)^2} + \frac{2}{\alpha_{x}^{2}(r-\rho)^{2}}.
\end{equation}
Plugging this into \eqref{eq:for:constants:1} and from the conditions of $\alpha_x$ gives
\begin{align}
  \textcircled{1}
  \leq r + L_x + 2L_y + \eta^2(L_y+1) \le &~ r + L_x + 2L_y + \frac{2(r+L_{x})^2(L_y+1)}{(r-\rho)^2} + \frac{2(L_y+1)}{\alpha_{x}^{2}(r-\rho)^{2}} \le \frac{1}{4\alpha_x}.\label{eq:for:constants:2}
\end{align}
For \textcircled{2} in \eqref{sufficient_decrease_init_0}, we use 
$\alpha_{y} \leq \frac{1}{40L_y}$ to  have 
\begin{align}
    \textcircled{2} \geq \frac{11}{16\alpha_y} - \frac{7}{40 \alpha_{y}} \geq \frac{1}{2\alpha_y}\label{second_constant_final}.
\end{align}
For \textcircled{3} in \eqref{sufficient_decrease_init_0}, by $\beta \leq \frac{1}{30}$ and $\alpha_{y} \leq \frac{1}{40L_y}$, it follows
\begin{align}
    \textcircled{3} \geq \frac{(2-\frac{1}{30})r\beta}{2} - \frac{4r\beta}{30} - \frac{r\beta}{10} - \frac{4320r\beta}{900\cdot1600} - \frac{2r\beta}{125} \geq \frac{r\beta}{2}\label{Third_constant_final}.
\end{align}
For \textcircled{4} in \eqref{sufficient_decrease_init_0}, by $\alpha_{x} \leq \frac{1}{2L_x}$, $\alpha_y \leq \frac{1}{40L_y}$, and   $\beta \leq \min\left\{\frac{1}{30}, \frac{1}{30r}\right\}$, 
it holds
\begin{align}
\textcircled{4}
\leq &~ \frac{1}{M}\left(\frac{16L_y}{40}+\frac{540r}{30\cdot1600}+\frac{L_x}{2}+\frac{5400\eta^2}{30\cdot4\cdot1600}+\frac{20\eta^2L_y}{4\cdot40}+\frac{3L_y}{50}+\frac{3\eta^2L_y}{20}\right)\nonumber\\
\leq &~ \frac{1}{2M}\left(r + L_x + 2L_y + \eta^2(L_y+1)\right)\nonumber\\
\leq &~ \frac{1}{8\alpha_xM},\label{Fourth_constant_Final}
\end{align}
where the last inequality follows from \eqref{eq:for:constants:2}.
For  \textcircled{5} in \eqref{sufficient_decrease_init_0}, it follows from $\alpha_x \leq \frac{1}{2L_y}$, $\alpha_y \leq \frac{1}{40L_y}$, and $\beta \leq \frac{1}{30r}$ that
\begin{align}
\textcircled{5}
\leq &~ \frac{1}{M}\left(\frac{16L_y}{40}+\frac{540}{30\cdot1600}+\frac{L_y}{2}+\frac{5400\alpha_x^2L_{y}^2\eta^2}{30\cdot1600}+\frac{20\eta^2L_y^3\alpha_x^2}{40}+\frac{3L_y}{50}+\frac{3\eta^2L_y^3\alpha_x^2}{20}\right).\label{Fifth_constant_Final_init}
\end{align}
Moreover, from \eqref{eq:bound-eta-2} 
and 
$\alpha_x \leq \frac{r - (\rho + 2 L_y)}{2 L_y (L_x + r)}$, we have
\begin{align}
\alpha_{x}^2L_{y}^2\eta^2
\leq &~ \frac{2\alpha_x^2(r+L_{x})^2L_{y}^2}{(r-\rho)^2} + \frac{2L_{y}^2}{(r-\rho)^{2}} \nonumber\\
\leq &~ \frac{2(r+L_{x})^2L_{y}^2}{(r-\rho)^2}\left(\frac{r - (\rho + 2 L_y)}{2 L_y (L_x + r)}\right)^{2} + \frac{2L_{y}^2}{(r-\rho)^{2}}\nonumber\\
= & ~ \frac{1}{2} +\frac{2L_{y}}{(r-\rho)} \left(\frac{2L_{y}}{(r-\rho)} - 1\right)\nonumber\\
{\leq} &~ \frac{1}{2} - \frac{L_{y}}{r-\rho}\nonumber\\
\leq &~ \frac{1}{2}\label{eq:alpha_L_y_eta}, 
\end{align}
where the third inequality follows from $r \geq \rho + 4L_{y}$. Plugging this bound into \eqref{Fifth_constant_Final_init} yields
\begin{align}
\textcircled{5}
\leq &~ \frac{1}{M}\left(\frac{16L_y}{40}+\frac{540}{30\cdot1600}+\frac{L_y}{2}+\frac{5400L_y}{60\cdot1600}+\frac{20L_y}{80}+\frac{3L_y}{50}+\frac{3L_y}{40}\right) \leq \frac{1}{M}\left(2L_{y}+1\right) \leq \frac{1}{4\alpha_y}.\label{Fifth_constant_Final}
\end{align}
where in the last inequality, we have used $\alpha_y \leq \frac{1}{4(2L_y+1)}$. {
 {For  \textcircled{6} in \eqref{sufficient_decrease_init_0}, by 
 $\alpha_y \le \frac{1}{40 L_y}$, $\beta\leq \frac{1}{30r}$, and \eqref{eq:bound-eta-2}, we get
\begin{align}
\textcircled{6} \leq \frac{90}{1600}\alpha_x^2 \eta^2 + \frac{L_y \alpha_x^2 \eta^2}{4}  + \frac{3L_y \alpha_x^2 \eta^2}{10} + \frac{\alpha_x}{2} \le 2(1+L_y) \left( \frac{\alpha_x^2(r+L_{x})^2}{(r-\rho)^2} + \frac{1}{(r-\rho)^{2}} \right) + \frac{\alpha_x}{2} 
\label{Sixth_constant_Final}.
\end{align}
}
Similarly, from $\beta\leq\frac{1}{30r}$, we have
\begin{align}
\textcircled{7} \leq 16\alpha_y+18\alpha_y^2+6\alpha_y^2L_y = 16\alpha_y + 6\alpha_y^2\left(L_y+3\right)\label{seventh_constant}.
\end{align}}

Now substituting \eqref{eq:for:constants:2}, \eqref{second_constant_final}, \eqref{Third_constant_final}, \eqref{Fourth_constant_Final}, \eqref{Fifth_constant_Final}, \eqref{Sixth_constant_Final} and \eqref{seventh_constant} into \eqref{sufficient_decrease_init_0}, we get
\begin{align}
&\mathbb{E}\left[\Phi_r(x^{k}_{\tau}, y^{k}_{\tau}, z^{k}_{\tau}) - \Phi_r(x^{k}_{\tau + 1}, y^{k}_{\tau + 1}, z^{k}_{\tau + 1})\right]\nonumber\\
\geq &~ \frac{1}{4\alpha_{x}}\mathbb{E}\left\|x^{k}_{\tau+1} - x^{k}_{\tau}\right\|^{2} + \frac{1}{2\alpha_{y}}\mathbb{E}\left\|y^{k}_{\tau+1} - y^{k}_{\tau}\right\|^{2}+\frac{r\beta}{2}\mathbb{E}\left\|x^{k}_{\tau+1} - z^{k}_{\tau}\right\|^{2}-\frac{1}{8\alpha_{x}M}\sum_{b=0}^{\tau-1}\mathbb{E}\left\| x^{k}_{b+1} - x^{k}_{b}\right\|^{2}\nonumber\\
\quad & -\frac{1}{4\alpha_{y}M}\sum_{b=0}^{\tau-1}\mathbb{E}\left\| y^{k}_{b+1} - y^{k}_{b}\right\|^{2}- \chi_{\theta}C_{\beta}{-C_{\sigma}}. \label{sufficient_decrease_init_0_final}
\end{align}


Before proceeding further we first note that 
\begin{align}\sum_{k=0}^{K-1}\left(\sum_{\tau=0}^{T-1}\sum_{b=0}^{\tau-1}\mathbb{E}\left\| x^{k}_{b+1} - x^{k}_{b}\right\|^{2}\right) \leq T\sum_{k=0}^{K-1}\sum_{\tau=0}^{T-1}\mathbb{E}\left\| x^{k}_{\tau+1} - x^{k}_{\tau}\right\|^{2},\label{intermediate_0}\\
\sum_{k=0}^{K-1}\left(\sum_{\tau=0}^{T-1}\sum_{b=0}^{\tau-1}\mathbb{E}\left\| y^{k}_{b+1} - y^{k}_{b}\right\|^{2}\right) \leq T\sum_{k=0}^{K-1}\sum_{\tau=0}^{T-1}\mathbb{E}\left\| y^{k}_{\tau+1} - y^{k}_{\tau}\right\|^{2}\label{intermediate_1_1st}.\end{align}
Additionally, we have
\begin{align}
\sum_{k=0}^{K-1}\sum_{\tau=0}^{T-1}\mathbb{E}\left[\Phi_r(x^{k}_{\tau}, y^{k}_{\tau}, z^{k}_{\tau}) - \Phi_r(x^{k}_{\tau + 1}, y^{k}_{\tau + 1}, z^{k}_{\tau + 1})\right] = \mathbb{E}\left[\Phi_r(x^{0}_{0}, y^{0}_{0}, z^{0}_{0}) - \Phi_r(x^{K}_{0}, y^{K}_{0}, z^{K}_{0})\right]. 
\label{intermediate_1_4th}
\end{align}
{We note that from the definitions of $F_{r}(x,y,z)$, $d_{r}(y,z)$ and $p_{r}(z)$ in Table \ref{tab:mylabel}, we have for all $x \in \cX$, $y \in \cY$ and $z \in \mathbb{R}^{d_{x}}$,
\begin{align}
\Phi_r(x, y, z) =&~ \left(F_{r}(x,y,z) - d_{r}(y,z)\right) + \left(p_{r}(z) - d_{r}(y,z)\right) + p_{r}(z)\nonumber\\
\geq&~p_{r}(z)\nonumber\\
{=}&~\max_{y \in \mathcal{Y}}\min_{x \in \mathcal{X}}F_{r}(x,y,z)\nonumber\\
\ge&~\max_{y \in \mathcal{Y}}\min_{x \in \mathcal{X}}F(x,y)\nonumber\\
{\geq}&~ \underline{F},
\label{intermediate_1_5th}
\end{align}
where 
{the second inequality follows from 
$F_{r}(x,y,z)\ge F(x,y)$, and 
the last inequality} holds from Assumption \ref{assumption_for_smooth_case}~$[vi.]$. Plugging \eqref{intermediate_1_5th} into \eqref{intermediate_1_4th} yields
\begin{align}
\sum_{k=0}^{K-1}\sum_{\tau=0}^{T-1}\mathbb{E}\left[\Phi_r(x^{k}_{\tau}, y^{k}_{\tau}, z^{k}_{\tau}) - \Phi_r(x^{k}_{\tau + 1}, y^{k}_{\tau + 1}, z^{k}_{\tau + 1})\right] \leq \Phi_0^0 - \underline{F}.
\label{intermediate_1_2nd}
\end{align}}
Taking sum on both sides of \eqref{sufficient_decrease_init_0_final}, plugging \eqref{intermediate_0}, \eqref{intermediate_1_1st} and \eqref{intermediate_1_2nd}, utilizing $T = M$, and rearranging terms renders
\begin{align*}
&\frac{1}{8\alpha_{x}}\sum_{k=0}^{K-1}\sum_{\tau=0}^{T-1}\mathbb{E}\left\|x^{k}_{\tau+1} - x^{k}_{\tau}\right\|^{2} + \frac{1}{4\alpha_{y}}\sum_{k=0}^{K-1}\sum_{\tau=0}^{T-1}\mathbb{E}\left\|y^{k}_{\tau+1} - y^{k}_{\tau}\right\|^{2}+\frac{r\beta}{2}\sum_{k=0}^{K-1}\sum_{\tau=0}^{T-1}\mathbb{E}\left\|x^{k}_{\tau+1} - z^{k}_{\tau}\right\|^{2}\\
\leq &~ (\Phi_0^0 - \underline{F}) +  \chi_{\theta}C_{\beta}KT{+C_{\sigma}KT}.
\end{align*}
The desired results follow immediately from the above, thus completing the proof.
\end{proof}

\begin{remark}
In \eqref{eq:cond-alpha-x}, both lower and upper bounds are imposed on $\alpha_x$. We discuss how to choose $r$ such that the interval for $\alpha_x$ is not empty, namely, we require
\[\underbrace{\frac{24(L_{y}+1)}{(r-\rho)^2}}_{\textcircled{i}} \leq  \min\left\{\underbrace{\frac{1}{12(r+L_{x} + 2L_{y})}}_{\textcircled{ii}},\;\underbrace{\frac{(r-\rho)^{2}}{24(r+L_{x})^2(L_y+1)}}_{\textcircled{iii}},\; \underbrace{\frac{r - (\rho + 2 L_y)}{2 L_y (L_x + r)}}_{\textcircled{iv}}\right\}.\]  
First, we have
\begin{align*}
\textcircled{ii} \geq \textcircled{i} \iff & (r-\rho)^2 \geq 288\left(L_{y}+1\right)\left(r+L_{x}+2L_y\right) \\
\iff & \big(r - (\rho + 144(L_y+1))\big)^2 \;\ge\; 288(L_y+1)\big[\rho + 72(L_y+1) + L_x + 2L_y\big],
\end{align*}
which is implied by
\[r \ge \rho + 144(L_y+1) + \sqrt{288(L_y+1)\big(\rho + 72(L_y+1) + L_x + 2L_y\big) \big)}.\]
Second, it holds
\begin{align*}\textcircled{iii} \geq \textcircled{i} \iff (r - \rho)^4 \ge 576 (L_y + 1)^2 (r + L_x)^2 \iff (r - \rho)^2 \ge 24 (L_y + 1) (r + L_x),
\end{align*}
which is indicated by 
\[r \ge \rho + 12(L_y + 1) + \sqrt{24(L_y + 1)\big(\rho + L_x + 6(L_y + 1)\big)}.\]
Third, assuming $r\geq \rho + 2L_{y} + 1$, we have
$\textcircled{iv} \geq \frac{1}{2 L_y (L_x + r)}.$ 
Hence, it suffices to require
\[\frac{1}{2 L_y (L_x + r)} \geq \textcircled{i} \iff \big(r - (\rho + 24 L_y (L_y + 1))\big)^2 \geq
48 L_y (L_y + 1) \, \big(\rho + L_x + 12 L_y (L_y + 1)\big),\]
which is implied by
\[r \ge \rho + 24 L_y (L_y + 1) + \sqrt{48 L_y (L_y + 1) \, \big(\rho + L_x + 12 L_y (L_y + 1)\big)}.\]
Therefore, a valid value of $\alpha_x$ exists if
\begin{align*}
r \ge \max\Bigg\{ & \rho + 144(L_y+1) + \sqrt{288(L_y+1) }\;\sqrt{ \rho +L_x +74L_y+ 72 },\\
& ~~\rho + 24 L_y (L_y + 1) + \sqrt{48 L_y (L_y + 1)}\;\sqrt{\rho + L_x + 12 L_y (L_y + 1)}\Bigg\}.
\end{align*}
By $\sqrt{a+b}\le \sqrt{a} + \sqrt{b}$ and $2\sqrt{ab} \le a+b$ for any nonnegative numbers $a$ and $b$, we can further relax the bound on $r$ and take it as 
\begin{align}\label{eq:bound-on-r}
r = \max\Bigg\{ & 2\rho + 325(L_y+1) + 12\sqrt{L_x}\sqrt{2(L_y+1) },\  2\rho + 54 L_y (L_y + 1) + 4\sqrt{L_x}\sqrt{3 L_y (L_y + 1)}\Bigg\}.
\end{align}
\end{remark}
{Building on these auxiliary bounds, the next two lemmas establish results that are crucial for proving the optimality of the output of Algorithm \ref{alg:smoothed_primal_dual_spider}.}
\begin{lemma}\label{lemma_sc}
 Under Assumption \ref{assumption_for_smooth_case}, let 
$\{x_{\tau}^k, y_{\tau}^k, z_{\tau}^k\}$ be generated by Algorithm \ref{alg:smoothed_primal_dual_spider} with $T = M$. Then for any $k\geq 0$, $0 \leq \tau \leq T-1$, it holds 
\begin{align*}
&\mathbb{E}\left[\mathrm{dist}\left(0, \nabla_x F(\tilde x, \tilde y) + \cN_\cX(\tilde x)\right)^{2}\right]+ \mathbb{E}\left[\mathrm{dist}\left(0, -\nabla_{y}F\left(\tilde x, \tilde y\right) + \cN_\cY(\tilde y)\right)^{2} \right] \\
\leq &~ \frac{1}{KT}\left(\frac{4}{\alpha_{x}^{2}}+16L_{x}^{2}+8r^{2}+6L_{y}^{2}\right)\sum_{k=0}^{K-1}\sum_{\tau=0}^{T-1}\mathbb{E}\left\|x^{k}_{\tau+1} - x_{\tau}^{k}\right\|^{2} + \frac{1}{KT}\left(\frac{3}{\alpha_{y}^{2}}+22L_{y}^{2}\right)\sum_{k=0}^{K-1}\sum_{\tau=0}^{T-1}\mathbb{E}\left\|y^{k}_{\tau+1} - y_{\tau}^{k}\right\|^{2}\nonumber\\
\quad & + \frac{8r^{2}}{KT}\sum_{k=0}^{K-1}\sum_{\tau=0}^{T-1}\mathbb{E}\left\|x_{\tau+1}^{k} - z_{\tau}^{k}\right\|^{2}+{\left(4C_{\sigma,x}+3C_{\sigma,y}\right)}\;.
\end{align*}
\end{lemma}
\begin{proof} 
We first bound $\mathbb{E}\left[\mathrm{dist}\left(0, \nabla_{x}F\left(x^{k}_{\tau + 1},y^{k}_{\tau + 1}\right) + \cN_\cX(x^{k}_{\tau + 1})\right)^{2} \right]$. By the update for $x_{\tau+1}^{k}$ in Algorithm \ref{alg:smoothed_primal_dual_spider}, it holds 
\begin{align*}
0 &\in \left(x^{k}_{\tau + 1}- x_{\tau}^{k}\right) + \alpha_{x} \Bigg(G_{x,\tau}^{k} + r(x_{\tau}^{k}-z_{\tau}^{k})\Bigg) + \cN_\cX\left(x^{k}_{\tau+1}\right)\\
&\implies \frac{1}{\alpha_{x}}\left(x^{k}_{\tau} - x^{k}_{\tau + 1}\right) + r(z_{\tau}^{k} - x_{\tau}^{k}) - G_{x,\tau}^{k} + \nabla_{x}F\left(x_{\tau+1}^{k},y_{\tau + 1}^{k}\right) \in  \nabla_{x}F\left(x_{\tau+1}^{k},y_{\tau + 1}^{k}\right) + \cN_\cX\left(x^{k}_{\tau + 1}\right)
\end{align*}
Hence, we use Young's inequality and Lemma~\ref{lem:err-grad-est} to have 
\begin{align}
&\mathbb{E}\left[\mathrm{dist}\left(0, \nabla_{x}F\left(x^{k}_{\tau + 1},y^{k}_{\tau + 1}\right) + \cN_\cX(x^{k}_{\tau + 1})\right)^{2} \right]\nonumber\\
\leq &~ \mathbb{E}\left\| \frac{1}{\alpha_{x}}\left(x^{k}_{\tau} - x^{k}_{\tau + 1}\right) + r(z_{\tau}^{k} - x_{\tau}^{k}) - G_{x,\tau}^{k} + \nabla_{x}F\left(x_{\tau+1}^{k},y_{\tau + 1}^{k}\right) \right\|^{2}  \nonumber\\
{\leq} &~\frac{4}{\alpha_{x}^{2}}\mathbb{E}\left\|x^{k}_{\tau+1} - x_{\tau}^{k}\right\|^{2} + 4r^{2}\mathbb{E}\left\|x_{\tau}^{k} - z_{\tau}^{k}\right\|^{2} + 4\mathbb{E}\left\|G_{x,\tau}^{k} - \nabla_{x}F\left(x_{\tau}^{k},y_{\tau}^{k}\right)  \right\|^{2}  \nonumber\\
\quad & + 4\mathbb{E}\left\|\nabla_{x}F\left(x_{\tau}^{k},y_{\tau}^{k}\right)   - \nabla_{x}F\left(x_{\tau+1}^{k},y_{\tau+1}^{k}\right)  \right\|^{2}\nonumber\\
{\leq} &~\left(\frac{4}{\alpha_{x}^{2}}+8L_{x}^{2}+8r^{2}\right)\mathbb{E}\left\|x^{k}_{\tau+1} - x_{\tau}^{k}\right\|^{2} + 8r^{2}\mathbb{E}\left\|x_{\tau+1}^{k} - z_{\tau}^{k}\right\|^{2}+ \frac{8L_{x}^{2}}{M}\sum_{b=0}^{\tau-1}\mathbb{E}\left\| x^{k}_{b+1} - x^{k}_{b}\right\|^{2}\nonumber\\
\quad & + \frac{8L_{y}^{2}}{M}\sum_{b=0}^{\tau-1}\mathbb{E}\left\| y^{k}_{b+1} - y^{k}_{b}\right\|^{2} + 8L_{y}^{2}\mathbb{E}\left\|y^{k}_{\tau+1} - y_{\tau}^{k}\right\|^{2}{+4C_{\sigma,x}}\;.
\label{eq:main_bound-supp}
\end{align}

Second, we bound 
$\mathbb{E}\left[\mathrm{dist}\left(0, -\nabla_{y}F\left(x^{k}_{\tau + 1},y^{k}_{\tau + 1}\right) + \cN_\cY(y^{k}_{\tau + 1})\right)^{2} \right]$.
By the update of $y_{\tau+1}^{k}$ in Algorithm \ref{alg:smoothed_primal_dual_spider}, 
we have
\begin{align*}
0 &\in \left(y^{k}_{\tau + 1}- y_{\tau}^{k}\right) - \alpha_{y} G_{y,\tau}^{k} + \cN_\cY\left(y^{k}_{\tau+1}\right)\\
\implies &~ \frac{1}{\alpha_{y}}\left(y^{k}_{\tau} - y^{k}_{\tau + 1}\right) + G_{y,\tau}^{k} - \nabla_{y}F\left(x_{\tau+1}^{k},y_{\tau + 1}^{k}\right) \in -  \nabla_{y}F\left(x_{\tau+1}^{k},y_{\tau + 1}^{k}\right) + \cN_\cY\left(y^{k}_{\tau + 1}\right)
\end{align*}
Hence, by Young's inequality, the smoothness of $F$, and Lemma~\ref{lem:err-grad-est}, it follows 
\begin{align}
&~\mathbb{E}\left[\mathrm{dist}\left(0, -\nabla_{y}F\left(x^{k}_{\tau + 1},y^{k}_{\tau + 1}\right) + \cN_\cY(y^{k}_{\tau + 1})\right)^{2} \right]\nonumber\\
 \leq &~ \mathbb{E}\left\| \frac{1}{\alpha_{y}}\left(y^{k}_{\tau} - y^{k}_{\tau + 1}\right) + G_{y,\tau}^{k} - \nabla_{y}F\left(x_{\tau+1}^{k},y_{\tau + 1}^{k}\right) \right\|^{2}  \nonumber\\
{\leq} &~\frac{3}{\alpha_{y}^{2}}\mathbb{E}\left\|y^{k}_{\tau+1} - y_{\tau}^{k}\right\|^{2} + 3\mathbb{E}\left\|G_{y,\tau}^{k} - \nabla_{y}F\left(x_{\tau}^{k},y_{\tau}^{k}\right)  \right\|^{2}  + 3\mathbb{E}\left\|\nabla_{y}F\left(x_{\tau}^{k},y_{\tau}^{k}\right)   - \nabla_{y}F\left(x_{\tau+1}^{k},y_{\tau+1}^{k}\right)  \right\|^{2}\nonumber\\
{\leq} &~\left(\frac{3}{\alpha_{y}^{2}}+3L_{y}^{2}\right)\mathbb{E}\left\|y^{k}_{\tau+1} - y_{\tau}^{k}\right\|^{2} + \frac{3L_{y}^{2}}{M}\sum_{b=0}^{\tau-1}\mathbb{E}\left\| y^{k}_{b+1} - y^{k}_{b}\right\|^{2}+ \frac{3L_{y}^{2}}{M}\sum_{b=0}^{\tau-1}\mathbb{E}\left\| x^{k}_{b+1} - x^{k}_{b}\right\|^{2}+ 3L_{y}^{2}\mathbb{E}\left\|x^{k}_{\tau+1} -  x^{k}_{\tau} \right\|^{2}\nonumber\\
\quad & {+3C_{\sigma,x}} .
\label{stationarity_second_condition-supp}
\end{align}
Adding \eqref{eq:main_bound-supp} and \eqref{stationarity_second_condition-supp} we get
\begin{align}\label{supp-9}
&\mathbb{E}\left[\mathrm{dist}\left(0, \nabla_{x}F\left(x^{k}_{\tau + 1},y^{k}_{\tau + 1}\right) + \cN_\cX(x^{k}_{\tau + 1})\right)^{2} \right] + \mathbb{E}\left[\mathrm{dist}\left(0, \nabla_{y}F\left(x^{k}_{\tau + 1},y^{k}_{\tau + 1}\right) + \cN_\cX(y^{k}_{\tau + 1})\right)^{2} \right]\nonumber\\ 
\leq &~ \left(\frac{4}{\alpha_{x}^{2}}+8L_{x}^{2}+8r^{2}+3L_{y}^{2}\right)\mathbb{E}\left\|x^{k}_{\tau+1} - x_{\tau}^{k}\right\|^{2} + \left(\frac{3}{\alpha_{y}^{2}}+11L_{y}^{2}\right)\mathbb{E}\left\|y^{k}_{\tau+1} - y_{\tau}^{k}\right\|^{2} + 8r^{2}\mathbb{E}\left\|x_{\tau+1}^{k} - z_{\tau}^{k}\right\|^{2}\nonumber\\
\quad & + \frac{8L_{x}^{2}+3L_{y}^{2}}{M}\sum_{b=0}^{\tau-1}\mathbb{E}\left\| x^{k}_{b+1} - x^{k}_{b}\right\|^{2}+ \frac{11L_{y}^{2}}{M}\sum_{b=0}^{\tau-1}\mathbb{E}\left\| y^{k}_{b+1} - y^{k}_{b}\right\|^{2} +{4C_{\sigma,x}+3C_{\sigma,y}}.
\end{align}
Taking sum on both sides of \eqref{supp-9}, plugging \eqref{intermediate_0}, \eqref{intermediate_1_1st} and utilizing $T = M$, we get
\begin{align}
&\sum_{k=0}^{K-1}\sum_{\tau=0}^{T-1}\mathbb{E}\left[\mathrm{dist}\left(0, \nabla_{x}F\left(x^{k}_{\tau + 1},y^{k}_{\tau + 1}\right) + \cN_\cX(x^{k}_{\tau + 1})\right)^{2} \right] + \mathbb{E}\left[\mathrm{dist}\left(0, -\nabla_{y}F\left(x^{k}_{\tau + 1},y^{k}_{\tau + 1}\right) + \cN_\cX(y^{k}_{\tau + 1})\right)^{2} \right]\nonumber\\ 
\leq &~ \left(\frac{4}{\alpha_{x}^{2}}+16L_{x}^{2}+8r^{2}+6L_{y}^{2}\right)\sum_{k=0}^{K-1}\sum_{\tau=0}^{T-1}\mathbb{E}\left\|x^{k}_{\tau+1} - x_{\tau}^{k}\right\|^{2} + \left(\frac{3}{\alpha_{y}^{2}}+22L_{y}^{2}\right)\sum_{k=0}^{K-1}\sum_{\tau=0}^{T-1}\mathbb{E}\left\|y^{k}_{\tau+1} - y_{\tau}^{k}\right\|^{2}\nonumber\\
\quad & + 8r^{2}\sum_{k=0}^{K-1}\sum_{\tau=0}^{T-1}\mathbb{E}\left\|x_{\tau+1}^{k} - z_{\tau}^{k}\right\|^{2}+{\left(4C_{\sigma,x}+3C_{\sigma,y}\right)KT}\nonumber. 
\end{align}
Dividing both sides of the above by $KT$, and using the definition of the uniform random output $\left(\tilde{x}, \tilde{y}\right)$, yields the desired result.
\end{proof}

\begin{lemma}\label{dz_for_smooth_case}
Under Assumption \ref{assumption_for_smooth_case}, let 
$\{x_{\tau}^k, y_{\tau}^k, z_{\tau}^k\}$ be generated by Algorithm \ref{alg:smoothed_primal_dual_spider} with $T = M$, 
{and $r$ and $\alpha_x$ satisfying the conditions in \eqref{eq:cond-r} and \eqref{eq:cond-alpha-x}.} Then for any $k\geq 0$ and $0 \leq \tau \leq T-1$, it holds
\begin{align*}
&~\mathbb{E}\left\|\nabla_{z}d_{r}\left(\tilde{y}, \tilde{x}\right)\right\|^{2}\\
\leq &~ \frac{1}{KT}\left(
15r^2 + \frac{97}{\alpha_x^2}\right)\sum_{k=0}^{K-1}\sum_{\tau=0}^{T-1}\mathbb{E}\left\|x^{k}_{\tau+1} - x^{k}_{\tau}  \right\|^{2}  + \frac{35r^2}{KT}\sum_{k=0}^{K-1}\sum_{\tau=0}^{T-1}\mathbb{E} \|y^{k}_{\tau} - y^{k}_{\tau + 1}\|^{2} \\
\quad & + \frac{12r^{2}}{KT}\sum_{k=0}^{K-1}\sum_{\tau=0}^{T-1}\mathbb{E} \| x^{k}_{\tau + 1} - z^{k}_{\tau}  \|^{2} {+61 C_{\sigma,x}.} 
\end{align*}
\end{lemma}
\begin{proof} 
By Danskin's Theorem, it holds
{$\nabla_{z}d_{r}(y^{k}_{\tau + 1},x^{k}_{\tau + 1}) = r\left( x^{k}_{\tau + 1} - x_{r}\left(y^{k}_{\tau + 1},x^{k}_{\tau + 1}\right)\right)$.}
Hence,
\begin{align}
&\mathbb{E}\left\|\nabla_{z}d_{r}(y^{k}_{\tau + 1},x^{k}_{\tau + 1})\right\|^{2} \nonumber\\ 
= &~ r^{2}\mathbb{E}\left\| x^{k}_{\tau + 1} - x_{r}\left(y^{k}_{\tau + 1},x^{k}_{\tau + 1}\right) \right\|^{2}\nonumber \\
\overset{(i)}{\leq} &~ 3r^{2}\left( \mathbb{E}\left\| x^{k}_{\tau + 1} - x_{r}(y^{k}_{\tau},z^{k}_{\tau}) \right\|^{2} + \mathbb{E}\left\| x_{r}(y^{k}_{\tau},z^{k}_{\tau}) - x_{r}(y^{k}_{\tau + 1},z^{k}_{\tau}) \right\|^{2} + \mathbb{E}\left\| x_{r}(y^{k}_{\tau + 1},z^{k}_{\tau}) - x_{r}(y^{k}_{\tau + 1},x^{k}_{\tau + 1}) \right\|^{2}   \right)\nonumber \\
\overset{(ii)}{\leq} &~  3r^{2}\mathbb{E}\left\| x^{k}_{\tau + 1} - x_{r}(y^{k}_{\tau},z^{k}_{\tau}) \right\|^{2} + 3r^{2}\sigma_{2}^{2}\mathbb{E} \|y^{k}_{\tau} - y^{k}_{\tau + 1}\|^{2} + 3r^{2}\sigma_{1}^{2}\mathbb{E} \| x^{k}_{\tau + 1} - z^{k}_{\tau}  \|^{2}\nonumber \\ 
\overset{(iii)}{\leq} &~ \left(\frac{15r^2\eta^{2}}{2}+15r^2\right)\mathbb{E}\left\|x^{k}_{\tau+1} - x^{k}_{\tau}  \right\|^{2} + 27r^{2}\mathbb{E} \|y^{k}_{\tau} - y^{k}_{\tau + 1}\|^{2} + 12r^{2}\mathbb{E} \| x^{k}_{\tau + 1} - z^{k}_{\tau}  \|^{2} \nonumber\\
\quad & + \frac{15r^2\eta^{2}\alpha_{x}^{2}L_{x}^{2}}{M}\sum_{b=0}^{\tau-1}\mathbb{E}\left\| x^{k}_{b+1} - x^{k}_{b}\right\|^{2} +\frac{15r^2\eta^{2}\alpha_{x}^{2}L_{y}^{2}}{M}\sum_{b=0}^{\tau-1}\mathbb{E}\left\| y^{k}_{b+1} - y^{k}_{b}\right\|^{2}{+\frac{15\alpha_x^2r^2\eta^2C_{\sigma,x}}{2}}, \label{eq:main_bound1}
\end{align}
where $(i)$ follows from Young's inequality, $(ii)$ is by Lemma~\ref{lem:lip-cont-x-r}, and $(iii)$ follows from Lemma \ref{lemma:lemma_primal_error_bound}, $\sigma_{1}\leq 2$ and  $\sigma_{2} \leq 3$ (since $r \geq \max\{2\rho, L_{y}+\rho\}  $). Taking sum on both sides of \eqref{eq:main_bound1},  plugging \eqref{intermediate_0} along with \eqref{intermediate_1_1st}, and utilizing $T=M$ renders
\begin{align}
&\sum_{k=0}^{K-1}\sum_{\tau=0}^{T-1}\mathbb{E}\left\|\nabla_{z}d_{r}(y^{k}_{\tau + 1},x^{k}_{\tau + 1})\right\|^{2} \nonumber\\ 
\leq &~ \underbrace{\left(\frac{15r^2\eta^{2}}{2}+15r^2+15r^2\eta^{2}\alpha_{x}^{2}L_{x}^{2}\right)}_{\textcircled{1}}\sum_{k=0}^{K-1}\sum_{\tau=0}^{T-1}\mathbb{E}\left\|x^{k}_{\tau+1} - x^{k}_{\tau}  \right\|^{2} + \underbrace{\left(27r^{2}+15r^2\eta^{2}\alpha_{x}^{2}L_{y}^{2}\right)}_{\textcircled{2}}\sum_{k=0}^{K-1}\sum_{\tau=0}^{T-1}\mathbb{E} \|y^{k}_{\tau} - y^{k}_{\tau + 1}\|^{2} \nonumber\\
\quad & + 12r^{2}\sum_{k=0}^{K-1}\sum_{\tau=0}^{T-1}\mathbb{E} \| x^{k}_{\tau + 1} - z^{k}_{\tau}  \|^{2}{+\underbrace{\frac{15\alpha_x^2r^2\eta^2}{2}}_{\textcircled{3}}C_{\sigma,x}KT}.
\label{final_dr_main}
\end{align}
Utilizing $\alpha_x \leq \frac{1}{2L_{x}}$ and \eqref{eq:bound-eta-2} , we bound \textcircled{1} in \eqref{final_dr_main} by
\[\textcircled{1} \leq \frac{24r^2(r+L_{x})^2}{(r-\rho)^2} + \frac{24r^2}{\alpha_{x}^{2}(r-\rho)^{2}} + 15r^2 {\le 96(r+L_{x})^2 + 15r^2 + \frac{96}{\alpha_x^2} \le 15r^2 + \frac{97}{\alpha_x^2}},\]
where the second inequality holds by $r\ge 2\rho$ and thus $\frac{r}{r-\rho} \le 2$, and the last inequality follows from $\frac{1}{\alpha_x} \ge 12(r+L_x)$.  
Also, from \eqref{eq:alpha_L_y_eta}, we have $\textcircled{2} \leq 35r^2$. {Furthermore, {utilizing \eqref{eq:bound-eta-2} and $r\ge 2\rho$, 
we get 
$$\textcircled{3}\leq \frac{15\alpha_x^2 r^2(r+L_{x})^2}{(r-\rho)^2} + \frac{15 r^2}{(r-\rho)^{2}} \le 60(\alpha_x^2(r+L_{x})^2 + 1)\le 61,$$
where the last inequality follows from $\alpha_x \le \frac{1}{12(r+L_x)}$.
}
}
Plugging these bounds into \eqref{final_dr_main}, dividing both sides of the above by $KT$, and using the definition of the uniform random output $\left(\tilde{x}, \tilde{y}\right)$, yields the desired result.
%
\end{proof}

We are now ready to present the main convergence result for the smooth case, {encompassing both the finite--sum and online settings}. 

\begin{proof}[Proof of Theorem \ref{main:theorem2}]
By substituting the bounds from Lemma \ref{lemma:scenario1} into Lemma \ref{lemma_sc}, and rearranging terms, we get
\begin{align}
&\mathbb{E}\left[\mathrm{dist}\left(0, \nabla_x F(\tilde x, \tilde y) + \cN_\cX(\tilde x)\right)^{2}\right]+ \mathbb{E}\left[\mathrm{dist}\left(0, -\nabla_{y}F\left(\tilde x, \tilde y\right) + \cN_\cY(\tilde y)\right)^{2} \right] \nonumber\\
\leq &~ \underbrace{\frac{8\alpha_x\Delta\Phi}{KT}
\left(\frac{4}{\alpha_x^{2}}+16L_x^{2}+8r^{2}+6L_y^{2}\right)}_{\textcircled{1}}
+ \underbrace{\frac{4\alpha_y\Delta\Phi}{KT}
\left(\frac{3}{\alpha_y^{2}}+22L_y^{2}\right)}_{\textcircled{2}} 
+ \underbrace{\frac{16r\Delta\Phi}{\beta KT}}_{\textcircled{3}}\nonumber\\
\quad &
+ \underbrace{8\alpha_x\chi_{\theta} C_{\beta}
\left(\frac{4}{\alpha_x^{2}}+16L_x^{2}+8r^{2}+6L_y^{2}\right)}_{\textcircled{4}} + \underbrace{4\alpha_y\chi_{\theta} C_{\beta}\left(\frac{3}{\alpha_y^{2}}+22L_y^{2}\right)}_{\textcircled{5}} 
+ \underbrace{\frac{16r\chi_{\theta} C_{\beta}}{\beta}}_{\textcircled{6}}\nonumber\\
\quad &
{+ \underbrace{8\alpha_x C_{\sigma}
\left(\frac{4}{\alpha_x^{2}}+16L_x^{2}+8r^{2}+6L_y^{2}\right)}_{\textcircled{7}} + \underbrace{4\alpha_yC_{\sigma}\left(\frac{3}{\alpha_y^{2}}+22L_y^{2}\right)}_{\textcircled{8}} 
+ \underbrace{\frac{16rC_{\sigma}}{\beta}}_{\textcircled{9}}+\underbrace{4C_{\sigma,x}+3C_{\sigma,y}}_{\textcircled{10}}}. \label{final_expression_for_optimality}
\end{align}
{We first discuss the finite--sum setting. In this case, since $C_{\sigma,x}$ and $C_{\sigma,y} = 0$, the last four terms vanish.} Hence, to have $(\tilde x, \tilde y)$ as an $\varepsilon$-GS point, it suffices to make each of the six remaining underbraced terms upper bounded by $\frac{\varepsilon^2}{6}$. When $\theta\in [0,\frac{1}{2}]$, $\chi_\theta=0$, and thus we only need to consider the first three terms. By the choice of $\alpha_x$ and $\alpha_y$, it follows that 
\textcircled{2} is dominated by \textcircled{1}, and to have both of \textcircled{1} and \textcircled{2} in the order of $\varepsilon^2$, it suffices to have $KT = \Theta\left(\frac{\Delta\Phi}{\alpha_x\varepsilon^{2}}\right)$. In addition, by the assumption of $\rho = \cO\big( \min\{L_x, L_y\} \big)$ and $\min\{L_x, L_y\} = \Omega(1)$, we have from \eqref{eq:bound-on-r} that $r=\Theta(L_y^2 + L_y\sqrt{L_x})$, and thus by \eqref{eq:cond-alpha-x}, it follows $\alpha_x = \Theta\left(\frac{1}{L_x+L_y^2}\right)$. Therefore, having \begin{equation}\label{eq:KT-1}
KT = \Theta\left(\frac{\Delta\Phi(L_x+L_y^2)}{\varepsilon^{2}}\right)
\end{equation}
will push both of \textcircled{1} and \textcircled{2} to $\cO(\varepsilon^2)$. Moreover, to have $\textcircled{3} = \cO(\varepsilon^2)$, we need $KT = \Theta\left(\frac{r\Delta\Phi}{\beta\varepsilon^{2}}\right)$. Notice that $\varpi = \Theta\left(\frac{L_y^2}{\mu^2(r-\rho)}\right)$. Thus $\frac{L_{y}}{2r\varpi}=\Theta\left(\frac{\mu^2}{L_y}\right)$, and the chosen $\beta= \Theta\left(\min\{\frac{1}{r}, \frac{\mu^2}{L_y}\}\right)$. Hence, having \begin{equation}\label{eq:KT-2-case}
KT=\Theta\left(\frac{\Delta\Phi\max\{r^2, \frac{rL_y}{\mu^2}\}}{\varepsilon^{2}}\right)=\Theta\left(\frac{L_y^2\Delta\Phi\max\{L_y^2 + L_x, \frac{L_y + \sqrt{L_x}}{\mu^2}\}}{\varepsilon^{2}}\right)
\end{equation}
will ensure $\textcircled{3} = \cO(\varepsilon^2)$. Together with \eqref{eq:KT-1}, 
this completes the proof for the case of $\theta\in [0,\frac{1}{2}] $\,{ in the finite--sum setting}.

When $\theta\in (\frac{1}{2},1]$, we have $\kappa = \Theta\left(\frac{L_y^{\frac{1}{\theta}}}{r \mu^{\frac{1}{\theta}}}\right)$, and thus 
\begin{align}
C_{\beta} = \left(\left(\frac{2\theta-1}{2\theta}\right)\left(\frac{20r\kappa}{\left(2\theta L_{y}\right)^{\frac{1}{2\theta}}}\right)^{\frac{2\theta}{2\theta-1}}\right)\beta^{\frac{2\theta}{2\theta-1}} = \Theta\left(\frac{L_y^\frac{1}{2\theta-1}}{\mu^\frac{2}{2\theta-1}}\beta^{\frac{2\theta}{2\theta-1}}\right).\label{C_beta_eq}
\end{align}
Again, by the choice of $\alpha_x$ and $\alpha_y$, \textcircled{5} is dominated by \textcircled{4}.  Since $\beta = \cO\left(\alpha_x^{\frac{2\theta - 1}{2\theta}} L_y^{-\frac{1}{2\theta}}\mu^{\frac{1}{\theta}}\varepsilon^{\frac{2\theta - 1}{\theta}}\right)$, we have
\[\textcircled{4} = \Theta\left(\alpha_{x}^{-1}C_{\beta}\right) = \Theta\left(\alpha_{x}^{-1}L_{y}^{\frac{1}{2\theta - 1}}\mu^{-\frac{2}{2\theta-1}}\beta^{\frac{2\theta}{2\theta - 1}}\right) 
= \cO\left(\varepsilon^2\right),\]
and $\textcircled{5}= \cO\left(\varepsilon^2\right)$ as well. In addition, by $\beta = \Theta\left(r^{-(2\theta - 1)} \mu^2 L_{y}^{-1}\; \varepsilon^{4\theta - 2}\right)$, we get
\[\textcircled{6} = \Theta\left(\frac{rC_{\beta}}{\beta}\right) = \Theta\left(r\;L_{y}^{\frac{1}{2\theta - 1}}\mu^{-\frac{2}{2\theta-1}}\beta^{\frac{1}{2\theta - 1}}\right) = \cO\left(\varepsilon^2\right).\]
To have $\textcircled{3} = \cO(\varepsilon^2)$, from the choice of $\beta$, we need 
$$KT = \Theta\left(\frac{r\Delta\Phi}{\beta\varepsilon^{2}}\right) = \Theta\left(\frac{r\Delta\Phi}{\min\left\{r^{-1}, \alpha_x^{\frac{2\theta - 1}{2\theta}} L_y^{-\frac{1}{2\theta}}\mu^{\frac{1}{\theta}}\varepsilon^{\frac{4\theta - 1}{\theta}}, r^{-(2\theta - 1)} \mu^2 L_{y}^{-1} \varepsilon^{4\theta },\; r^{-\frac{2\theta-1}{\theta}}L_{y}^{\frac{\theta-1}{\theta}} \mu^{\frac{1}{\theta}} \varepsilon^{\frac{4\theta-1}{\theta}}\right\}}\right).$$
Now plugging $\alpha_x = \Theta\left(\frac{1}{L_x+L_y^2}\right)$ and $r=\Theta(L_y^2 + L_y\sqrt{L_x})$ and combining the requirements on $KT$ in \eqref{eq:KT-1} and \eqref{eq:KT-2-case} completes the proof for the case of $\theta\in (\frac{1}{2},1]$ in the {finite--sum setting}.

{We now discuss the online setting. In this regime, since $C_{\sigma, x} = \frac{\sigma_x^2}{B}$ {and} $C_{\sigma, y} = \frac{\sigma_y^2}{B}$, we have
\begin{align}
C_{\sigma} = &~ \frac{\sigma_x^2}{B}
\left(2(1+L_y) \left( \frac{\alpha_x^2(r+L_{x})^2}{(r-\rho)^2} + \frac{1}{(r-\rho)^{2}} \right) + \frac{\alpha_x}{2}\right)
+\frac{\left(\left(16\alpha_y + 6\alpha_y^2\left(L_y+3\right)\right)\right)\sigma_y^2}{B}\nonumber\\
= &~ {\cO\left(\frac{\left(\sigma_x^2+\sigma_y^2\right)\alpha_x}{B}\right)}\label{for:c_sigma1},
\end{align}
{where the second equation follows from $\alpha_x = \cO\big(\frac{1}{r+L_x}\big)$, $L_y = \Omega(1)$, $r= \Theta(L_y^2 + L_y\sqrt{L_x})$, and $\alpha_y = \Theta(\alpha_x)$.}
{Hence, by the fact that \textcircled{7} dominates \textcircled{8},}
it suffices to have 
\begin{align}
B = \Theta\left(\frac{\left(\sigma_x^2+\sigma_y^2\right)\alpha_x}{\alpha_x\varepsilon^2}\right) = \Theta\left(\frac{\sigma_x^2+\sigma_y^2}{\varepsilon^2}\right)\label{B-eq-1}
\end{align}
to have both \textcircled{7} and \textcircled{8} to be  $\cO\left(\varepsilon^2\right)$. Additionally, {because $C_{\sigma, x} = \frac{\sigma_x^2}{B}$ and $C_{\sigma, y} = \frac{\sigma_y^2}{B}$}, having $B$ in \eqref{B-eq-1} makes 
\textcircled{10} is $\cO\left(\varepsilon^2\right)$ as well. 

Lastly, for bounding \textcircled{9}, we 
consider the cases when $\theta \in [0,\frac{1}{2}]$ and $(\frac{1}{2},1]$ separately. When $\theta \in [0,\frac{1}{2}]$, from $\beta= \Theta\left(\min\{\frac{1}{r}, \frac{\mu^2}{L_y}\}\right)$ {and \eqref{for:c_sigma1}}, having
\[B=\Theta\left(\frac{\left(\sigma_x^2+\sigma_y^2\right)\alpha_x\max\{r^2, \frac{rL_y}{\mu^2}\}}{\varepsilon^{2}}\right)=\Theta\left(\frac{\left(\sigma_x^2+\sigma_y^2\right)\alpha_x{L_y^2}\max\{L_y^2 + L_x, \frac{L_y + \sqrt{L_x}}{\mu^2}\}}{\varepsilon^{2}}\right)\]
ensures \textcircled{9} is $\cO\left(\varepsilon^2\right)$. When $\theta \in \left(\frac{1}{2},1\right]$, from the choice of $\beta$, we need
$$B = \Theta\left(\frac{r\alpha_x\left(\sigma_x^2+\sigma_y^2\right)}{\beta\varepsilon^{2}}\right) = \Theta\left(\frac{r\alpha_x\left(\sigma_x^2+\sigma_y^2\right)}{\min\left\{r^{-1}, \alpha_x^{\frac{2\theta - 1}{2\theta}} L_y^{-\frac{1}{2\theta}}\mu^{\frac{1}{\theta}}\varepsilon^{\frac{4\theta - 1}{\theta}}, r^{-(2\theta - 1)} \mu^2 L_{y}^{-1} \varepsilon^{4\theta },\; r^{-\frac{2\theta-1}{\theta}}L_{y}^{\frac{\theta-1}{\theta}} \mu^{\frac{1}{\theta}} \varepsilon^{\frac{4\theta-1}{\theta}}\right\}}\right)$$
to render \textcircled{9} is $\cO\left(\varepsilon^2\right)$. 
Hence, plugging $\alpha_x = \Theta\left(\frac{1}{L_x+L_y^2}\right)$ and $r=\Theta(L_y^2 + L_y\sqrt{L_x})$ and combining all the above requirements on $B$ completes the proof for the online setting.
}
We now prove $\mathbb{E}\left\|\nabla_{z}d_{r}\left(\tilde{y}, \tilde{x}\right)\right\|^{2} =\cO\left(\varepsilon^2\right)$. Substituting the bounds from Lemma \ref{lemma:scenario1} into Lemma \ref{dz_for_smooth_case} and rearranging terms, we get 
\begin{align}
&\mathbb{E}\left\|\nabla_{z}d_{r}\left(\tilde{y}, \tilde{x}\right)\right\|^{2} \nonumber\\
\leq &~ \underbrace{\frac{8\alpha_x\Delta\Phi}{KT}
\left(15r^2 + \frac{97}{\alpha_x^2}\right)}_{\textcircled{i}}
+ \underbrace{\frac{140\alpha_yr^2\Delta\Phi}{KT}
}_{\textcircled{ii}} 
+ \underbrace{\frac{24r\Delta\Phi}{\beta KT}}_{\textcircled{iii}}
 \nonumber\\
\quad & + \underbrace{8\alpha_x\chi_{\theta} C_{\beta}
\left(15r^2 + \frac{97}{\alpha_x^2}\right)}_{\textcircled{iv}} + \underbrace{140\alpha_yr^2\chi_{\theta}C_{\beta}}_{\textcircled{v}} 
+ \underbrace{\frac{24r\chi_{\beta}C_{\beta}}{\beta}}_{\textcircled{vi}}\nonumber\\
\quad & {+ \underbrace{8\alpha_x C_{\sigma}
\left(15r^2 + \frac{97}{\alpha_x^2}\right)}_{\textcircled{vii}} + \underbrace{140\alpha_yr^2C_{\sigma}}_{\textcircled{viii}} 
+ \underbrace{\frac{24rC_{\sigma}}{\beta}}_{\textcircled{vi}}{+\underbrace{61C_{\sigma,x}}_{\textcircled{ix}} }}. \label{final_expression_for_optimality2}
\end{align}
{{As before, we first discuss the finite--sum setting. Since in this regime, $C_{\sigma,x} = C_{\sigma,y} = 0$, the last four terms vanish.} Comparing the six remaining underbraced terms in \eqref{final_expression_for_optimality2} to those in \eqref{final_expression_for_optimality}, we only need to show that \textcircled{ii} is $\cO(\varepsilon^2)$ when $\theta\in[0,\frac{1}{2}]$ and that \textcircled{ii} and \textcircled{v} are both $\cO(\varepsilon^2)$ when $\theta\in(\frac{1}{2},1]$ with the selected $KT$ and $\beta$. Since $r=\Theta(L_y^2 + L_y\sqrt{L_x})$ and $\alpha_y = \Theta(\frac{1}{L_y})$, to ensure \textcircled{ii} is $\cO(\varepsilon^2)$, it suffices to have 
$$KT = \Theta\left(\frac{L_{y}\Delta\Phi(L_y + \sqrt{L_x})^2}{\varepsilon^{2}}\right)=\Theta\left(\frac{L_{y}\Delta\Phi(L_y^2 + L_x)}{\varepsilon^{2}}\right).$$
This finishes the proof for the case of $\theta\in[0,\frac{1}{2}]$. 
In addition, when $\theta\in(\frac{1}{2},1]$, using $\alpha_y = \Theta\left(\frac{1}{L_{y}}\right)$ and $\beta = \cO\left(r^{-\frac{2\theta-1}{\theta}}L_{y}^{\frac{\theta-1}{\theta}} \mu^{\frac{1}{\theta}} \varepsilon^{\frac{2\theta-1}{\theta}}\right)$, we have
\[\textcircled{v} = \Theta\left(\alpha_{y}r^2C_{\beta}\right) = \Theta\left(r^2 L_{y}^{\frac{2-2\theta}{2\theta - 1}}\mu^{-\frac{2}{2\theta-1}}\beta^{\frac{2\theta}{2\theta - 1}}\right) 
= \cO\left(\varepsilon^2\right).\]

{For the online setting, {we only need to further check \textcircled{viii}, which is dominated by \textcircled{vii} and thus is $\cO\left(\varepsilon^2\right)$ as well.}

Finally, the sample complexity is given by 
\[2KB+4KTM = 2KT \left(\frac{B}{T}+2M\right)= 2KT \left(\frac{B}{M}+2M\right)\le 4 KT \left\lceil \sqrt{2B} \right\rceil = \cO\left(\left\lceil\sqrt{B}\right\rceil KT\right),\]
which completes the proof.
}
}
\end{proof}

\section{Proofs for the nonsmooth case} In this section, we give the proofs of the claimed results for the nonsmooth case.
{\begin{lemma}[Lipschitz continuity]
\label{lem:Lip-f-lambda}
Suppose Assumption \ref{assumption:F-nonsmooth} holds. Then for every fixed sample {$\bm{\xi} \in \Xi$}, the function $f^{\lambda}(\cdot, \cdot; \bm{\xi})$ defined in \eqref{eq:smoothed-func} is $\ell_{\lambda}$-Lipschitz continuous with $\ell_{\lambda} = \max\left\{\ell_{\varphi}\ell_{h}\ell_{c}\sqrt{d_{h}},\; \ell_{\varphi}\right\}$, 
i.e., 
    \[
    \left\|f^{\lambda}(x_1, y_1; \bm{\xi}) - f^{\lambda}(x_2, y_2; \bm{\xi})\right\| \leq \ell_{\lambda} \left(\|x_1 - x_2\| + \|y_1 - y_2\|\right), \forall\, (x_1, y_1), (x_2, y_2) \in \mathcal{X} \times \mathcal{Y}.
    \]
Also, $F^\lambda$ is $\hat\ell$-Lipschitz continuous.
\end{lemma}
\begin{proof}
We note that $\forall w_{1}, w_{2} \in \mathbb{R}^{d_{c}}$, we have
\begin{align}
\left\|h^\lambda(w_{1}) - h^\lambda(w_{2})\right\|^{2} 
= \sum_{j=1}^{d_{h}}\lvert h^{\lambda}_{j}\left(w_{1}\right) - h^{\lambda}_{j}\left(w_{2}\right)\rvert^{2}
\overset{(i)}{\leq} \sum_{j=1}^{d_{h}} \ell_{h}^{2}\left\|w_{1} - w_{2}\right\|^{2}
= d_{h}\ell_{h}^{2}\left\|w_{1} - w_{2}\right\|^{2},\label{h nonsmooth only}
\end{align}
where $(i)$ follows from \citep[Lemma 2.1]{DrusvyatskiyPaquette2019}. 
Hence, 
\begin{align*}
\| f^{\lambda}(x_1, y_1; \bm{\xi}) - f^{\lambda}(x_2, y_2; \bm{\xi})\| 
= &~ \left\| \varphi\left(h^{\lambda}\left(c\left(x_{1};\bm{\xi}\right)\right), y_{1} ; \bm{\xi}\right) - \varphi\left(h^{\lambda}\left(c\left(x_{2};\bm{\xi}\right)\right), y_{2} ; \bm{\xi}\right)  \right\|\nonumber\\
\overset{(i)}{\leq} &~ \ell_{\varphi}\left\|h^{\lambda}\left(c\left(x_{1};\bm{\xi}\right)\right) - h^{\lambda}\left(c\left(x_{2};\bm{\xi}\right)\right)\right\|  + \ell_{\varphi}\left\|y_{1} - y_{2}\right\|\\
\overset{(ii)}{\leq} &~ \ell_{\varphi}\ell_{h}\ell_{c}\sqrt{d_{h}}\left\|x_{1} - x_{2}\right\|  + \ell_{\varphi}\left\|y_{1} - y_{2}\right\|\\
\leq &~ \ell_{\lambda}\big(\left\|x_{1} - x_{2}\right\|  + \left\|y_{1} - y_{2}\right\|\big),
\end{align*}
where $(i)$ holds by the $\ell_{\varphi}$-lipschitz continuity of $\varphi(\cdot, \cdot)$, and $(ii)$ follows from  \eqref{h nonsmooth only} and the $\ell_{c}$-lipschitz continuity of $c$. This completes the proof.
\end{proof}
The next lemma shows the Lipschitz continuity of $\nabla F^\lambda$.
\begin{lemma}[Lipschitz smoothness]\label{lemma for smooth f lambda}
Suppose Assumption \ref{assumption:F-nonsmooth} holds. The gradient $\nabla_x f^{\lambda}(x,y;\boldsymbol{\xi})$ is $L_{\lambda, x}$-Lipschitz continuous in $x$ and $L_{\lambda, y}$-Lipschitz continuous in $y$ in expectation, 
i.e., 
\begin{align*}
\mathbb{E}\| \nabla_{x}f^{\lambda}(x_{1},y;\bm{\xi}) - \nabla_{x}f^{\lambda}(x_{2},y;\bm{\xi}) \|^{2} \leq L_{\lambda,x}^{2}\left\|x_{1} - x_{2}\right\|^{2}, \forall\, (x_1, y), (x_2, y) \in \mathcal{X} \times \mathcal{Y} \\
\mathbb{E}\| \nabla_{x}f^{\lambda}(x,y_{1};\bm{\xi}) - \nabla_{x}f^{\lambda}(x,y_{2};\bm{\xi}) \|^{2} \leq L_{\lambda,y}^{2}\left\|y_{1} - y_{2}\right\|^{2},\; \forall\, (x, y_1), (x, y_2) \in \mathcal{X} \times \mathcal{Y}
\end{align*}
where $L_{\lambda,x} = \sqrt{\frac{3 \ell_c^4 \ell_\varphi^2 d_{h}}{\lambda^{2}} +  3 d_h\ell_h^2 \ell_\varphi^2 L_{c}^{2} + 3 \ell_c^4 d_h^{2}\ell_h^4 L_{\varphi}^{2}}$ and $L_{\lambda,y} = \max\left\{\sqrt{d_{h}}L_{\varphi}\ell_{h}\ell_{c}, \; L_{\varphi}\right\}$. Additionally, $\nabla_y f^{\lambda}(\cdot\,, \cdot\,; \bm{\xi})$ is $L_{\lambda,y}$-Lipschitz continuous in expectation, i.e., 
\[\mathbb{E}\| \nabla_{y}f^{\lambda}(x_{1},y_{1};\bm{\xi}) - \nabla_{y}f^{\lambda}(x_{2},y_{2};\bm{\xi}) \|^{2} \leq L_{\lambda,y}^{2}\left(\|x_{1} - x_{2}\|^2+\left\|y_{1} - y_{2}\right\|^{2}\right) .\]
\end{lemma}
\begin{proof} We first note that $\forall w_{1}, w_{2} \in \mathbb{R}^{d_{c}}$,
\begin{align}
\left\|\nabla h^\lambda(w_{1}) - \nabla h^\lambda(w_{2})\right\|^{2} 
= \sum_{j=1}^{d_{h}}\left\| \nabla h^{\lambda}_{j}\left(w_{1}\right) - \nabla h^{\lambda}_{j}\left(w_{2}\right)\right\|^{2}
\overset{(i)}{\leq} \sum_{j=1}^{d_{h}} \frac{1}{\lambda^{2}}\left\|w_{1} - w_{2}\right\|^{2}
=\frac{d_{h}}{\lambda^{2}}\left\|w_{1} - w_{2}\right\|^{2}, \label{nabla h nonsmooth only}
\end{align}
where $(i)$ follows from \citep[Lemma 2.1]{DrusvyatskiyPaquette2019}. 
{In addition, for each $\bm{\xi}\in\{\bm{\xi}_i\}_{i=1}^N$, we have
$$\nabla_{x}f^{\lambda}(x,y;\bm{\xi}) = \nabla_x c(x;\bm{\xi}) \nabla h^{\lambda}\left(c(x;\bm{\xi})\right) \nabla_1 \varphi\left(h^{\lambda}\left(c\left(x;\bm{\xi}\right)\right), y ; \bm{\xi}\right),$$
where $\nabla_x c \in \mathbb{R}^{d_{x} \times d_{c}}$ denotes the Jacobian matrix of $c$ about $x$, $\nabla h^{\lambda} \in \mathbb{R}^{d_{c} \times d_{h}}$ is for the Jacobian matrix of $h^{\lambda}$, and $\nabla_1 \varphi \in \mathbb{R}^{d_{h}}$ denotes the partial gradient of $\varphi$ about the first argument. For any $(x,y)$, $(x_1,y_1)$ and $(x_2, y_2)$, we let
}
\begin{align*}
&c = c(x; \bm{\xi}), \;c_1 = c(x_1; \bm{\xi}),\; c_2 = c(x_2; \bm{\xi}),\\
& h^{\lambda} = h^{\lambda}\left(c\left(x;\bm{\xi}\right)\right), \;h^{\lambda}_{1} = h^{\lambda}\left(c\left(x_{1};\bm{\xi}\right)\right),\; h^{\lambda}_{2} = h^{\lambda}\left(c\left(x_{2};\bm{\xi}\right)\right),\\
&J = \nabla_x c(x, \bm{\xi}),\; J_1 = \nabla_x c(x_1, \bm{\xi}),\ J_2 = \nabla_x c(x_2, \bm{\xi}), \\
& H = \nabla h^\lambda(c),\; H_1 = \nabla h^\lambda(c_1), \ H_2 = \nabla h^\lambda(c_2), \\
&b_1 = \nabla_1 \varphi\left(h^{\lambda}_{1}, y ; \bm{\xi}\right), \ b_2 = \nabla_1 \varphi\left(h^{\lambda}_{2}, y ; \bm{\xi}\right),\\
&q_1 = \nabla_1 \varphi\left(h^{\lambda}, y_1 ; \bm{\xi}\right), \ q_2 = \nabla_1 \varphi\left(h^{\lambda}, y_2 ; \bm{\xi}\right).
\end{align*}
Then
\begin{align}
&~\mathbb{E}\| \nabla_{x}f^{\lambda}(x_{1},y;\bm{\xi}) - \nabla_{x}f^{\lambda}(x_{2},y;\bm{\xi}) \|^{2}\nonumber\\
= &~   \mathbb{E}\|J_1H_1b_1 - J_2H_2b_2\|^2\nonumber\\
\le &~ 3\|J_1H_1b_1 - J_2H_1b_1\|^2 + 3\|J_2H_1b_1 - J_2H_2b_1\|^2 + 3\|J_2H_2b_1 -  J_2H_2b_2\|^2 \nonumber\\
\le &~ 3\mathbb{E}\|H_1b_1\|^2\mathbb{E}\|J_1 - J_2\|^2 + 3\mathbb{E}\|J_2\|^2 \|b_1\|^2 \mathbb{E}\|H_1 - H_2\|^2 + 3\mathbb{E}\|J_2H_2\|^2 \mathbb{E}\|b_1 -  b_2\|^2\nonumber\\
\overset{(i)}{\le} &~  3 d_h\ell_h^2 \ell_\varphi^2  \mathbb{E}\|J_1 -J_2\|^2 + 3 \ell_c^2 \ell_\varphi^2 \mathbb{E}\|H_1 - H_2\|^2 + 3 \ell_c^2 d_h\ell_h^2 \mathbb{E}\|b_1 - b_2\|^2\nonumber\\
\overset{(ii)}{\leq} &~ 3 d_h\ell_h^2 \ell_\varphi^2 L_{c}^{2}  \|x_1 -x_2\|^2 + \frac{3 \ell_c^4 \ell_\varphi^2 d_{h}}{\lambda^{2}} \|x_1 - x_2\|^2 + 3 \ell_c^4 d_h^{2}\ell_h^4 L_{\varphi}^{2} \|x_1 - x_2\|^2 \nonumber\\
= &~ \left(\frac{3 \ell_c^4 \ell_\varphi^2 d_{h}}{\lambda^{2}} +  3 d_h\ell_h^2 \ell_\varphi^2 L_{c}^{2} + 3 \ell_c^4 d_h^{2}\ell_h^4 L_{\varphi}^{2} \right) \|x_1 - x_2\|^2,  \label{nonsmooth-lipschitz-1}
\end{align}
where $(i)$ follows from the $\ell_\varphi$- and $\ell_c$-Lipschitz continuity of $\varphi$ and $c$ and the $\sqrt{d_h} \ell_h$-Lipschitz continuity of $h^\lambda$ by \eqref{h nonsmooth only}, and $(ii)$ holds by the \eqref{nabla_x_for_c}, $\ell_{c}$-Lipschitz continuity of $c$ and $\frac{\sqrt{d_h}}{\lambda}$-Lipschitz continuity of $\nabla h^\lambda$ by \eqref{nabla h nonsmooth only}, \eqref{nabla_1_phi} and the $\sqrt{d_h} \ell_h$-Lipschitz continuity of $h^\lambda$ by \eqref{h nonsmooth only}. 

Similarly, we have
\begin{align}
\mathbb{E}\| \nabla_{x}f^{\lambda}(x,y_{1};\bm{\xi}) - \nabla_{x}f^{\lambda}(x,y_{2};\bm{\xi}) \|^{2}
= &~ \mathbb{E}\|JHq_1 - JHq_2\|^2\nonumber\\
= &~ \mathbb{E}\left\|J\right\|^{2}\left\|H\right\|^{2}\left\|q_{1} - q_{2}\right\|^{2}\nonumber\\
\overset{(i)}{\leq} &~ d_{h}\ell_{h}^{2}\ell_{c}^{2}L_{\varphi}^{2}\left\|y_{1} - y_{2}\right\|^{2} \label{nonsmooth-lipschitz-2},
\end{align}
where $(i)$ follows from \eqref{nabla_1_phi}, $\ell_{c}$-Lipschitz continuity of $c$ and $\sqrt{d_h}\ell_{h}$-Lipschitz continuity of $h^{\lambda}$. 
Moreover, we have
\begin{align}
\mathbb{E}\| \nabla_{y}f^{\lambda}(x_{1},y_{1};\bm{\xi}) - \nabla_{y}f^{\lambda}(x_{2},y_{2};\bm{\xi}) \|^{2}
= &~ \mathbb{E}\left\| \nabla_{y}\varphi\left(h^{\lambda}\left(c\left(x_{1};\bm{\xi}\right)\right), y_{1} ; \bm{\xi}\right) - \nabla_{y}\varphi\left(h^{\lambda}\left(c\left(x_{2};\bm{\xi}\right)\right), y_{2} ; \bm{\xi}\right)  \right\|^{2} \nonumber\\
\overset{(i)}{\leq} &~ L_{\varphi}^{2}\big(\left\|h^{\lambda}_{1} - h^{\lambda}_{2}\right\|^{2} + \left\|y_{1} - y_{2}\right\|^{2}\big)\nonumber\\
\overset{(ii)}{\leq} &~ d_{h}L_{\varphi}^{2}\ell_{h}^{2}\ell_{c}^{2}\left\|x_{1} - x_{2}\right\|^{2} + L_{\varphi}^{2} \left\|y_{1} - y_{2}\right\|^{2} \label{nonsmooth-lipschitz-2_2nd},
\end{align}
where $(i)$ follows from \eqref{nabla_1_phi}, and $(ii)$ is by $\ell_{c}$-Lipschitz continuity of $c$ and \eqref{h nonsmooth only}. The desired result then follows from the specified choices of $L_{\lambda,x}$ and $L_{\lambda,y}$ and thus, the proof is complete.
\end{proof}
We now show that $F$ and $F^\lambda$ are weakly convex and have the same weak convexity module.
{\begin{lemma}[Weak Convexity]\label{lem:weak-f-lambda}
Suppose Assumption \ref{assumption:F-nonsmooth} holds and let {$\rho_{\scriptscriptstyle \lambda} = d_{h}L_{\varphi}\ell_{h}^{2}\ell_{c}^{2} + L_{c} \ell_{\varphi}\ell_{h}\sqrt{d_{h}}$.}
Then for each $y\in\cY$, the functions $F(\cdot, y)$ and $F^\lambda(\cdot, y)$ are both $\rho_{\scriptscriptstyle \lambda}$-weakly convex.
\end{lemma}
\begin{proof}{For any $x, \hat{x}\in\cX$ and $y\in\cY$, 
we define the following notations for simplicity:
\[\vartheta_{\bm{\xi}} \in \partial h(c(x; \bm{\xi}));\; w_{\bm{\xi}} =  h(c(\hat{x}; \bm{\xi})) -  h(c(x; \bm{\xi})) -  \vartheta_{\bm{\xi}}^{\top}\left( c(\hat{x}; \bm{\xi}) - c(x; \bm{\xi}) \right); \; v_{\bm{\xi}} = c(\hat{x}; \bm{\xi}) - c(x; \bm{\xi}) - \nabla_{x} c(x;\bm{\xi})^{\top}(\hat{x} - x).\]
Then we have 
\begin{align}
&F(\hat{x},y) \nonumber\\
= &~ \mathbb{E}_{\bm{\xi} \sim \mathbb{P}}\left[\varphi\big( h(c(\hat{x}; \bm{\xi})), y; \bm{\xi} \big)\right]\nonumber\\
\overset{(i)}{\geq} &~ {\mathbb{E}_{\bm{\xi} \sim \mathbb{P}}\Bigg(\varphi\big( h(c(x; \bm{\xi})), y; \bm{\xi} \big) + \big\langle \nabla_{1} \varphi(h(c(x; \bm{\xi})), y; \bm{\xi})\;,\;h(c(\hat{x}; \bm{\xi})) -  h(c(x; \bm{\xi}))  \big\rangle}  \nonumber\\
\quad & {-\frac{L_{\varphi}}{2}\left\|h(c(\hat{x}; \bm{\xi})) -  h(c(x; \bm{\xi}))\right\|^{2}\Bigg)} \nonumber\\
\overset{(ii)}{\geq} &~ \mathbb{E}_{\bm{\xi} \sim \mathbb{P}}\left[\varphi\big( h(c(x; \bm{\xi})), y; \bm{\xi} \right] + \underbrace{\mathbb{E}_{\bm{\xi} \sim \mathbb{P}}\big\langle \nabla_{1} \varphi(h(c(x; \bm{\xi})), y; \bm{\xi})\;,\;\vartheta_{\bm{\xi}}^{\top}\left( c(\hat{x}; \bm{\xi}) - c(x; \bm{\xi}) \right)\big\rangle}_{\textcircled{1}} \nonumber\\
\quad &  + \underbrace{\mathbb{E}_{\bm{\xi} \sim \mathbb{P}}\big \langle \nabla_{1} \varphi(h(c(x; \bm{\xi})), y; \bm{\xi}),w_{\bm{\xi}} \big \rangle}_{\textcircled{2}}-\frac{d_{h}L_{\varphi}\ell_{h}^{2}\ell_{c}^{2}}{2}\left\|\hat{x} -x\right\|^{2} \label{lispchitz-constant-F-F-lambda},
\end{align}
where {$(i)$ holds  since $\nabla_{1}\varphi\left(\cdot,\cdot;\bm{\xi}\right)$ is $L_{\varphi}$-Lipschitz continuous},
and $(ii)$ follows from $\sqrt{d_{h}}\ell_{h}$-Lipschitz continuity of $h$ and $\ell_{c}$-Lipschitz continuity of $c$.

We bound \textcircled{1} in \eqref{lispchitz-constant-F-F-lambda} by 
\begin{align}
&\mathbb{E}_{\bm{\xi} \sim \mathbb{P}}\big\langle \nabla_{1} \varphi(h(c(x; \bm{\xi})), y; \bm{\xi})\;,\;\vartheta_{\bm{\xi}}^{\top}\left( c(\hat{x}; \bm{\xi}) - c(x; \bm{\xi}) \right)\big\rangle\nonumber\\
= &~ \mathbb{E}_{\bm{\xi} \sim \mathbb{P}}\big\langle \vartheta_{\bm{\xi}} \nabla_{1} \varphi(h(c(x; \bm{\xi})), y; \bm{\xi})\;,\;\nabla_{x} c(x;\bm{\xi})^{\top}(\hat{x} - x) \big\rangle + \mathbb{E}_{\bm{\xi} \sim \mathbb{P}}\big\langle \nabla_{1} \varphi(h(c(x; \bm{\xi})), y; \bm{\xi})\;,\;\vartheta_{\bm{\xi}}^{\top}v_{\bm{\xi}}\big\rangle\nonumber\\
\overset{(i)}{\geq} & \mathbb{E}_{\bm{\xi} \sim \mathbb{P}}\big\langle \nabla_{x} c(x;\bm{\xi})\;\vartheta_{\bm{\xi}}\;\nabla_{1} \varphi(h(c(x; \bm{\xi})), y; \bm{\xi}),\; \hat{x} - x \big\rangle - \mathbb{E}_{\bm{\xi} \sim \mathbb{P}}\left\|\nabla_{1} \varphi(h(c(x; \bm{\xi})), y; \bm{\xi})\right\|\left\|\vartheta_{\bm{\xi}}\right\|\left\|v_{\bm{\xi}}\right\|\nonumber\\
\overset{(ii)}{\geq} &  \mathbb{E}_{\bm{\xi} \sim \mathbb{P}}\big\langle \nabla_{x} c(x;\bm{\xi})\;\vartheta_{\bm{\xi}}\;\nabla_{1} \varphi(h(c(x; \bm{\xi})), y; \bm{\xi}),\; \hat{x} - x \big\rangle - \frac{\ell_{\varphi}\ell_{h}L_{c}\sqrt{d_{h}}}{2}\left\|\hat{x} - x\right\|^{2} \label{second-component},
\end{align}
where $(i)$ holds by Cauchy-Schwarz inequality, and $(ii)$ follows from $\ell_{\varphi}$-Lipschitz continuity of $\varphi$, $\sqrt{d_{h}}\ell_{h}$-Lipschitz continuity of $h$ and $\mathbb{E}_{\bm{\xi} \sim \mathbb{P}}\left\|v_{\bm{\xi}}\right\| \leq \frac{L_{c}}{2}\left\|\hat{x} - x\right\|^{2}$ (from \eqref{nabla_x_for_c}). Now, since $h$ is convex and $\varphi$ is non-decreasing in its first argument, it follows that 
\begin{align}
\textcircled{2} = \mathbb{E}_{\bm{\xi} \sim \mathbb{P}}\big \langle \nabla_{1} \varphi(h(c(x; \bm{\xi})), y; \bm{\xi}),w_{\bm{\xi}} \big \rangle \geq 0. \label{third-component}
\end{align}
Plugging \eqref{second-component} and \eqref{third-component} into \eqref{lispchitz-constant-F-F-lambda}, we get 
\begin{align}
F(\hat{x},y) &\ge F(x,y) + \left\langle \mathbb{E}_{\bm{\xi} \sim \mathbb{P}}\left[\nabla_{x} c(x;\bm{\xi})\;\vartheta_{\bm{\xi}}\;\nabla_{1} \varphi(h(c(x; \bm{\xi})), y; \bm{\xi})\right],\; \hat{x} - x \right\rangle - \frac{\rho_{\scriptscriptstyle \lambda}}{2}\left\|\hat{x} -x\right\|^{2},
\end{align}
which renders that $F(\cdot,y)$ is $\rho_{\scriptscriptstyle \lambda}$-weakly convex.} 

Noting that $h^{\lambda}(\cdot)$ is convex with the same Lipschitz continuity constant as $h$, i.e., $\ell_h \sqrt{d_h}$, we can follow the same arguments to show that $F^{\lambda}(x,y)$ is also $\rho_{\scriptscriptstyle \lambda}$-weakly convex and thus complete the proof.
\end{proof}}}

The lemma below establishes that the smoothed objective function $F^{\lambda}(x,y)$ satisfies the K\L{} property with respect to $y$ by choosing an appropriate smoothing parameter $\lambda$.
\begin{lemma}
\label{lemma:kl-regularized}
Suppose Assumptions \ref{assumption:F-nonsmooth} and \ref{assumption:g-kl} hold and let $\lambda \leq \frac{2\tilde{\delta}}{\ell_{h}^{2}\sqrt{d_{h}}}$. Then for all $x \in \mathcal{X}$, $F^{\lambda}(x,\cdot)$ satisfies 
the {extended K\L{} property}
with parameters $\mu$ and $ \theta$.
\end{lemma}

\begin{proof}
By  \citep[Lemma 2.1]{DrusvyatskiyPaquette2019}, 
it follows that for any 
$w \in \mathbb{R}^{d_c}$,
\begin{align}
\left\|h^\lambda(w) - h(w)\right\| 
= \sqrt{\sum_{j=1}^{d_{h}}\lvert  h^{\lambda}_{j}\left(w\right) - h_{j}\left(w\right)\rvert^{2}}
\leq \sqrt{\sum_{j=1}^{d_{h}} \frac{\lambda^{2}\ell_{h}^{4}}{4}}
=\frac{\lambda \ell_{h}^{2}\sqrt{d_{h}}}{2}.
\label{h_minus_h_lambda_nonsmooth}
\end{align}
Since $\lambda \leq \frac{2\tilde{\delta}}{\ell_{h}^{2}\sqrt{d_{h}}}$, {$\forall \bm{\xi} \in \Xi$}, it holds $\left\|h^\lambda(c(x;\bm{\xi})) - h(c(x;\bm{\xi}))\right\| \leq \tilde{\delta}.$ 
{Let $\hat{u}: \Xi \mapsto \mathbb{R}^{d_h}$ be defined as $u(\bm{\xi}) = h^\lambda(c(x;\bm{\xi}))$. Consequently, we get}
{\[
F^{\lambda}(x,y)
= \mathbb{E}_{\bm{\xi}\sim \mathbb{P}} \left[\varphi\left(h^{\lambda}(c(x;\bm{\xi})),\, y;\, \bm{\xi}\right)\right]
= g(\hat{u}, y),
\]}
and thus from Assumption \ref{assumption:g-kl}, 
the desired result follows. 
\end{proof}

{
The next lemma will be used to certify a $\delta$-subgradient defined in \eqref{eq:-delta-subdiff}.
\begin{lemma}
Under Assumption~\ref{assumption:F-nonsmooth}, it holds
\begin{align}
&\lvert F(x,y) - F^{\lambda}(x,y)\rvert \le \frac{\lambda L_{\varphi}L_{h}^{2}\sqrt{d_{h}}}{2},\ \forall\, x \in \mathcal{X}, y\in \mathcal{Y},\label{f_minus_f_lambda} \\
& \lvert F(x_{1},y) - F(x_{2},y)\rvert \le  \ell_{\varphi}\ell_{h}\ell_{c}\sqrt{d_{h}}\left\|x_{1} - x_{2}\right\|, \ \forall\, x_1, x_2 \in \mathcal{X}, y\in \mathcal{Y}, \label{f_minus_f_lambda2} \\
& \left\|\nabla_{y}F^{\lambda}\left(x,y\right) - \nabla_{y}F\left(x,y\right) \right\| \le \frac{\lambda L_{\varphi}\ell_{h}^{2}\sqrt{d_{h}}}{2}, \ \forall\, x \in \mathcal{X}, y\in \mathcal{Y}.
\label{nabla h nonsmooth y}
\end{align}
\end{lemma}
\begin{proof}
{For any $ x \in \mathcal{X}, y\in \mathcal{Y}$, it holds
\begin{align*}
\lvert F(x,y) - F^{\lambda}(x,y)\rvert 
= &~ \left\lvert \mathbb{E}_{\bm{\xi} \sim \mathbb{P}}\left[f(x,y,\bm{\xi})\right] -  \mathbb{E}_{\bm{\xi} \sim \mathbb{P}}\left[f^{\lambda}(x,y,\bm{\xi})\right] \right\rvert \nonumber\\
{\leq} & ~\mathbb{E}_{\bm{\xi} \sim \mathbb{P}} \left\lvert f(x,y,\bm{\xi}) - f^{\lambda}(x,y,\bm{\xi})  \right\rvert\nonumber\\
= &~ \mathbb{E}_{\bm{\xi} \sim \mathbb{P}}\left\lvert \varphi\big( h(c(x; \bm{\xi})), y; \bm{\xi} \big) - \varphi\big( h^{\lambda}(c(x; \bm{\xi})), y; \bm{\xi} \big) \right\rvert \nonumber\\
\overset{(i)}{\leq} &~ \mathbb{E}_{\bm{\xi} \sim \mathbb{P}} \left\| h(c(x; \bm{\xi})) - h^{\lambda}(c(x; \bm{\xi})) \right\|\nonumber\\
\overset{(ii)}{\leq} &~\frac{\lambda L_{\varphi}L_{h}^{2}\sqrt{d_{h}}}{2} ,
\end{align*}
where 
$(i)$ follows from $L_{\varphi}$-Lipschitz continuity of $\varphi$, and $(ii)$ holds by  \eqref{h_minus_h_lambda_nonsmooth}. This proves \eqref{f_minus_f_lambda}.

Additionally, for any $ x_{1}, x_{2} \in \mathcal{X}, y\in \mathcal{Y}$, we have 
\begin{align*}
\lvert F(x_{1},y) - F(x_{2},y)\rvert 
= &~ \left | \mathbb{E}_{\bm{\xi} \sim \mathbb{P}}\left[f(x_{1},y,\bm{\xi})\right] -  \mathbb{E}_{\bm{\xi} \sim \mathbb{P}}\left[f(x_{2},y,\bm{\xi})\right] \right | \nonumber\\
\leq &~ \mathbb{E}_{\bm{\xi} \sim \mathbb{P}} \left | f(x_{1},y,\bm{\xi}) - f(x_{2},y,\bm{\xi})  \right |\nonumber\\
= &~ \mathbb{E}_{\bm{\xi} \sim \mathbb{P}}\left | \varphi\big( h(c(x_{1}; \bm{\xi})), y; \bm{\xi} \big) - \varphi\big( h(c(x_{2}; \bm{\xi})), y; \bm{\xi} \big) \right | \nonumber\\
\leq &~ \ell_{\varphi}\ell_{h}\ell_{c}\sqrt{d_{h}}\left\|x_{1} - x_{2}\right\|,
\end{align*}
where the last inequality 
follows from $\ell_{\varphi}$-, $\sqrt{d_h}\ell_h$-, and $\ell_{c}$-Lipschitz continuity of $\varphi$, $h$, and $c$. Thus \eqref{f_minus_f_lambda2} follows. 

Moreover, we have  for any  $x \in \mathcal{X}, \; y \in \mathcal{Y}$ that  
\begin{align*}
\left\|\nabla_{y}F^{\lambda}\left(x,y\right) - \nabla_{y}F\left(x,y\right) \right\| 
= &~ \left\|\mathbb{E}_{\bm{\xi} \sim \mathbb{P}}\left[ \nabla_{y}f^{\lambda}(x, y; \bm{\xi}_i)\right] - \mathbb{E}_{\bm{\xi} \sim \mathbb{P}}\left[ \nabla_{y}f(x, y; \bm{\xi}_i)\right]\right\|\nonumber\\
\leq &~ \mathbb{E}_{\bm{\xi} \sim \mathbb{P}}\left\| \nabla_{y}f^{\lambda}(x, y; \bm{\xi}_i) - \nabla_{y}f(x, y; \bm{\xi}_i)\right\|\nonumber\\
= &~ \mathbb{E}_{\bm{\xi} \sim \mathbb{P}}\left\| \nabla_{y}\varphi\big( h^{\lambda}(c(x; \bm{\xi})), y; \bm{\xi} \big) - \nabla_{y}\varphi\big( h(c(x; \bm{\xi})), y; \bm{\xi} \big)\right\|\nonumber\\
\overset{(i)}{\leq} &~ \mathbb{E}_{\bm{\xi} \sim \mathbb{P}}L_{\varphi}\left\| h^{\lambda}(c(x; \bm{\xi})) - h(c(x; \bm{\xi}))\right\|\nonumber\\
\overset{(ii)}{\leq} &~ \frac{\lambda L_{\varphi}\ell_{h}^{2}\sqrt{d_{h}}}{2},
\end{align*}
where $(i)$ follows from 
\eqref{nabla_y_phi}, and $(ii)$ is by 
\eqref{h_minus_h_lambda_nonsmooth}. This proves \eqref{nabla h nonsmooth y} and completes the proof.}
\end{proof}
}

\begin{lemma}[Stationarity]
Suppose Assumption \ref{assumption:F-nonsmooth} holds. 
Then for all $x \in \mathcal{X}$ and $y \in \mathcal{Y}$ we have
\begin{align}
\mathrm{dist}\left(0,  \partial_{x}^{\delta} \big(F(x,y) + \iota_\cX(x)\big)\right)^{2}
\leq &~  \left\|\nabla_{z}d_{r}^{\lambda}(y,x)\right\|^{2},
\label{eq:dist_x_final}
\\
\mathrm{dist}\left(0, -\nabla_{y}F\left(x,y\right) + \cN_\cY(y)\right)^{2}
\leq &~ 2\mathrm{dist}\left(0, -\nabla_{y}F^{\lambda}\left(x,y\right) + \cN_\cY(y)\right)^{2} + \frac{\lambda^{2}d_{h}L_{\varphi}^{2} \ell_{h}^{4}}{2} ,\label{eq:dist_final}
\end{align}
{where}
 \[\delta = \frac{\rho_{\scriptscriptstyle \lambda} \; D_{\cX}}{r}\left\|\nabla_{z}d_{r}^{\lambda}(y,x)\right\| + \frac{\ell_{\varphi}\ell_{h}\ell_{c}\sqrt{d_{h}}}{r}\left\|\nabla_{z}d_{r}^{\lambda}(y,x)\right\|+\left(\frac{\rho_{\scriptscriptstyle \lambda}}{2r^{2}} + \frac{1}{r}\right)\left\|\nabla_{z}d_{r}^{\lambda}(y,x)\right\|^{2} + \lambda \ell_{\varphi}\ell_{h}^{2}\sqrt{d_{h}}. \]
\label{lemma:for_application}
\end{lemma}
\begin{proof}
The optimality condition for $x_{r}^{\lambda}(y,x)$ gives 
$0 \in \nabla_{x} F^{\lambda}(x_{r}^{\lambda}(y,x),y) - r\left(x - x_{r}^{\lambda}(y,x)\right) + \cN_\cX(x_{r}^{\lambda}(y,x))$. In addition, by Danskin's Theorem, we have 
\begin{align} 
\nabla_{z}d_{r}^{\lambda}(y,x) = r(x - x_{r}^{\lambda}(y,x)).\label{dz_smooth_lambda}\end{align}
Hence,
\begin{align}\nabla_{z}d_{r}^{\lambda}(y,x) \in \nabla_{x} F^{\lambda}(x_{r}^{\lambda}(y,x),y) + \cN_\cX(x_{r}^{\lambda}(y,x)).\label{dz_epsilon_subdiff}
 \end{align}
 
Now, from the $\rho_{\scriptscriptstyle \lambda}$-weak convexity of $F^{\lambda}(\cdot,y)$ and \eqref{dz_epsilon_subdiff}, we have for any $\hat{x} \in \mathcal{X}$ that
\begin{align*}
&F^{\lambda}(\hat{x},y) \geq F^{\lambda}(x_{r}^{\lambda}(y,x),y) + \langle \nabla_{z}d_{r}^{\lambda}(y,x), \hat{x} - x_{r}^{\lambda}(y,x) \rangle - \frac{\rho_{\scriptscriptstyle \lambda}}{2}\left\|\hat{x} - x_{r}^{\lambda}(y,x)\right\|^{2}
\end{align*}
{which can be written as}
\begin{align*}
  F(\hat{x},y) + \big(F^{\lambda}(\hat{x},y) - F(\hat{x},y)\big) &~\geq F(x,y) + \big(F(x_{r}^{\lambda}(y,x),y) - F(x,y)\big) +\big(F^{\lambda}(x_{r}^{\lambda}(y,x),y) - F(x_{r}^{\lambda}(y,x),y)\big)  \\
\quad & + \langle \nabla_{z}d_{r}^{\lambda}(y,x), \hat{x} - x \rangle + \langle \nabla_{z}d_{r}^{\lambda}(y,x), x - x_{r}^{\lambda}(y,x) \rangle - \frac{\rho_{\scriptscriptstyle \lambda}}{2}\left\|\hat{x} - x\right\|^{2} \nonumber\\
\quad &  - \rho_{\scriptscriptstyle \lambda}\left\langle \hat{x} - x, x - x_{r}^{\lambda}(y,x) \right\rangle - \frac{\rho_{\scriptscriptstyle \lambda}}{2}\left\|x - x_{r}^{\lambda}(y,x)\right\|^{2}.  \end{align*}
By \eqref{f_minus_f_lambda}, \eqref{f_minus_f_lambda2}, and Cauchy-Schwarz inequality, we obtain from the above inequality that
\begin{align*}
&
F(\hat{x},y) + \frac{\lambda \ell_{\varphi}\ell_{h}^{2}\sqrt{d_{h}}}{2}  \geq F(x,y) -\ell_{\varphi}\ell_{h}\ell_{c}\sqrt{d_{h}}\left\|x - x_{r}^{\lambda}(y,x)\right\| -\frac{\lambda \ell_{\varphi}\ell_{h}^{2}\sqrt{d_{h}}}{2}  + \langle \nabla_{z}d_{r}^{\lambda}(y,x), \hat{x} - x \rangle \nonumber  \\
\quad &  + \langle \nabla_{z}d_{r}^{\lambda}(y,x), x - x_{r}^{\lambda}(y,x) \rangle - \frac{\rho_{\scriptscriptstyle \lambda}}{2}\left\|\hat{x} - x\right\|^{2} - \rho_{\scriptscriptstyle \lambda}\left\| \hat{x} - x\right\|\left\| x - x_{r}^{\lambda}(y,x) \right\| - \frac{\rho_{\scriptscriptstyle \lambda}}{2}\left\|x - x_{r}^{\lambda}(y,x)\right\|^{2},
\end{align*}
which together with \eqref{dz_smooth_lambda} and $\left\|\hat{x} - x\right\| \leq D_\cX$ implies
\begin{align*}
&
F(\hat{x},y)  \geq F(x,y) + \langle \nabla_{z}d_{r}^{\lambda}(y,x), \hat{x} - x \rangle  - \frac{\rho_{\scriptscriptstyle \lambda}}{2}\left\|\hat{x} - x\right\|^{2} \\
\quad & - \left[\frac{\rho_{\scriptscriptstyle \lambda} \; D_{\cX}}{r}\left\|\nabla_{z}d_{r}^{\lambda}(y,x)\right\| + \frac{\ell_{\varphi}\ell_{h}\ell_{c}\sqrt{d_{h}}}{r}\left\|\nabla_{z}d_{r}^{\lambda}(y,x)\right\|+\left(\frac{\rho_{\scriptscriptstyle \lambda}}{2r^{2}} + \frac{1}{r}\right)\left\|\nabla_{z}d_{r}^{\lambda}(y,x)\right\|^{2} + \lambda \ell_{\varphi}\ell_{h}^{2}\sqrt{d_{h}}\right].\end{align*}
Hence, from the definition of $\delta$, it follows
\begin{align*}
F(\hat{x},y)  \geq F(x,y) + \langle \nabla_{z}d_{r}^{\lambda}(y,x), \hat{x} - x \rangle  - \frac{\rho_{\scriptscriptstyle \lambda}}{2}\left\|\hat{x} - x\right\|^{2} - \delta .
\end{align*}
Thus, we deduce from the definition of $\delta$-subdifferential that 
\[\nabla_{z}d_{r}^{\lambda}(y,x) \in \partial_{x}^{\delta} \big(F(x,y) + \iota_\cX(x)\big) \implies \mathrm{dist}\left(0,  \partial_{x}^{\delta} \big(F(x,y) + \iota_\cX(x)\big)\right)^{2}
\leq  \left\|\nabla_{z}d_{r}^{\lambda}(y,x)\right\|^{2}.\]

Finally, by Young's inequality and \eqref{nabla h nonsmooth y}, we have
\begin{align}
\mathrm{dist}\left(0, -\nabla_{y}F\left(x,y\right) + \cN_\cY(y)\right)^{2}
\leq &~ 2\mathrm{dist}\left(0, -\nabla_{y}F^{\lambda}\left(x,y\right) + \cN_\cY(y)\right)^{2} +  2\left\|\nabla_{y}F^{\lambda}\left(x,y\right) - \nabla_{y}F\left(x,y\right) \right\|^{2}\nonumber\\
\leq &~  2\mathrm{dist}\left(0, -\nabla_{y}F^{\lambda}\left(x,y\right) + \cN_\cY(y)\right)^{2} + \frac{\lambda^{2}d_{h}L_{\varphi}^{2} \ell_{h}^{4}}{2}.\nonumber
\end{align}
This completes the proof.
\end{proof}
{We are now ready to present the main convergence result for the nonsmooth case.}
\begin{proof}[Proof of Theorem \ref{last:theorem}] 
We first note that from the definitions of $F^{\lambda}_{r}(x,y,z)$, $d_{r}(y,z)$ and $p_{r}(z)$ in Table \ref{tab:mylabel} we have for all $x \in \cX$, $y \in \cY$ and $z \in \mathbb{R}^{d_{x}}$,
{\begin{align}
\hat{\Phi}_r(x, y, z) =&~ \left(F^{\lambda}_{r}(x,y,z) - d^{\lambda}_{r}(y,z)\right) + \left(p^{\lambda}_{r}(z) - d^{\lambda}_{r}(y,z)\right) + p^{\lambda}_{r}(z)\nonumber\\
\geq&~p^{\lambda}_{r}(z)\nonumber\\
{=}&~\max_{y \in \mathcal{Y}}\min_{x \in \mathcal{X}}F^{\lambda}_{r}(x,y,z)\nonumber\\
\ge&~\max_{y \in \mathcal{Y}}\min_{x \in \mathcal{X}}F^{\lambda}(x,y)\nonumber\\
\ge&~ {\max_{y \in \mathcal{Y}}\min_{x \in \mathcal{X}}F(x,y) - \frac{\lambda L_{\varphi}L_{h}^{2}\sqrt{d_{h}}}{2}} \nonumber\\
{\geq}&~ {\underline{F} - \frac{\lambda L_{\varphi}L_{h}^{2}\sqrt{d_{h}}}{2} =:\underline{F}^{\lambda}},
\label{intermediate_1_6th}
\end{align}
where 
{the second inequality follows from  
$F_{r}^{\lambda}(x,y,z)\ge F^{\lambda}(x,y)$, the third one is by \eqref{f_minus_f_lambda},} and 
the last inequality holds from Remark \ref{remark-f_lambda_compact}.} 

Now, in light of Lemma \ref{lemma:for_application}, we have
\begin{align}
&\mathbb{E}\left[\mathrm{dist}\left(0,  \partial_{x}^{\delta} \big(F(x,y) + \iota_\cX(x)\big)\right)^{2}\right]+ \mathbb{E}\left[\mathrm{dist}\left(0, -\nabla_{y}F\left(\tilde{x},\tilde{y}\right) + \cN_\cY(\tilde{y})\right)^{2} \right]\nonumber\\
\leq &~ \underbrace{\mathbb{E}\left\|\nabla_{z}d_{r}^{\lambda}\left(\tilde{y}, \tilde{x}\right)\right\|^{2}}_{\textcircled{1}}+ \underbrace{2\mathbb{E}\left[\mathrm{dist}\left(0, -\nabla_{y}F^{\lambda}\left(\tilde x, \tilde y\right) + \cN_\cY(\tilde y)\right)^{2} \right]}_{\textcircled{2}} + \underbrace{\frac{\lambda^{2}d_{h}L_{\varphi}^{2} \ell_{h}^{4}}{2}}_{\textcircled{3}}\label{final:1st},
\end{align} 
where
\begin{align}
\delta =  \frac{\rho_{\scriptscriptstyle \lambda} \; D_{\cX}}{r}\left\|\nabla_{z}d_{r}^{\lambda}(\tilde{y}, \tilde{x})\right\| + \frac{\ell_{\varphi}\ell_{h}\ell_{c}\sqrt{d_{h}}}{r}\left\|\nabla_{z}d_{r}^{\lambda}(\tilde{y}, \tilde{x})\right\|+\left(\frac{\rho_{\scriptscriptstyle \lambda}}{2r^{2}} + \frac{1}{r}\right)\left\|\nabla_{z}d_{r}^{\lambda}(\tilde{y}, \tilde{x})\right\|^{2} + \lambda \ell_{\varphi}\ell_{h}^{2}\sqrt{d_{h}} .\label{final_delta_final_1st}
\end{align}
Directly from Theorem \ref{main:theorem2}, it follows that the algorithm produces $(\tilde{x}, \tilde{y})$ such that 
\(
\textcircled{1}, \textcircled{2} = \cO\left(\varepsilon^{2}\right).
\)
Additionally, since $\lambda = \Theta\left(\varepsilon\right)$, \textcircled{3} in \eqref{final:1st} is also $\mathcal{O}\left(\varepsilon^{2}\right)$. 

The argument for the sample complexity is the same as the one presented in Theorem \ref{main:theorem2} and thus, the proof is complete.
\end{proof}

\section{Motivating Applications - Formulations}\label{application sections}
{In this section, we present a detailed description of two classes of problems whose cost function can be formulated in accordance with \eqref{eq:composite-f}. 
\begin{itemize}
\item[i. ] \textbf{$\phi$-Divergence DRO}

As discussed earlier, the objective of distributionally robust modeling is to ensure that the learned models remain reliable not only under adversarial perturbations but also in the presence of distributional shifts. In contrast to traditional empirical risk minimization, which optimizes performance under the empirical data distribution, distributionally robust optimization seeks to minimize the worst-case loss over an uncertainty set of probability distributions. 

\citet{levy2020large} study this DRO problem and consider a slight generalization of $\phi$-divergence--based DRO. Specifically, let $\phi : \mathbb{R}_+ \to \mathbb{R} \cup \{+\infty\}$ be a convex function satisfying $\phi(1) = 0$, and define the $\phi$-divergence between distributions $\mathbb{Q}$ and $\mathbb{P}$ as follows:
\begin{align}\label{phi_divergence_dro_problem}
D_{\phi}(\mathbb{Q}, \mathbb{P}) := \int \phi\left(\frac{d\mathbb{Q}}{d\mathbb{P}}\right) d\mathbb{P},
\end{align}
where $\mathbb{Q}$ is absolutely continuous with respect to $\mathbb{P}$. Then, for a convex function $\psi$ with $\psi(1) = 0$, a constraint radius $\rho \ge 0$, penalty parameter $\lambda \ge 0$, and $\ell(\theta;\cdot)$ denoting the loss function parameterized by $\theta$, the objective considered in \cite{levy2020large}
takes the form:
\begin{align}\label{dro_orig}
 \min_{\theta \in \Theta} \max_{\mathbb{Q} : D_{\phi}(\mathbb{Q}, \mathbb{P}) \le \rho}
\left\{ \mathbb{E}_{\bm{\xi} \sim \mathbb{Q}}[\ell(\theta; \bm{\xi})] - \lambda D_{\psi}(\mathbb{Q}, \mathbb{P}) \right\}.
\end{align}

We consider the case when $\mathbb{P}$ is the empirical distribution induced by $N$ i.i.d samples with support $\Xi = \left\{\bm{\xi}_{i},\ldots, \bm{\xi}_{N}\right\}$. Since $\mathbb{Q}$ is absolutely continuous wrt $\mathbb{P}$, it follows that $\mathbb{Q}$  assigns probability mass only to points in $\Xi$. By the Radon--Nikodym theorem, for each $\bm{\xi}_i \in \Xi$, we then have:
\begin{align}\label{radon-nikodym} 
\mathbb{Q}(\bm{\xi}_i) = \frac{d\mathbb{Q}}{d\mathbb{P}}(\bm{\xi}_i)\, \mathbb{P}(\bm{\xi}_i) \implies \frac{d\mathbb{Q}}{d\mathbb{P}}(\bm{\xi}_i) = \frac{\mathbb{Q}(\bm{\xi}_i)}{\mathbb{P}(\bm{\xi}_i)}.
\end{align}
Without loss of generality, denote $\mathbb{Q}(\bm{\xi}_{i}) = q_{i}$ with $q \in \Delta_{N}$, the $N$-probability simplex. Then, 
utilizing  \eqref{phi_divergence_dro_problem} and \eqref{radon-nikodym} yields 
\[D_{\phi}(\mathbb{Q}, \mathbb{P}) =\sum_{i=1}^{N} \phi\left(\frac{d\mathbb{Q}}{d\mathbb{P}}(\bm{\xi}_i)\right) \mathbb{P}(\bm{\xi}_i) = \sum_{i=1}^{N} \phi\left(\frac{\mathbb{Q}(\bm{\xi}_{i})}{\mathbb{P}(\bm{\xi}_{i})}\right) \mathbb{P}(\bm{\xi}_{i}) = \sum_{i=1}^{N} \phi\left(\frac{q_i}{\frac{1}{N}}\right) \frac{1}{N} = \frac{1}{N}\sum_{i=1}^{N}\phi(Nq_{i}). \]
Similarly, we get
\[D_{\psi}(\mathbb{Q}, \mathbb{P}) = \frac{1}{N}\sum_{i=1}^{N}\psi(Nq_{i}). \]
Now, assuming each data sample $\bm{\xi}_i$ is a pair $(x_{i},y_{i})$ of features and targets, denoting the predictor function parameterized by $\theta$ as $f_{\theta}$ and letting $e_{\bm{\xi}_{i}}$ denote the $i^{th}$ standard basis vector in $\mathbb{R}^{N}$, the DRO problem in \eqref{dro_orig} can be expressed as
\begin{align}
\min_{\theta \in \Theta} \max_{q \in \Delta_{N}:\sum_{i=1}^{N} \phi\left(Nq_i\right)\leq N\rho} \;
\frac{1}{N} \sum_{i=1}^N 
Nq_{i}\ell\left(f_{\theta}(x_{i}),y_{i}\right) - \lambda\psi (Nq_{i}).\label{final-eq:app}
\end{align}
Under the choices of
\begin{equation*}
\begin{aligned}
c\left(\theta; \bm{\xi}_{i}\right) = \left(f_{\theta}(x_{i}), y_{i}\right),\;
h(v_{1}, v_{2}) = \ell(v_{1}, v_{2}),\;\text{and }
\varphi\left(w_{1}, w_{2}; \bm{\xi}_{i}\right) = N\left(e_{\bm{\xi}_{i}}^{\top} w_{2}\right) w_{1}
   - \lambda\psi\left(
      N\left(e_{\bm{\xi}_{i}}^{\top} w_{2}\right)
     \right),
\end{aligned}
\end{equation*}
it then follows that the cost function in \eqref{final-eq:app} exhibits the structure in \eqref{eq:composite-f}.

 \item[ii.] \textbf{Group--DRO} 

In many real-world applications, data points naturally belong to disjoint groups or subpopulations, and models trained via standard empirical risk minimization often underperform on minority or underrepresented groups. Group distributionally robust optimization (Group--DRO) addresses this by explicitly optimizing the worst-case performance across groups, rather than the average loss. This approach is particularly useful in domains where performance differences across groups can have undesired consequences; in healthcare where it can help address subgroup performance disparities \cite{Pfohl2022}, in recommendation systems to improve worst-case user experience \cite{wen2022distributionallyrobust}, and in Natural Language Processing (NLP) to train models that are robust to topic shifts \cite{OrenSagawaHashimotoLiang2019}.

Formally, let the dataset be partitioned into $M$ groups with distributions $\{\mathbb{P}_g\}_{g=1}^M$, and let $f_\theta$ denote a predictor function parameterized by $\theta$. Following \cite{SagawaKohHashimotoLiang2020}, we have the Group--DRO problem objective as follows:
\begin{equation*}
\begin{aligned}
\min_{\theta \in \Theta} \max_{q \in \Delta_M} \; \sum_{g=1}^M q_g \; \mathbb{E}_{(x,y) \sim \mathbb{P}_g} \left[ \ell(f_{\theta}(x), y) \right],
\end{aligned}
\end{equation*}
where $\Delta_M$ is the $M$-dimensional probability simplex, and $q$ assigns weights to the groups. We note that:
\begin{equation*}
\begin{aligned}
\sum_{g=1}^M q_g 
\mathbb{E}_{(x,y) \sim \mathbb{P}_g}
\left[ \ell\left(f_{\theta}(x), y\right)\right]= \sum_{g=1}^M q_g \;
\frac{1}{\lvert \mathcal{G}_{g} \rvert}
\sum_{i=1}^{\lvert \mathcal{G}_{g} \rvert}
\ell\left(f_{\theta}(x_{i}), y_{i}\right) \overset{(i)}{=}
\frac{1}{N}
\sum_{i=1}^{N}
\frac{N q_{g_{i}}}{\lvert \mathcal{G}_{g_{i}} \rvert}
\ell\left(f_{\theta}(x_{i}), y_{i}\right),
\end{aligned}
\end{equation*}
where $(i)$ follows from the fact that the groups are disjoint and
collectively cover all data points, and $g_i$ denotes the index of the
group to which the data point $(x_i, y_i)$ belongs. Let $e_{g_i}$ denote the $g_i^{\text{th}}$ standard basis vector in
$\mathbb{R}^{M}$. Then under the choices of
\begin{equation*}
\begin{aligned}
\mathcal{G}
= \left\{ \lvert \mathcal{G}_{j} \rvert \right\}_{j=1}^{M}, \;
\bm{\xi}_{i} = (x_{i}, y_{i}), \;
c\left(\theta; \bm{\xi}_{i}\right) = \left(f_{\theta}(x_{i}), y_{i}\right), \;
h(v_{1}, v_{2}) = \ell(v_{1}, v_{2}),\;
\varphi\left(w_{1}, w_{2}; \bm{\xi}_{i}\right)
= \frac{
N \left(e_{g_{i}}^{\top} w_{2}\right) w_{1}
}{
e_{g_{i}}^{\top} \mathcal{G}
},
\end{aligned}
\end{equation*}
it then follows that the cost function exhibits the structure in \eqref{eq:composite-f}.
\end{itemize}}

{\section{K\L{} example}\label{final-section-kl}
In this section we present an explicit example of a smooth but nonconcave function that satisfies the K\L{} property.
Consider the function $g : [-2,2] \to \mathbb{R}$ defined as
\begin{equation}\label{eq:example-KL}
g(y) =
\begin{cases}
2e^{y+1}-1, & -2 \leq y \leq -1, \\[6pt]
- y^2 + 2, & -1 < y \leq 1, \\[6pt]
2e^{-y+1}-1, & 1 < y \leq 2.
\end{cases}
\end{equation}
As can be observed, the function $g$ is continuously differentiable on the compact set $[-2,2]$. Moreover, despite being nonconcave, it attains its unique global maximum at $y^\star = 0$.

Figure~\ref{fig:Nonconcave function} illustrates the shape of the function.
\begin{figure}[H]
    \centering
    \includegraphics[width=0.7\textwidth]{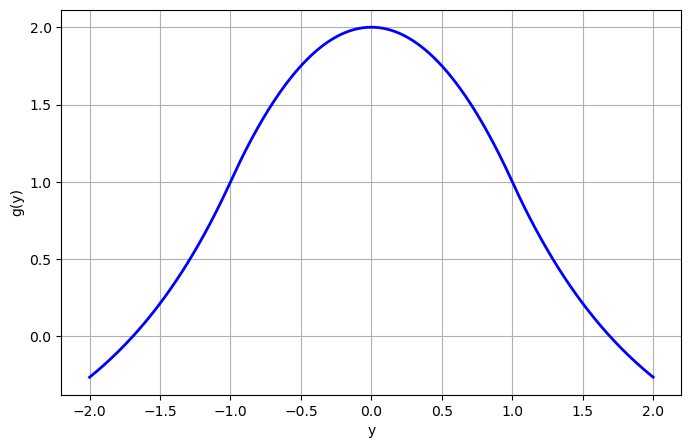}
    \caption{Illustration of the nonconcave function defined in \eqref{eq:example-KL}.}
    \label{fig:Nonconcave function}
\end{figure}

\begin{proposition}
    The function $g(y)$ defined in \eqref{eq:example-KL} satisfies the K\L{} property with $\theta = \frac{1}{2}$ and $\mu = \frac{1}{10}$, i.e., for all $y \in [-2,2]$, 
    \begin{equation}\label{eq:kl-prop-example}\mathrm{dist}\left(0, -g'(y) + \cN_{[-2,2]}(y)\right) \geq \mu\left(\max_{y^{'} \in [-2,2]}g(y') - g(y)\right)^{\theta}.\end{equation}
\end{proposition}
\begin{proof}
First, we note that  $\max_{y' \in [-2,2]} g(y') = 2$. Additionally, we have
\[\cN_{[-2,2]}(y) = \begin{cases}
(-\infty,\,0], & y = -2, \\[6pt]
0, & y \in (-2,2), \\[6pt]
[0,\,\infty), & y = 2,\end{cases}  \qquad \text{and} \qquad g'(y) = \begin{cases}
2e^{y+1}, & -2 \leq y \leq -1, \\[6pt]
- 2y, & -1 < y \leq 1, \\[6pt]
-2e^{-(y-1)}, & 1 < y \leq 2.\end{cases}\]

Since $g'(-2)>0$, we have
\[\mathrm{dist}\left(0, -g'(-2) + \cN_{[-2,2]}(-2)\right) = \mathrm{dist}\Big(0, \left(-\infty,- g'(-2)\right] \Big) = \lvert g'(-2) \rvert = 2e^{-1}\geq \frac{1}{10}\left(3 - 2e^{-1}\right)^{\frac{1}{2}}.\]
Thus \eqref{eq:kl-prop-example} is satisfied for $y = -2$. 
Similarly, since $g'(2) < 0$, we have
\[\mathrm{dist}\left(0, -g'(2) + \cN_{[-2,2]}(2)\right) = \mathrm{dist}\Big(0, \left[-g'(2), \infty\right) \Big) = \lvert g'(2) \rvert = 2e^{-1}\geq \frac{1}{10}\left(3 - 2e^{-1}\right)^{\frac{1}{2}}.\]
Hence, \eqref{eq:kl-prop-example} is also satisfied at $y = 2$.

We now discuss the scenario when $y \in (-2,2)$.  Here, we have
$\mathrm{dist}\left(0, -g'(y) + \cN_{[-2,2]}(y)\right) = \lvert g'(y)\rvert.$
Thus, it suffices to verify
       \begin{align}
        \lvert g'(y)\rvert \geq  \frac{1}{10}\left(2 - g(y)\right)^{\frac{1}{2}},\; \forall y \in \left(-2,2\right) .\label{kl_example_smooth}\end{align}

For the interval $y \in (-2,-1]$, we have $g'(y) = 2e^{y+1}$ and $2 - g(y) = 3 - 2e^{y+1}$. Then
\[\inf_{y \in (-2,-1]} \, \lvert g'(y)\rvert = 2e^{-1} \approx 0.37, \qquad
\sup_{y \in (-2,-1]} \frac{1}{10} \left(3 - 2e^{y+1}\right)^{\frac{1}{2}} = \frac{1}{10} \left(3 - 2e^{-1}\right)^{\frac{1}{2}} \approx 0.15.\]
Thus, \eqref{kl_example_smooth} is satisfied on this interval.

For the interval $y \in (-1,1]$, we have $g'(y) = -2y$ and $2 - g(y) = y^2 = \lvert y \rvert^2$. Then \eqref{kl_example_smooth} follows from $2 \lvert y\rvert \geq \frac{1}{10}\lvert y\rvert$. 

For the interval $y \in (1,2)$, we have $g'(y) = -2 e^{-(y-1)}$ and $2 - g(y) = 3 - 2 e^{-(y-1)}$. Then
\[\inf_{y \in (1,2)} \, \lvert g'(y)\rvert = 2e^{-1} \approx 0.37, \qquad
\sup_{y \in (1,2)} \frac{1}{10} \left(3 - 2e^{-(y-1)}\right)^{\frac{1}{2}} = \frac{1}{10} \left(3 - 2e^{-1}\right)^{\frac{1}{2}} \approx 0.15.\] Hence, \eqref{kl_example_smooth} is satisfied on this interval as well. Thus, the proof is complete.
\end{proof}

}

\end{document}